\documentclass[psamsfonts]{amsart}

\usepackage{amssymb,amsfonts,amscd}
\usepackage[color=Aquamarine,textsize=footnotesize,obeyFinal]{todonotes}

\usepackage[colorlinks,linktocpage]{hyperref}
\hypersetup{linkcolor=[rgb]{0,0,0.715}}
\hypersetup{citecolor=[rgb]{0,0.715,0}}

\newtheorem{thm}{Theorem}[section]
\newtheorem{cor}[thm]{Corollary}
\newtheorem{prop}[thm]{Proposition}
\newtheorem{lem}[thm]{Lemma}
\newtheorem{conj}[thm]{Conjecture}
\newtheorem{quest}[thm]{Question}

\theoremstyle{definition}
\newtheorem{defn}[thm]{Definition}

\theoremstyle{remark}
\newtheorem{rem}[thm]{Remark}

\DeclareMathOperator{\RCD}{RCD}

\newcommand{\dist}{\mathsf{d}}

\newcommand{\meas}{\mathfrak{m}}

\newcommand{\heat}{\mathsf{h}}
\newcommand{\Heat}{\mathsf{H}}
\newcommand{\HeatK}{\mathsf{K}}
\newcommand{\HeatQ}{\mathsf{Q}}

\numberwithin{equation}{section}

\title[]{Singular Weyl's Law with Ricci curvature bounded below} %(Loi de Weyl singulier avec courbure de Ricci d\'elimit\'ee en dessous)}

\author[]{Xianzhe Dai}
\address[Xianzhe Dai]{Department of Mathematics, University of California,  Santa Barbara, CA, USA.}
\email{dai@math.ucsb.edu}
\thanks{X. Dai is partially supported by the Simons Foundation Grant 635115.}
	
\author[]{Shouhei Honda}
\address[Shouhei Honda]{Mathematical Institute, Tohoku University, Sendai, Japan.}
\email{shouhei.honda.e4@tohoku.ac.jp}
\thanks{S. Honda is partially supported by the Grant-in-Aid for Scientific Research (B) of 20H01799, the
	Grant-in-Aid for Scientific Research (B) of 21H00977 and Grant-in-Aid for Transformative
	Research Areas (A) of 22H05105.}

\author[]{Jiayin Pan}
\address[Jiayin Pan]{Department of Mathematics, University of California, Santa Cruz, CA, USA.}
\email{jpan53@ucsc.edu}

\author[]{Guofang Wei}
\address[Guofang Wei]{Department of Mathematics, University of California,  Santa Barbara, CA, USA.}
\email{wei@math.ucsb.edu}
\thanks{G. Wei is partially supported by NSF DMS 2104704. }
\begin{document}

\begin{abstract}
	We establish two surprising types of  Weyl's laws for some compact $\RCD(K, N)$/Ricci limit spaces. The first type could have power growth of any order (bigger than one). The other one has an order corrected by logarithm similar to some fractals even though the space is 2-dimensional. Moreover the limits in both types can be written in terms of the singular sets of null capacities, instead of the regular sets.  These are the first examples with such features for $\RCD(K,N)$ spaces. Our results depend crucially on analyzing and developing important properties of the examples constructed in \cite{PanWei}, showing them isometric to the $\alpha$-Grushin halfplanes. Of independent interest, this also allows us to provide  counterexamples to  conjectures in \cite{ChCo97, KKK}.
	
	%Version Francais
	%Nous établissons deux types surprenants de lois de Weyl pour des espaces compacts $\RCD(K, N)$/Ricci limites. Le premier type pourrait avoir une croissance de puissance de n'importe quel ordre (supérieur à un). L'autre a un ordre corrigé par logarithme similaire à certaines fractales même si 
	%la dimension de l'espace est 2. De plus, les limites dans les deux types peuvent être écrites en termes d'ensembles singuliers de capacités nulles, au lieu d'ensembles réguliers. Ce sont les premiers exemples avec de telles fonctionnalités pour les espaces $\RCD(K,N)$. Nos résultats dépendent essentiellement de l'analyse et du développement de propriétés importantes des exemples construits dans \cite{PanWei}, en montrant qu’elles sont isométriques aux demi-plans $\alpha$-Grushin. D'un intérêt indépendant, cela nous permet également de fournir des contre-exemples aux conjectures en \cite{ChCo97, KKK}.
\end{abstract}

\maketitle

\section{Introduction}

Weyl's law describes the asymptotic behavior of eigenvalues of the Laplace-Beltrami operator. Its study has a long rich history extended over a century, and plays an important role both in mathematics and physics. See for instance \cite{IV} for a survey.

For a closed Riemannian manifold $M^n=(M^n, g)$,  the minus Laplacian $-\Delta=-\Delta^g$ has  discrete unbounded spectrum as follows:
\begin{equation}  \label{eigenvalues}
	0=\lambda_0<\lambda_1\le \lambda_2\le \cdots \to \infty,
\end{equation} 
where $\lambda_i$ is counted with multiplicities.  It satisfies the following well-known Weyl's law
	\begin{equation}\label{eq:smooth}
	\lim_{\lambda \rightarrow \infty} \frac{N(\lambda)}{\lambda^{n/2}} = \frac{\omega_n}{(2\pi)^n} \mathrm{vol} (M^n), \end{equation}
	where $\mathrm{vol}=\mathrm{vol}^g$ denotes the Riemannian volume measure,
		\begin{equation}  \label{N-lambda}
		N(\lambda)=\sharp \left\{ i \in \mathbb{N}\ | \  \lambda_i \le \lambda \right\},
	\end{equation}
and $\omega_n$ is the volume of the $n$-dimensional Euclidean unit ball. 

In this paper we discuss the validity of similar Weyl's law for singular spaces with Ricci curvature bounded below, the so-called  Ricci limit spaces, and more generally we deal with $\RCD(K, N)$ spaces. 

Roughly speaking, an $\RCD(K, N)$ space $(X, \dist, \meas)$ is a metric measure space with Ricci curvature bounded below by $K$ and dimension bounded above by $N$ in a synthetic sense. Moreover the heat flow $\heat_t$ on $X$ is linear. The linearity of the heat flow allows us to define the heat kernel $\Heat$ on $X$ and the short time behavior of $\Heat$ should be related to the validity of expected Weyl's law on $X$ as in the case of closed Riemannian manifolds. Let us emphasize that the reference measure $\meas$ is not necessarily a Hausdorff measure. See \S \ref{pre} for the preliminaries on this topic. 

An $\RCD(K, N)$ space is called a Ricci limit space if it can be approximated by smooth Riemannian manifolds with Ricci curvature bounded below and dimension bounded above. See for instance \cite{ChCo96, ChCo97, ChCo00a, ChCo00b} for the structure theory on Ricci limit spaces. An important example of Ricci limit spaces is a closed weighted Riemannian manifold $(M^n, g, e^{-f}d\mathrm{vol})$ for some $f \in C^{\infty}(M^n)$ (see \cite{John}). In this case the corresponding (metric measure) Laplacian coincides with the Witten Laplacian $\Delta^g -g(\nabla f, \nabla \cdot)$ and (\ref{eq:smooth}) is still satisfied.

In order to state our results, here we recall the following. Let $(X, \dist, \meas)$ be an $\RCD(K, N)$ space for some $K \in \mathbb{R}$ and some finite $N \in [1, \infty)$. % Assume that $X$ is not a single point.
Then it is proved in \cite[Theorem 0.1]{BrueSemola} (after \cite{CoNa12} for Ricci limit spaces) that there exists a unique $n \in \mathbb{N} \cap [1, N]$ such that the $n$-dimensional regular set $\mathcal{R}_n$ has $\meas$-full measure. This $n$ is called the rectifiable dimension or the essential dimension of $(X, \dist, \meas)$. A point  $x \in X$ is called $n$-regular if 
\begin{equation}
	\left(X, \frac{1}{r}\dist, \frac{\meas}{\meas (B_r(x))},x\right)\stackrel{\mathrm{pmGH}}{\to} \left(\mathbb{R}^n, \dist_{\mathbb{R}^n}, \frac{\mathcal{H}^n}{\omega_n}, 0_n\right),\quad \text{as $r \to 0^+$.}
\end{equation}
%where $\omega_n$ is the $n$-dimensional volume of a unit ball in $\mathbb{R}^n$, $\omega_n=\mathcal{H}^n(B_1(0_n))$.
It is called singular if it is not $k$-regular for any integer $k$. 

For a compact $\RCD(K,N)$ space $(X, \dist, \meas)$,  the canonical inclusion $$\iota: H^{1,2}(X, \dist, \meas) \\ \hookrightarrow L^2(X, \meas)$$ is a compact operator. %(see for instance \cite[Theorem 5.1]{HK}). 
Hence the eigenvalues of the minus Laplacian still satisfies \eqref{eigenvalues}. Its Weyl's law has been studied in \cite{AmbrosioHondaTewodrose, ZZ}, see the earlier work \cite{Ding02} for Ricci limit spaces. In these works, in particular in \cite{AmbrosioHondaTewodrose}, the asymptotic is related to the Hausdorff (not reference) measure of the reduced regular set $\mathcal{R}_n^*$ of the rectifiable dimension $n$ under the assumption that
\begin{equation}\label{eq}
\lim_{r\to 0^+}\int_X\frac{r^n}{\meas (B_r(x))}d\meas=\int_X\lim_{r\to 0^+}\frac{r^n}{\meas(B_r(x))}d\meas<\infty.
\end{equation}
Moreover (\ref{eq}) is equivalent to the validity of such Weyl's law:
\begin{equation}\label{eq:weylreg}
\lim_{\lambda \rightarrow \infty} \frac{N(\lambda)}{\lambda^{n/2}} = \frac{\omega_n}{(2\pi)^n} \mathcal{H}^n(\mathcal{R}_n^*)<\infty.
\end{equation}
See Theorem~\ref{R-weyl} for the precise statement. In this paper we call (\ref{eq:weylreg}) \textit{regular} Weyl's laws. It seems to us that experts may believe the validity of (\ref{eq}) for any compact $\RCD(K, N)$ space.
It is worth mentioning that if $(X, \dist, \meas)$ is non-collapsed, then (\ref{eq}) is easily satisfied, thus (\ref{eq:weylreg}) holds. See Remark \ref{remnoncollapsed}.

In this paper we show Weyl's law relating to the \textit{singular} set, giving the first example with this feature.  This is an unexpected result,  since  in particular  (\ref{eq})  is not satisfied.
\begin{thm}[Corollary \ref{cor:weyl sing}]\label{thm1}
For any  $\beta \in (2, \infty)$,	there is a compact $\RCD(-1,N)$ space $(X, \dist, \meas)$ for %some $K \in \mathbb{R}$ and 
 $N$ big depending on $\beta$, whose rectifiable dimension is $2$,
% with $\beta$-Hausdorff measure $0<\mathcal{H}^\beta (X) < \infty$,   
such that the following limit exists %and
	\[
0<	\lim_{\lambda \rightarrow \infty} \frac{N(\lambda)}{\lambda^{\beta/2}} < \infty. %= C(\beta, N)\, \mathcal{H}^\beta (\mathcal{S}) =  C(\beta, N)\, \mathcal{H}^\beta (X),  
	\]
	Moreover the limit coincides with the Hausdorff measure of the singular set $\mathcal{S}$ of $X$ up to multiplication by a ``canonical'' constant.
\end{thm}
Note that $\beta$ may not be an integer, unlike all earlier results, and that the canonical constant stated above is determined by the short time behavior of the heat kernel $\Heat(x,x,s^2)$ on the diagonal near the singular set. See also Theorem \ref{thm:weyl}. Moreover, it is worth mentioning that the regular set $\mathcal{R}_2$ of the example above is a smooth weighted $2$-dimensional (incomplete) Riemannian manifold and that the singular set $\mathcal{S}=X\setminus \mathcal{R}_2$ has null $2$-capacity with respect to $\meas$, in particular the eigenvalues on $X$ coincides with that of the Dirichlet Laplacian on $\mathcal{R}_2$. See Proposition \ref{propcap}.

In another direction we also show the surprising case that the power law $N(\lambda) \sim \lambda^\beta$ is not satisfied for any $\beta \ge 0$ as $\lambda \rightarrow \infty$. Namely,
\begin{thm}[Theorem \ref{2weyl}]\label{thmsingw}
	There is a compact $\RCD(-1,10)$ space $(X, \dist, \meas)$  with both the Hausdorff and rectifiable dimensions $2$ such that  
%	\begin{equation*}
%		\limsup_{\lambda \to \infty}\frac{N(\lambda)}{\lambda^{\beta}}=\infty,\quad \text{for any $\beta \le 1$}
%	\end{equation*}
%	and 
%	\begin{equation*}
%		\lim_{\lambda \to \infty}\frac{N(\lambda)}{\lambda^{\beta}}=0,\quad \text{for any $\beta > 1$}.
%	\end{equation*}
%	In fact, %there is slowly varying function $L$ such that 
	\begin{equation*}
	\lim_{\lambda \to \infty}\frac{N(\lambda)}{\lambda \log \lambda} = \frac{1}{4 \pi}. 
	\end{equation*}
%Moreover $C_1 \log (\lambda) \le L(\lambda) \le C_2 \log (\lambda)$. 	
\end{thm}
Note that the metric structure of any examples above is not an Alexandrov spaces though they are topologically $2$-dimensional. This fact should be compared with \cite{KL, LS}. Moreover the value $\frac{1}{4\pi}$ also coincides with the $2$-dimensional Hausdorff measure of the singular set $\mathcal{S}$ up to multiplication by a ``canonical'' constant; see Theorem \ref{2weyl}. It is worth mentioning that any neighborhood of $\mathcal{S}$ has infinite $2$-dimensional Hausdorff measure, a very different situation from Theorem \ref{thm1} where the space has a finite $\beta$-dimensional Hausdorff measure; see Remark \ref{rem:infinite}. In particular this also provides a counterexample to \cite[Conjecture 1.34]{ChCo97}; see Remark \ref{rem:counterexample}.
\begin{rem} There are many studies of Weyl's law for singular, or subRiemannian, or almost Riemannian incomplete  metrics with various measures, see e.g. \cite{Lian, Boscain-PS2016, CPR, dVHT}. 
%	 In \cite{CPR} Weyl's law with log term is shown for an incomplete unweighted Riemannian manifold $(\Omega^n, g)$ with the compact metric completion, see also \cite{Boscain-PS2016}. 
The techniques do not apply to our case though. In \cite{Boscain-PS2016, CPR}, the measure blows up at the singular set while our measure degenerates at the singular set. In \cite{dVHT} the measure is regular (i.e., neither blows up nor degenerates at the singularity).  Moreover in these three papers the spaces are not $\RCD$, and the power of the asymptotic growth is always rational (integer or half integer) while our power could be irrational.  With a change of variable as in Remark \ref{rem:change}, a weak form of our results  (i.e. only the rate of growth instead of precise asymptotics) can be derived from \cite[5.3]{Lian}.
But the  method in \cite{Lian} does not allow one to  obtain the precise asymptotics,  as examples in that paper show. Indeed, without curvature condition, precise asymptotics do not necessarily exist in that case (note that such a behavior is known in the fractal setting \cite[Corollary 7.2]{BH}). 
 Our focus is to find  the precise asymptotics as well as the exact leading order term. 
	 Moreover  we can write down the limit of $N(\lambda)$ as $\lambda \to \infty$ in terms of the singular set explicitly. The surprising points in the examples above are; the spaces have lower bounds of Ricci curvature in a synthetic sense and the limits can be written by sets of null-capacity (note that any set of null 2-capacity can be removable from the point of view of potential theory). Such behaviors seem to be unknown even in the fractal setting.
	 
In another direction, if one considers Schr\"odinger operators on complete noncompact manifolds, then Weyl's law could also involve log term, see \cite[Theorem 1.1 and Remark 1.2]{DaiYan22}. For related results on the small time heat kernel expansion of Witten Laplacian, see \cite[Theorem 1.1]{DaiYan20}. Such expansion implies a Weyl law by the Karamata-Tauberian type theorem \cite[Theorem 2.4]{DaiYan22}
\end{rem}
	
	Crucial to our construction we further develop  the examples constructed by the last two authors in \cite{PanWei}. These examples are the first Ricci limit spaces whose Hausdorff dimensions may not be integer and moreover,  the Hausdorff dimension of the singular set is bigger than that of the regular set, see Theorem~\ref{Pan-Wei example}. Thus the examples obtained above are actually Ricci limit spaces and we call the above asymptotic behaviors of eigenvalues \textit{singular} Weyl's laws.
	
Let us explain how to achieve the results above. By Karamata theorem \cite[Page 445, Theorem 2]{Feller}
%Tauberian theorem {\color{blue}for power orders} (for instance \cite{Simon});
\begin{equation}\label{taubTer}
\lim_{\lambda \to \infty}\frac{N(\lambda)}{\lambda^{\alpha/2} L(\lambda^{-1})}=C \in  [0, \infty) \Longleftrightarrow \lim_{t \to 0^+} \frac{\sqrt{t}^{\alpha}}{ L(t)} \sum_ie^{-\lambda_it}=\Gamma (\alpha+1)C \in [0, \infty) 
\end{equation}
for any slowly varying function $L(t): \mathbb R_+ \rightarrow \mathbb R_+$ as $t\rightarrow 0^+$ (i.e. $\lim_{t \rightarrow 0^+} \tfrac{L(a t)}{L(t)} = 1$ for any fixed $a>0$). We will apply this to $L$ a constant function for Theorem \ref{thm1} and $L$ the log function for Theorem \ref{thmsingw}.  
Hence 
the asymptotic of the eigenvalues is reduced to the short time behavior of the heat trace. To understand the heat kernel,  we first derive a complete and explicit description  of the metric and measure of the Pan-Wei examples $Y$ in \S 3.1, 3.2.   Surprisingly it turns out that $Y$ is isometric to the $2\alpha$-Grushin halfplane. The Grushin plane is a well known space in sub-Riemannian geometry \cite[Section 3.1]{Bel}; see \cite{Pan22} for further connection of Ricci limit space and sub-Riemannian geometry. In particular the space $Y$ has
%the distance on $Y$ is induced by an explicit incomplete Riemannian metric and furthermore, there are very interesting and
 important metric dilations as follows.
\begin{thm}
	 For each $\lambda>0$, there is a homeomorphism $F_\lambda: Y \rightarrow Y$ such that
	$$\dist (F_\lambda(y_1),F_\lambda(y_2))=\lambda\cdot \dist (y_1,y_2)$$
	for all $y_1,y_2\in Y$. 
\end{thm}
See Theorem~\ref{dist} for more detail.

For Riemannian manifolds, only the standard Euclidean space $\mathbb R^n$ has such dilation property. For Alexandrov spaces with metric dilations, they must be metric cones. The Pan-Wei examples are the first Ricci limit spaces with dilations that are not metric cones. Note that in \cite{Menguy} a (nonpolar) Ricci limit space is constructed as a tangent cone at infinity which is invariant under certain scales, but not all scales.   We also mention that the Carnot groups admit dilations, but these spaces are not RCD spaces; they even cannot be $\mathrm{CD} (K,N)$ for any $K \in \mathbb R$, $N \ge 1$ \cite{Juillet2021}. 
	
Even with the explicit metric and measure, computing the heat kernel is still a challenge. By Li-Yau's estimate \eqref{eq:gauusian}, the study of the asymptotic rate (although not the precise asymptotic) can be reduced to that of the measure of small balls instead of the heat kernel.  Still, the measure of balls cannot be computed explicitly; instead we give good effective estimates with the help of explicit description of the metric and measure. To get a grip on the small time asymptotic of the heat kernel,  we consider the ratios of the measure of balls to suitable rectangular shaped domains whose measures can be computed explicitly.  With sophisticated use of the dilation property we obtain formulas for integrals of the heat kernel in terms of certain functions, some explicit and some not, but with good control of asymptotics. These allow us to estimate and compute several limits playing important role for later study of  asymptotic behaviors;  see Theorem~\ref{int-heat-Y}.  These are carried out in \S 3.3-3.5. 
	
The space $Y$ above is a noncompact $\RCD(0,N)$ space. To construct the compact  $\RCD(-1,N)$ space $X$, first we study the heat kernel of quotient $\bar Y = Y/\mathbb Z$ by extending the relation of heat kernels between covering spaces for manifolds in \cite{L} to $\RCD$ space, see Proposition~\ref{prop:heat-cover}. For $\alpha = \frac 12$, the eigenvalues and eigenfunctions of $\bar Y$ with respect to the unweighted measure have been computed explicitly in \cite{Boscain-PS2016}.  Amazingly in this case the eigenvalues for weighted measure with $n =9$ turns out to be exactly the same as in the unweighted case, see Subsection 4.2. In this case $\bar Y$ is $\RCD(0,10)$ ($10$ is the smallest to have $\RCD$ $0$ lower bound. See Remark \ref{bakryemery}). 
	% involves doing  a quotient
Then we do a  perturbation of $\bar Y$, and a doubling to obtain $X$.   By keeping track of the heat kernels in each step, we  show that the small time asymptotic of the heat kernel of $X$ is essentially determined by that of  the heat kernel of $Y$, see Theorem~\ref{int-heat-X}.  

Applying the above while changing the geometric parameters appeared in the process, we finally establish the results above. %{\color{blue}Note that instead of (\ref{taubTer}), we need Tauverian theorem with a slowly varying function in order to establish Theorem \ref{thmsingw}.}

Naively, for a metric measure space, the short time behavior of the heat kernel depends more on the distance/metric as compared to the large time behavior,  where the measure character shows up more.  Therefore it is natural to make the following conjecture. 
\begin{conj}
	For a compact $\RCD(K,N)$ space $(X, \dist, \meas)$ with Hausdorff dimension $\beta$, the limit  $\lim_{t \rightarrow 0} t^{\beta/2} \int_X \Heat (x,x,t) d\meas (x)$ is the Hausdorff measure of $X$ up to multiplication by a ``canonical'' constant. 
\end{conj}
Note that the conjecture includes the case when $\mathcal{H}^{\beta}(X)=\infty$. Actually for the example $(X, \dist, \meas)$ in Theorem \ref{thmsingw}, the regular part of any open ball at any singular point has infinite $2$-dimensional Hausdorff measure; see Remark \ref{rem:infinite}. This property has an independent interest and of course, gives supporting evidence for the conjecture. We plan to show the valid of the conjecture for a large class of $\RCD$ spaces in a future paper.

Finally in connection with the validity of  Weyl's law, we also deal with the asymptotic behavior of eigenfunctions in \S \ref{regw}. In particular, after introducing a negative answer to a question by Ding in \cite{Ding02}, we provide related Weyl type asymptotics. See Theorem \ref{prop:ding general} and Proposition \ref{weyl eig2}.

\textit{Acknowledgments.} The first named author thanks Junrong Yan for useful discussions. The second named author wishes to thank Atsushi Katsuda and Tatsuya Tate for explaining him on the quantum ergodicity, in particular \cite{JZ}. Moreover he is very grateful to Naotaka Kajino and Takashi Kumagai for suggestions on Theorem \ref{thmsingw} which help improve the result. The third named author would like to thank Richard Montgomery for introducing him the Grushin planes.
Finally we wish to thank the referee for valuable suggestions for the revision.

%\begin{conj}
%	For a compact $\RCD(K,N)$ space $(X, \dist, \meas)$ with Hausdorff dimension $k$, then  $\lim_{t \rightarrow 0} t^{d/2} \int_X \Heat (x,x,t) d\meas (x) =0$ for all $d >k$. 
%\end{conj}

\section{Preliminaries: Heat Kernels on $\RCD(K, N)$ spaces}\label{pre}

All terminologies for singular spaces we will discuss in the sequel are prepared from the theory of metric measure spaces with Ricci curvature bounded below in a synthetic sense (namely $\RCD(K, N)$ spaces).
%In order to keep our presentation short, let us assume that the readers are familiar to the theory. 
%However let us
Here we give a quick introduction on the terminology and results that will be needed later. See for instance \cite{Ambrosio} for a nice survey.

A triple $(X, \dist, \meas)$ is said to be a metric measure space if $(X, \dist)$ is a complete separable metric space and $\meas$ is a Borel measure on $X$, where this is finite on each bounded subset and the support coincides with $X$. Roughly speaking a metric measure space $(X, \dist, \meas)$ is said to be an $\RCD(K, N)$ space for some $K \in \mathbb{R}$ and some $N \in [1, \infty]$ if the Ricci curvature is bounded below by $K$ and the dimension is bounded above by $N$ in a synthetic sense, moreover the $H^{1,2}$-Sobolev space $H^{1,2}(X, \dist, \meas)$ is a Hilbert space. 

The precise definition is as follows. Define the Cheeger energy $\mathrm{Ch}:L^2(X, \meas) \to [0, \infty)$ by
\begin{equation}\label{eq:1110}
\mathrm{Ch}(f):=\inf_{\{f_i\}_i}\left\{\liminf_{i\to \infty}\int_X(\mathrm{Lip}f_i)^2d\meas\right\},
\end{equation}
where the infimum of the right hand side above runs over all bounded Lipschitz functions $f_i \in L^2(X, \meas)$ $L^2$-converging to $f$ on $X$ and $\mathrm{Lip}f(x)$ denotes the local slope of $f$ at $x$ (see also \cite{Cheeger2}). Then the $H^{1,2}$-Sobolev space $H^{1,2}(X, \dist, \meas)$ is defined by the finiteness domain of $\mathrm{Ch}$. 
We denote by $|\nabla f| \in L^2(X, \meas)$ the minimal object as in the right hand side of (\ref{eq:1110}), so-called the minimal relaxed slope of $f$, then we know the following equality
\begin{equation}
\mathrm{Ch}(f)=\int_X|\nabla f|^2d\meas.
\end{equation}
Note that in general $H^{1,2}(X, \dist, \meas)$ equipped with the norm
\begin{equation}
\|f\|_{H^{1,2}}=\left(\|f\|_{L^2}^2+\mathrm{Ch}(f)\right)^{1/2}
\end{equation}
is a Banach space. We say that $(X, \dist, \meas)$ is infinitesimally Hilbertian (IH) if $H^{1,2}(X, \dist ,\meas)$ is Hilbert.
In the sequel let us assume that $(X, \dist, \meas)$ is IH.

Then for all $f_i \in H^{1,2}(X, \dist, \meas)(i=1,2)$, 
\begin{equation}
\langle \nabla f_1, \nabla f_2\rangle =\lim_{\epsilon \to 0}\frac{|\nabla (f_1+\epsilon f_2)|^2-|\nabla f_1|^2}{2\epsilon} \in L^1(X, \meas).
\end{equation}
is well-defined. The domain $D(\Delta)$ of the Laplacian $\Delta$ of $(X, \dist, \meas)$ is defined by the set of all $f \in H^{1,2}(X, \dist, \meas)$ such that there exists a unique $h \in L^2(X, \meas)$, denoted by $\Delta f$, with
\begin{equation}
\int_X\langle \nabla f, \nabla \phi \rangle d\meas=-\int_Xh\phi d\meas,\quad \text{for any $\phi \in H^{1,2}(X, \dist, \meas)$.}
\end{equation} 
We are ready to introduce the heat flow $\heat_t:L^2(X, \meas) \to L^2(X, \meas)$. It is defined by a continuous curve $t \mapsto \heat_tf \in L^2(X, \meas)$ on $[0, \infty)$, which is locally absolutely continuous on $(0, \infty)$ (more strongly smooth, namely in $C^{\infty}((0, \infty), H^{1,2}(X, \dist, \meas))$, see \cite[Proposition 5.2.12]{GP}), with $\heat_tf \in D(\Delta)$ for $t>0$, $\heat_0f=f$ and 
\begin{equation}\label{eq:heat flow}
\frac{d}{dt} \heat_tf=\Delta \heat_tf,\quad \text{for $\mathcal{L}^1$-a.e. (or equivalently for any) $t>0$}.
\end{equation}
This will play the core role to define a main target of the paper, so-called the heat kernel below. 

We are now in a position to introduce the $\RCD(K, N)$ condition. An IH metric measure space $(X, \dist, \meas)$ is said to be an $\RCD(K, N)$ space for some $K \in \mathbb{R}$ and some $N \in [1, \infty]$ if  the following three conditions are satisfied.
\begin{itemize}
\item{(Volume growth)} There exist $C>1$ and $x \in X$ such that $\meas (B_r(x))\le Ce^{Cr^2}$ holds for any $r>1$.
\item{(Sobolev-to-Lipschitz)} If $f \in H^{1,2}(X, \dist, \meas)$ satisfies $|\nabla f|(x) \le 1$ for $\meas$-a.e. $x \in X$, then $f$ has a $1$-Lipschitz representative.
\item{(Bochner inequality)} It holds that
\begin{equation}
\frac{1}{2}\Delta |\nabla f|^2 \ge \frac{(\Delta f)^2}{N}+\langle \nabla \Delta f, \nabla f\rangle +K|\nabla f|^2
\end{equation}
in a weak sense.
\end{itemize}
See for instance \cite{AmbrosioGigliSavare13, AmbrosioGigliSavare14, AmbrosioMondinoSavare, CavallettiMilman, G2,ErbarKuwadaSturm} for the details above. We often say that $X$ is an $\RCD(K, N)$ space, without noting the distance and measure, for short.
%It is worth mentioning that for $\RCD(K, \infty)$ spaces, the $L^1$-Bakry-\'Emery estimate is satisfied in the sense;
%\begin{equation}\label{eq:be}
%|\nabla h_tf|(x) \le e^{-Kt}h_t|\nabla f|(x),\quad \text{for $\meas$-a.e. $x \in X$.}
%\end{equation}
%See \cite{Savare} for the proof.

By Bishop-Gromov inequality (\cite[Theorem 5.31]{LottVillani} and \cite[Theorem 2.3]{Sturm06b}),  and a $(1,1)$-Poincar\'e inequality (\cite[Theorem 1]{Rajala}), it is known  from \cite[Proposition 2.3]{Sturm95} and \cite[Corollary 3.3]{Sturm96}  that there exists a unique continuous function $\Heat:X \times X \times (0, \infty) \to (0, \infty)$, so-called the heat kernel of $(X, \dist, \meas)$, such that 
\begin{equation}\label{eq:heat kernel}
\heat_tf(x)=\int_Xf(y)\Heat (x,y,t)d\meas(y)
\end{equation}
for any $f \in L^2(X, \meas)$.
Moreover by \cite[Theorem 1.2]{JiangLiZhang}, Li-Yau inequality is satisfied in this setting, namely we know  the following Gaussian estimates for $\Heat$: for any $\epsilon \in (0, 1)$
\begin{align}\label{eq:gauusian}
\frac{C_1^{-1}}{\meas (B_{\sqrt{t}}(x))}\exp \left( -\frac{\dist (x, y)^2}{(4-\epsilon)t}-C_2t\right) &\le \Heat (x,y, t) \nonumber \\
&\le \frac{C_1}{\meas (B_{\sqrt{t}}(x))}\exp \left( -\frac{\dist(x, y)^2}{(4+\epsilon)t}+C_2t\right)
\end{align}
for some $C_1=C_1(K, N, \epsilon)>1$ and some $C_2=C_2(K, N, \epsilon) \ge 0$. Moreover $C_2 =0$ when $K=0$. 

Furthermore, with \cite[Theorem 1.2]{Jiang2015} and \cite[Theorem 4]{Davies}  we have 
\begin{equation}\label{eq:grad gaus}
|\nabla_x\Heat (x,y, t)| \le \frac{C_1}{\sqrt{t}\meas (B_{\sqrt{t}}(x))}\exp \left( -\frac{\dist (x, y)^2}{(4+\epsilon)t}+C_2t\right),
\end{equation}
which shows that $\Heat$ is locally Lipschitz, and 
\begin{equation}\label{eq:lap gaussian}
\left|\Delta_x\Heat (x,y,t)\right|=\left|\frac{d}{dt}\Heat (x,y,t)\right| \le \frac{C_1}{t\meas (B_{\sqrt{t}}(x))}\left( -\frac{\dist (x, y)^2}{(4+\epsilon)t}+C_2t\right). 
\end{equation}
It is worth pointing out that if $K=0$, then $C_2$ can be chosen as $0$ (see \cite{JiangLiZhang} for the details) and that the heat flow can be extended to $L^p(X, \meas)$ for any $1\le p \le \infty$ by (\ref{eq:heat kernel}) based on the estimates above (or by density).

Note that under rescaling $(X, a\dist, b\meas)$ for some $a>0$ and some $b>0$, the space is an $\RCD(a^{-2}K, N)$ space and we have the following:
\begin{equation} \label{eq:heat scaling}
\Heat_{(X, a\dist, b\meas)}(x, y, t)=b^{-1}\Heat_{(X, \dist, \meas)}(x, y, a^{-2}t).
\end{equation}

When $(X, \dist, \meas)$ is compact,  denoting by $f_i$ a corresponding eigenfunction of $\lambda_i$ with $\|f_i\|_{L^2}=1$, the standard functional analysis allows us to show that $\{f_i\}_i$ (can be chosen to) form an $L^2$-orthonormal basis of $L^2(X, \meas)$  and that 
\begin{equation}\label{eq:heat exp}
\Heat (x,y,t)=\sum_ie^{-\lambda_it}f_i(x)f_i(y),\quad \text{in $C(X \times X)$}
\end{equation}
is satisfied.

Finally let us mention that the all objects above are preserved by isomorphisms of metric measure spaces. In particular if
$T: (X_1, \dist_1, \meas_1) \rightarrow (X_2, \dist_2, \meas_2)$ is an isomorphism, namely isometry and measure preserving, between $\RCD(K, N)$ spaces, then 
\begin{equation} \label{eq:heat isom}
 \Heat_{(X_2, \dist_2, \meas_2)}(T(x), T(y), t) = \Heat_{(X_1, \dist_1, \meas_1)}(x,y,t).  
\end{equation} 
\begin{rem}
The local isomorphism preserves the gradient operators  and the Laplacians because of their localities, however it does not preserve the heat kernel in general as discussed in \S \ref{subsec:weyl}.
\end{rem}
\section{Geometry of Pan-Wei's examples}

First we recall the construction of the Ricci limit space $Y$ in \cite{PanWei}. Consider 
$M=M^{n+1}=[0,\infty)\times_f \mathbb{S}^{n-1} \times_h \mathbb{S}^1$, with
$$f(r)=r(1+r^2)^{-\tfrac 14} \sim \sqrt{r},\quad h(r)=(1+r^2)^{-\alpha} \sim r^{-2\alpha} \ \ \mbox{as} \ r \rightarrow \infty, $$ where $\alpha>0$. Then $M$ is a Riemannian manifold diffeomorphic to $\mathbb{R}^n\times \mathbb{S}^1$.
When integer $n\ge \max\{4\alpha+3, 16\alpha^2+8\alpha+1\}$, $M$ has positive Ricci curvature.

 Let $p\in M$ be a point in $\{r=0\}$,  $(\widetilde{M},\tilde{p})$ be the universal cover of $(M,p)$. Let $Y$ be the asymptotic cone of $(\widetilde{M},\tilde{p})$. Then $Y$ is clearly an $\RCD(0, N)$ space with $N=n+1$.

\begin{thm}[Pan-Wei] \label{Pan-Wei example} When $\alpha\ge 1/2$, the Hausdorff dimension of the Ricci limit space $Y$ is $1+2\alpha$, which can be arbitrarily large by raising $\alpha$ and $n$, while the rectifiable dimension of $Y$ is $2$.  One can also construct a compact $\RCD(K, n+1)$ space $X$ with this feature for some negative $K$.
\end{thm}

For more details, see \cite[Theorem A]{PanWei}; for construction of a compact $X$, see \cite[Remark 1.8]{PanWei} and \S \ref{heat=X}. Also see \cite{PanWei22}, where similar constructions give examples whose Busemann functions at a point are not proper.

In order to study the Weyl's law of $X$, we derive precise descriptions of the distance, metric, measure of $Y$  in the subsections below, extending the distance estimate in \cite{PanWei}.

\subsection{Metric and dilations of $Y$}

Let $S$ be a surface of revolution $[0,\infty)\times_h \mathbb{S}^1$ with boundary and let $\widetilde{S}$ be its universal cover. Let $q\in S$ at $r=0$ and let $\tilde{q}\in \widetilde{S}$ be a lift of $q$. 

\begin{lem}\label{M to S}
	Let $r_i\to\infty$ be a sequence. Then the two sequences $(r_i^{-1}\widetilde{M},\tilde{p})$ and $(r_i^{-1}\widetilde{S},\tilde{q})$ are equivalent in the pointed Gromov-Hausdorff topology, namely, pointed Gromov-Hausdorff distances between $(r_i^{-1}\widetilde{M},\tilde{p})$ and $(r_i^{-1}\widetilde{S},\tilde{q})$ go to $0$ as $i \to \infty$.
\end{lem}

\begin{proof}
	We write $(r,v)$ as a point in $S$, where $r\in[0,\infty)$ and $v\in \mathbb{S}^1=[-1/2,1/2]/\sim$. This also naturally defines an $(r,v)$-coordinate on $\widetilde{S}$, where $r\in[0,\infty)$ and $v\in \mathbb{R}$, such that the preimage of $(r,0)\in S$ is
	$\{(r,v)|v\in\mathbb{Z}\}.$
	Recall that the Riemannian metric on $S$ is given by 
	$g_S=dr^2 + h(r)^2 dv^2;$
	this expression also gives the Riemannian metric on $\widetilde{S}$. 
	
	On $M$, we can write each point in the form of $(r,x,v)$, where $r\in[0,\infty)$, $x\in \mathbb{S}^{n-1}$, and $v\in \mathbb{S}^1$. From $g$ on $M$, it is clear that
	$$\dist^g((r,x,v),(r,x',v))\le f(r) \dist_{n-1}(x,x'),$$
	where $\dist_{n-1}$ is the standard distance on the unit sphere $\mathbb{S}^{n-1}$. We fix a point $x_0 \in \mathbb{S}^{n-1}$ and define an immersion $$\phi: S \to M ,\quad (r,v) \mapsto (r,x_0,v).$$
	Note that $\phi$ is indeed a Riemannian immersion and $\phi(S)$ is convex. For $r>0$, we write
	$$S_r=\{(r,v)| v\in \mathbb{S}^1\},\quad M_r=\{(r,x,v)|x\in S^{n-1},v\in S^1\}.$$
	$\phi(S_r)$ is $f(r)$-dense in $M_r$.
	
	Similar to $\widetilde{S}$, we can naturally assign an $(r,x,v)$-coordinate on $\widetilde{M}$. We also write
	$$\widetilde{S}_r=\{(r,v)\in\widetilde{S}| v\in \mathbb{R}\},\quad \widetilde{M}_r=\{(r,x,v)\in\widetilde{M}|x\in \mathbb{S}^{n-1},v\in \mathbb{R}\}.$$
	The map $\phi$ lifts to a Riemannian immersion
	$$\tilde{\phi}:\widetilde{S}\to \widetilde{M},\quad (r,v)\mapsto (r,x_0,v)$$
	with a convex image. Moreover, $\tilde{\phi}(\widetilde{S}_r)$ is $f(r)$-dense in $\widetilde{M}_r$. Now the result follows from the fact that $f(r)\sim \sqrt{r}$ as $r\to \infty$. 
\end{proof}

\begin{lem}\label{S convergence}
	We write $\widetilde{S}_{+}=\{(r,v)\in\widetilde{S}|r>0,v\in \mathbb{R}\}$ and set a reference point $z=(1,0)\in\widetilde{S}_{+}$. Then as $\lambda\to\infty$, $(\widetilde{S}_+,\lambda^{-2}g_{\widetilde{S}},(1,0))$ converges (smoothly on each compact subset of $(0,\infty)\times \mathbb{R}$) to a limit (incomplete) Riemannian metric
	\begin{equation}
	g_\infty=dr^2+r^{-4\alpha} dv^2
	\end{equation}
	defined on $(0,\infty)\times \mathbb{R}$; moreover, the reference point $z$ converges to $(0,0)$, a point in the metric completion of $g_\infty$.
\end{lem}

\begin{proof}
	For convenience, we write $\lambda^{-2}g_{\widetilde{S}}=g_\lambda$ below. We apply a change of variables $s=\lambda^{-1}r$ and $w=\lambda^{-1-2\alpha}v$. Then
	\begin{align*}
		g_\lambda&=\lambda^{-2}(dr^2+(1+r^2)^{-2\alpha}dv^2)\\
		&=ds^2+\left(\dfrac{1+\lambda^2s^2}{\lambda^2}\right)^{-2\alpha} dw^2.
	\end{align*}
	In other words, the space $(\widetilde{S}_+,g_\lambda,z)$ is isometric to 
	$$((0,\infty)\times \mathbb{R},ds^2+((1+\lambda^2s^2)/\lambda^2)^{-2\alpha}dw^2,(\lambda^{-1},0)).$$
	Let $\lambda\to\infty$, then the result follows.
\end{proof}

\begin{rem}\label{rem:change}
The coordinate change $y=(\frac{r}{1+2\alpha})^{1+2\alpha}, x=\frac{v}{(1+2\alpha)^{2\alpha}}$ turns the metric into the form
$$
\frac{dx^2+dy^2}{y^{2c}}, \ \ \ \ \ \ c=\frac{2\alpha}{1+2\alpha}.
$$
Note that the hyperbolic metric corresponds to $c=1$.
\end{rem}

\begin{thm}\label{dist}
	$\widetilde{M}$ has a unique asymptotic cone $(Y,y)$ with the following properties:\\
	(1) $Y$ has an $(r,v)$-coordinate, where $(r,v)\in[0,\infty)\times\mathbb{R}$, and $y=(0,0)$;\\
	(2) the distance function $\dist$ on $Y$ is induced by the incomplete Riemannian metric $g_\infty=dr^2+ r^{-4\alpha} dv^2$ defined on $(0,\infty)\times \mathbb{R}$;\\
	(3) for each $\lambda>0$, the map $F_\lambda: (r,v)\mapsto (\lambda r, \lambda^{1+2\alpha }v)$ is a metric dilation of $Y$ with scale $\lambda$, that is,
	$$\dist (F_\lambda(y_1),F_\lambda(y_2))=\lambda\cdot \dist (y_1,y_2)$$
	for all $y_1,y_2\in Y$; \\
	(4) $\dist ((0,v),(0,0))=C|v|^{\frac{1}{1+2\alpha}}$ for all $v\in\mathbb{R}$, where $C>0$ is a constant.
\end{thm}

\begin{proof}
	(1) and (2) follow directly from Lemmas \ref{M to S} and \ref{S convergence}. 
	
	We prove (3). It is direct to check that $F_\lambda$ satisfies
	$$F_\lambda^\star g_\infty=\lambda^2 g_\infty.$$
	Therefore, $F_\lambda$ is a metric dilation with scale $\lambda$ on $(0,\infty)\times \mathbb{R}$. We extend $F_\lambda$ to $Y$ by continuity. The result follows.
	
	For (4), we assume $v>0$ without lose of generality. Applying the metric dilation $F_\lambda$ with $\lambda=v^{\frac{1}{1+2\alpha}}$, we have
	$$\dist ((0,v),(0,0))=\dist (F_\lambda(0,1),(0,0))=\lambda \dist((0,1),(0,0)).$$
\end{proof}

We make some remarks related to Theorem \ref{dist}.

\begin{rem}
	It follows from the metric dilations in Theorem \ref{dist}(3) that $Y$ is scaling invariant, that is, $(sY,y)$ is isometric to $(Y,y)$ for all $s>0$. In particular, the tangent cone of $Y$ at $y$ and the one at infinity are both unique and isometric to $Y$.
\end{rem}

\begin{rem}\label{rem:infinite}
Take a box $B=(0,1]\times [0,1]\subseteq \mathcal{R}$. Then
	$$\mathcal{H}^2(B)=\int_B r^{-2\alpha} drdv = \int_0^1 r^{-2\alpha} dr. $$
	The indefinite integral is finite if and only if $\alpha< 1/2$. Recall that $\dim_H(\mathcal{S})=1+2\alpha$ and $\dim_H(\mathcal{R})=2$. Thus in these examples, we always have
	$$\dim_H (\mathcal{S})<\dim_H(\mathcal{R}) \iff \mathcal{H}^2(B)<\infty.$$
\end{rem}

\begin{rem}\label{rem:counterexample}
	When $\alpha=1/2$, $Y$ has both Hausdorff dimension and rectifiable dimension $2$. Based on Remark \ref{rem:infinite}, we can see that $(Y,\dist,\mathcal{H}^2)$ does not satisfy the doubling condition. In fact, let $x=(1,0)\in Y$ expressed in the $rv$-coordinate and let $r=2/3$. Then by triangle inequality, $B_{2r}(x)$ contains the set $D=[0,1/3]\times [-(1/3C)^2,(1/3C)^2]$, where $C$ is the constant in Theorem \ref{dist}(4). $D$, thus $B_{2r}(x)$, has infinite $\mathcal{H}^2$ measure from Remark \ref{rem:infinite}. On the other hand, $\mathcal{H}^2(B_r(x))$ is finite. Therefore, there is no upper bound for the ratio
	$$\dfrac{\mathcal{H}^2(B_{2r}(x))}{\mathcal{H}^2(B_{r}(x))}.$$
	This answers \cite[Conjecture 1.34]{ChCo97} in the negative even in the case when the space is $2$-dimensional with the full support $\mathcal{H}^2$.
\end{rem}

\begin{rem}
	The Grushin plane is $\mathbb{R}^2$ with a subRiemannian metric (see \cite[Section 3.1]{Bel}). Under the $(x,y)$-coordinate, its distribution is generated by $X=\partial_x$ and $Y=|x| \partial_y$. Setting $\{X,Y\}$ orthonormal defines a subRiemannian metric on $\mathbb{R}^2$. Note that $Y$ degenerates only along the $y$-axis. Thus the distribution has full rank almost everywhere; a sub-Riemannian manifold with this property are called almost-Riemannian. One can also define the $\alpha$-Grushin planes by alternatively setting $Y=|x|^\alpha \partial_y$, where $\alpha>0$. Outside the $y$-axis, the sub-Riemannian metric becomes Riemannian and equals 
	$$g= dx^2 + |x|^{-2\alpha} dy^2.$$
	By Theorem \ref{dist}, our example $Y$ is isometric to the $2\alpha$-Grushin halfplane.
\end{rem}

\begin{rem}
It is conjectured in \cite{KKK} that any $\RCD(K,N)$ space with an upper curvature bound in the sense of Alexandrov, namely CAT$(\kappa)$ space, satisfies that the topological dimension coincides with the Hausdorff dimension. See \cite[Theorem 3.15]{KKK} and a sentence afterwords.
Since the metric structure of $Y$ is CAT$(0)$ (the proof is the same to that of a standard fact that any Hadamard manifold is a CAT$(0)$ space, because the smooth part of $Y$ has nonpositive sectional curvature and all two points in $Y$ can be joined by a unique minimal geodesic), $Y$ gives a counterexample to this conjecture.
\end{rem}

Let 
$$\gamma:Y\to Y,\quad (r,v)\mapsto (r,v+1),$$
which is an isometry of $Y$. We prove the following distance estimate which will be used later. 

\begin{lem} \label{r-distance}
	There is a function $C(r)>0$ such that
	$$C(r) l^{\frac{1}{1+2\alpha}} \le \dist(\gamma^l x,x)\le 3 l^{\frac{1}{1+2\alpha}}$$
	holds for all $x=(r,v)\in Y$ and all $l\in \mathbb{Z}_+$.
\end{lem}

\begin{proof}
	Without lose of generality, we assume that $x=(r,0)$. When $r=0$, we have
	$$\dist (\gamma^l x,x)=\dist ((0,l),(0,0))=C l^{\frac{1}{1+2\alpha}}$$
	from (4) of Theorem \ref{dist}. Below we assume $r>0$.
	
	Let $s>0$ and we consider a path $\sigma_s$ from $x$ to $\gamma^l(x)=(r,l)$ as follows: $\sigma_s$ first go through a horizontal segment from $(r,0)$ to $(r+s,0)$, then go through a $v$-curve to $(r+s,l)$, then go back to $(r+s,0)$ via a horizontal segment. We have
	$$\dist (\gamma^l x,x) \le \mathrm{length}(\sigma_s) =2s+(r_0+s)^{-2\alpha}l$$
	for all $s\ge 0$. Using $s=l^{\frac{1}{1+2\alpha}}$, we have
	$$\dist (\gamma^l x,x)\le 2l^{\frac{1}{1+2\alpha}}+\left(r+l^{\frac{1}{1+2\alpha}}\right)^{-2\alpha}l \le 3 l^{\frac{1}{1+2\alpha}}.$$
	
	Let $c$ be a minimal geodesic from $(r,0)$ to $(r,l)$. We set 
	$$R:=\max \{t-r|(t,v)\in c\}\ge 0.$$
	Then 
	$$\dist (\gamma^l x,x)= \mathrm{length}(c)\ge (r+R)^{-2\alpha} l.$$
	Together with the upper bound, they yield
	$$(r+R)^{-2\alpha} l \le 3 l^{\frac{1}{1+2\alpha}}.$$
	Thus
	$$R\ge C_1 l^{\frac{1}{1+2\alpha}} -r,$$
	where $C_1=3^{-1/(2\alpha)}$. When $2r\le C_1l^{\frac{1}{1+2\alpha}}$, we have
	$$\dist (\gamma^l x,x)\ge 2R \ge 2C_1l^{\frac{1}{1+2\alpha}}-2r\ge C_1 l^{\frac{1}{1+2\alpha}}.$$ 
	
	Next, we consider the case $2r>C_1l^{\frac{1}{1+2\alpha}}$, that is, $r^{-2\alpha} l<(2/C_1)^{1+2\alpha}r$. Note that
	$$R\le \dfrac{1}{2}\dist (\gamma^lx,x)\le \dfrac{1}{2}r^{-2\alpha} l.$$
	Thus 
	\begin{align*}
		\dist (\gamma^l x,x)&\ge (r+R)^{-2\alpha} l\ge \left(r+\dfrac{1}{2}r^{-2\alpha}l\right)^{-2\alpha}l\\
		&\ge \left(r+\dfrac{1}{2} \left(\dfrac{2}{C_1}\right)^{1+2\alpha}r \right)^{-2\alpha}l\\
		&= C_2 r^{-2\alpha} l \ge C_2 r^{-2\alpha} l^{\frac{1}{1+2\alpha}},    	
	\end{align*}
	where $C_2=(1+2^{2\alpha}/C_1^{1+2\alpha})^{-2\alpha}$. The last inequality holds because $l\ge 1$.
	
	In conclusion, we choose $C(r)=\min \{C_1,C_2r^{-2\alpha}\}$. Then the desired inequality holds.
\end{proof}

\subsection{Limit measure}
 In this subsection we derive the limit measure of $Y$. 
%The metric ball $B_r(p)$ in $M$ has the volume estimate
%$$\dfrac{\mathrm{vol}^g(B_r(p))}{r^{\frac{n+1}{2}-2\alpha}}\to C \in (0,\infty) $$
%as $s\to \infty$.
%Thus the asymptotic cone $(\bar X,\bar{x})=([0,\infty),0)$ of $M$ has a unique limit measure 
%$$\bar{ \meas }=r^{\frac{n-1}{2}-2\alpha}dr.$$

%On $Y$ the limit measure is 
%\begin{equation} 
%	d\meas = r^{\frac{p-1}{2}-2\alpha}dr dv.   \label{limit-measure}
%	\end{equation}

\begin{lem}\label{measure density}
		Let $\meas$ be a limit measure on $Y$. Then 
		\begin{equation}  d\meas=c_{\meas} r^{\frac{n-1}{2}-2\alpha} drdv  \label{limit-measure}
				\end{equation}
		for some constant $c_{\meas} \in [C_1,C_2]$, where $C_1,C_2>0$ are constants depending only on $n$ and $\alpha$.
	\end{lem}
	
	\begin{proof}
		We fix a sequence $\lambda_i\to \infty$. We first estimate the volume of $B_{\lambda_i}(\tilde{p})$. We continue to use the $(r,x,v)$-coordinate on $\widetilde{M}$, introduced in the proof of Lemma \ref{M to S}. By \cite[Lemma 1.1]{PanWei}, we have distance estimate
		$$C_1|l|^{\frac{1}{1+2\alpha}}-2\le \dist ((0,x,0),(0,x,l))\le C_2|l|^{\frac{1}{1+2\alpha}}$$
		for all $l\in\mathbb{Z}$ large. We denote 
		$$D_i=\{(r,x,v)|0\le r\le \lambda_i, x\in \mathbb{S}^{n-1},|v|\le C_3\lambda^{1+2\alpha}_i \}\subseteq \widetilde{M},$$
		where $C_3= (2C_1)^{-1-2\alpha}$. $D_i$ has volume
		$$\mathrm{vol}(D_i)=2C_3 \lambda_i^{1+2\alpha} \int_0^{\lambda_i} f(r)^{n-1} h(r) dr.$$
		 Note that the boundary $\partial D_i$ equals 
		$$\left([0,\lambda_i]\times \mathbb{S}^{n-1} \times \{\pm C_3\lambda_i^{1+2\alpha}\}\right)\cup \left( \{\lambda_i\}\times \mathbb{S}^{n-1} \times [-C_3\lambda^{1+2\alpha}_i,C_3\lambda^{1+2\alpha}_i] \right)$$
		By triangle inequality, for any $(r,x,v)\in D_i$,
		$$\dist (\tilde{p},(r,x,v))\le C_2|v|^{\frac{1}{1+2\alpha}}+1+r\le(C_2C_3^{\frac{1}{1+2\alpha}}+1)\lambda_i+1.$$
		Also, when $(r,x,v)\in \{\lambda_i\}\times \mathbb{S}^{k-1} \times [-C_3\lambda^{1+2\alpha}_i,C_3\lambda^{1+2\alpha}_i]$,$$\dist (\tilde{p},(r,x,v))\ge \lambda_i;$$
		when $(r,x,v)\in [0,\lambda_i]\times \mathbb{S}^{k-1} \times \{\pm C_3\lambda_i^{1+2\alpha}\}$,
		$$\dist (\tilde{p},(r,x,v))\ge \dist (\tilde{p},(0,x,v))-\dist ((0,x,v),(r,x,v))\ge C_1|v|^{\frac{1}{1+2\alpha}}-\lambda_i \ge \lambda_i.$$ 
		Therefore,
		$$B_{\lambda_i}(\tilde{p})\subseteq D_i \subseteq B_{C_4 \lambda_i} (\tilde{p})$$
		holds for all $i$ large, where $C_4=C_2C_3^{\frac{1}{1+2\alpha}}+2$. By Bishop-Gromov inequality, there is a constant $C_5>0$, depending on $n$ and $\alpha$, such that
		$$  C_5 \mathrm{vol}(D_i) \le \mathrm{vol}(B_{\lambda_i}(\tilde{p}))\le \mathrm{vol}(D_i).$$ 
		Passing to a subsequence of $\lambda_i$, we assume that
		$$\dfrac{\mathrm{vol}(B_{\lambda_i}(\tilde{p}))}{\mathrm{vol}(D_i)}\to c_0 \in [C_5,1]$$
		as $i\to\infty$.
		
		To check that the limit measure $\meas$ on $Y$ has the stated expression, it suffices to check $\meas$ on each rectangle in $Y$. Let 
	    $$R=[a,b]\times [c,d]\subseteq Y$$
	    written in $(r,v)$-coordinate of $Y$, where $0\le a<b$ and $c<d$.
	    Let 
	    $$R_i=\{(r,x,v)|r\in[\lambda_i a,\lambda_i b],x\in \mathbb{S}^{n-1},v\in [\lambda_i^{1+2\alpha}c,\lambda_i^{1+2\alpha}d] \}\subseteq \widetilde{M}.$$
	    By Lemma \ref{M to S} and the proof of Lemma \ref{S convergence}, $R_i$ converges to $R$ with respect to the convergent sequence
	    $$(\widetilde{M},\lambda_i^{-2}g_{\widetilde{M}},\tilde{p})\overset{GH}\longrightarrow (Y,\dist,y).$$
	    $R_i$ has volume 
	    $$\mathrm{vol}(R_i)=\lambda_i^{1+2\alpha}(c-d)\int_{\lambda_i a}^{\lambda_i b} f(r)^{n-1}h(r) dr.$$
	    Then 
	    \begin{align*}
	    \meas (R)&=\lim\limits_{i\to\infty}\dfrac{\mathrm{vol}(R_i)}{\mathrm{vol}(B_{\lambda_i}(\tilde{p}))}= \lim\limits_{i\to\infty}\dfrac{\mathrm{vol}(R_i)}{c_0 \mathrm{vol}(D_i)}\\
	    &=c_0^{-1} \lim\limits_{i\to\infty} \dfrac{\lambda_i^{1+2\alpha}(c-d)\int_{\lambda_i a}^{\lambda_i b} f(r)^{n-1}h(r) dr}{2C_3 \lambda_i^{1+2\alpha} \int_0^{\lambda_i} f(r)^{n-1} h(r) dr}\\
	    &= (2c_0C_3)^{-1}\cdot (c-d) \cdot \left( b^{\frac{n+1}{2}-2\alpha}-a^{\frac{n+1}{2}-2\alpha} \right)\\
	    &= \int_R c_{\meas} r^{\frac{n-1}{2}-2\alpha} drdv,
	    \end{align*}
    where $c_{\meas}=(2c_0C_3)^{-1}\cdot (\frac{n+1}{2}-2\alpha)^{-1}$. Here in order to realize the limit formula for $\meas(R)$, we used a fact that the boundary of $R$ is $\meas$-negligible which is justified by, for instance, \cite[Theorem 4.6]{ChCo00b}.
	\end{proof}

\begin{rem}
	One can show directly that our  $Y=[0, \infty)\times \mathbb R$ with the metric $dr^2 + r^{-4\alpha} dv^2 (\alpha >0)$, and measure $r^{\frac{n-1}{2} -2\alpha} dr dv$ is $\RCD(0, n+1)$ when $n\ge \max\{4\alpha+3, 16\alpha^2+8\alpha+1\}$ by checking the $n$-Bakry-\`Emery Ricci curvature is nonnegative under the condition in the interior and that the second fundamental form of the boundary of $Y_{r_0}=[r_0, \infty)\times \mathbb R$ is strictly positive for $r_0 >0$, and let $r_0 \rightarrow 0$. 
\end{rem}

Let us introduce another immediate consequence of the lemma above. For any open subset $U$, $H^{1,2}_0(U, \dist, \meas)$ denotes the $H^{1,2}$-closure of $\mathrm{Lip}_c(U, \dist)$. See for instance \cite{bb} for the definition of the capacity.
\begin{prop}\label{propcap}
The singular set $\mathcal{S}$ has null $2$-capacity with respect to $\meas$, namely $H^{1,2}(Y, \dist, \meas)=H^{1,2}_0(\mathcal{R}, \dist, \meas)$. In particular $H^{1,2}(Y, \dist, \meas)$ coincides with the $H^{1,2}$-closure of $C^{\infty}_c(\mathcal{R})$. 
\end{prop}
\begin{proof}
It is enough to check that for any $f \in \mathrm{Lip}_c(X, \dist)$ there exists a sequence $f_i \in \mathrm{Lip}_c(\mathcal{R}, \dist)$ such that $f_i \to f$ in $H^{1,2}(X, \dist, \meas)$. First the  lemma above tells us that for any bounded open subset $U$
\begin{equation}\label{eq:bdd}
\int_U\frac{1}{\dist_{\mathcal{S}}^{2}}d\meas<\infty
\end{equation}
holds, where $\dist_{\mathcal{S}}$ is the distance function from $\mathcal{S}$. Take cut-off functions $\phi_r$ with $\phi_r =0$ on $B_r(\mathcal{S})$, $\phi_r=1$ on $Y \setminus B_{2r}(\mathcal{S})$ and $|\nabla \phi_r| \le r^{-1}$. Then the sequence $\phi_rf \in \mathrm{Lip}_c(\mathcal{R}, \dist)$ satisfies 
\begin{equation}
\int_Y|\nabla (\phi_rf)|^2d\meas \le 2\sup |f|^2\int_Y|\nabla \phi_r|^2d \meas +2\int_Y|\nabla f|^2d\meas \le C<\infty
\end{equation}
because of (\ref{eq:bdd}),
namely $\phi_rf$ is a bounded sequence in $H^{1,2}(Y, \dist, \meas)$. Thus $\phi_rf$ $H^{1,2}$-weakly converge to $f$ as $r \to 0^+$. Thus Mazur's lemma completes the proof.
\end{proof}

\begin{rem}
We can prove that if
\begin{equation}\label{eq:dense}
\int_U\frac{1}{\dist_{\mathcal{S}}^4}d \meas<\infty
\end{equation}
for any bounded open subset $U \subset Y$, 
then  $D(\Delta)$ coincides with the closure of $C^{\infty}_c(\mathcal{R})$, where the norm of $D(\Delta)$ is $(\|f||_{H^{1,2}}^2+\|\Delta f\|_{L^2}^2)^{1/2}$. In particular the Laplacian of any $f \in D(\Delta)$ can be approximated in $L^2$ by the Witten Laplacians of smooth functions with compact supports in $\mathcal{R}$.
Since this is not directly related to our works below, let us only give a sketch of the proof as follows.

Let $f \in D(\Delta)$. Considering the heat flow $\heat_tf$ (which converge to $f$ in $D(\Delta)$ as $t \to 0^+$), it is enough to consider the case when $f$ is smooth on $\mathcal{R}$ (c.f. \cite[Theorem 7.20]{Gr}). 
Recalling the proof of the existence of good cut-off functions in \cite[Lemma 3.1]{MN}, applying such cut-off functions, with no loss of generality, we can assume that $f$ has bounded support and that for any $0<r<1$ we can find a smooth cut-off function $\phi_r \in C^{\infty}(\mathcal{R})$ with $\phi_r=0$ on $B_r(\mathcal{S})$, $\phi_r=1$ on $Y \setminus B_{2r}(\mathcal{S})$, and $r|\nabla \phi_r|+r^2|\Delta \phi_r| \le C$. Then a similar argument as in the proof of  Proposition \ref{propcap} with (\ref{eq:dense}) allows us to conclude that a sequence of finite convex combinations of $\phi_{r_i}f$ converge to $f$ in $D(\Delta)$, where $r_i \to 0^+$. Thus we get the desired result.

It is worth pointing out that (\ref{eq:dense}) is satisfied if $n$ is large. 
\end{rem}

\subsection{ Heat Kernel of $Y$}
%For a metric measure space $(X, \dist, \meas)$, let $p(x, y, t)$ be the heat kernel. We have the following relation under scaling
%$$
%p_{(X, a\dist, b\meas)}(x, y, t)=b^{-1}p_{(X, \dist, \meas)}(x, y, a^{-2}t).
%$$
%If $T: (X_1, \dist_1, \meas_1) \rightarrow (X_2, \dist_2, \meas_2)$ is an isomorphism, i.e. isometry and measure preserving, then 
%\[ p_{(X_2, \dist_2, \meas_2)}(T(x), T(y), t) = p_{(X_1, \dist_1, \meas_1)}(x,y,t).  \]
In the rest of this subsection we denote by $\Heat (x,y,t)$ the heat kernel of the Ricci limit space $(Y,\dist, \meas)$.

Since $F_\lambda: (Y, \lambda \dist, \meas) \rightarrow (Y,\dist, (F_\lambda)_*\meas)$ is an isomorphism, and $$(F_\lambda)_*\meas = \lambda^{-\frac{n+3}{2}} \meas,$$ by \eqref{eq:heat isom} and \eqref{eq:heat scaling} we have
\begin{equation}
\Heat (x,y,t) = \lambda^{\frac{n+3}{2}} \Heat (F_\lambda(x), F_\lambda(y), \lambda^{2}t). \label{heat-dilation}
\end{equation}

Denote $s = \sqrt{t},$  and put $$ q(x,y,s) = \meas(B_s(x)) \cdot \Heat (x,y,s^2).$$ 
Since $\meas(B_s(x)) = s^{\frac{n+3}{2}} \meas(B_1(F_{1/s}(x)))$, combining with \eqref{heat-dilation}, we have 
\begin{equation}
q(x,y,s) = q(F_\lambda(x),F_\lambda(y), \lambda s).  \label{q-scale}
\end{equation}
%As $F_\lambda(x) \in \Ss$ for all $x \in \Ss, \lambda >0$, by \eqref{q-scale} and the translational invariance, we have $q(x,x,s)= q(x,x,1) = c_1$ for all $x \in \Ss$. 

For $x= (r,v) \in Y$, $q(x,x,s)$ only depends on $r, \ s$. Denote $h(r,s) = q(x,x,s)$. 
\begin{lem} \label{lem:h} $h$ is invariant under the space-time scaling, i.e.
for all $\lambda >0, \ s >0, \ r \ge 0$ 	\[ h(\lambda r, \lambda s) = h(r, s).\]
In particular $h(0,s) = h(0,1)$. 
	Moreover  for some $C=C(n)>1$ \[
		C^{-1} \le h(r,s) \le C. \]		
\end{lem}
\begin{proof}
	The space-time scaling invariance is a reformulation of  \eqref{q-scale} by setting $x=y=(r, 0)$. The inequality follows from that $Y$ is an $\RCD(0,n+1)$ space and 
 the Li-Yau estimate (\ref{eq:gauusian}) when $K=0$. %for all $\RCD(K,N)$ spaces \cite{JiangLiZhang}, for  $\forall s >0, x,y$, 
%\begin{equation}   \label{Li-Yau}
	%C_1^{-1} \exp(-\frac{d^2(x,y)}{4(1-\epsilon)s^2} -C_2 s^2) \le q(x,y,s) \le C_1  \exp(-\frac{d^2(x,y)}{4(1+\epsilon)s^2} +C_2 s^2),
%\end{equation}
%where $\epsilon >0$, and $C_1, C_2$ are constants depend only on  $K,N,\epsilon$, and $C_2 =0$ when $K=0$. 
\end{proof}

\subsection{Measure of balls in $Y$}

First we give the following estimate on $\meas (B_s(x))$ when the ball is away from the singular set of $Y$.
\begin{lem}  \label{ball-asym}
	For $x = (r_0, v_0) \in Y$ and $r_0 > s>0$, we have 
	\begin{equation}
c_{\meas} \pi s^2 (r_0 -s)^{\frac{n-1}{2}} \le		\meas (B_s(x)) \le c_{\meas} \pi s^2 (r_0 +s)^{\frac{n-1}{2}}. 
	\end{equation}
In particular, $\meas (B_s(x)) =  c_{\meas} \pi s^2  r_0^{\frac{n-1}{2}} [1+ O(r^{-1}_0s)]$ when $ s \rightarrow 0$ or $ r_0 > > s.$
\end{lem}
\begin{proof}
	Since the distance on the regular set of $Y$ is induced by the Riemannian metric \[
	g = dr^2 + r^{-4\alpha} dv^2,\]
	we have $B_s(x) \subseteq \{ (r,v) \, | \, r_0-s < r < r_0 +s\}$ and \[
	B_s^{g_2}(x) \subseteq B_s(x) \subseteq B_s^{g_1}(x), \]
	where $g_1 = dr^2 + (r_0 +s)^{-4\alpha} dv^2, \ g_2 =  dr^2 + (r_0 -s)^{-4\alpha} dv^2$ are Euclidean metrics on regular set of $Y$. Hence 
	\begin{eqnarray*}
		\meas (B_s(x)) & \le & \int_{ B_s^{g_1}(x)} d\meas = c_{\meas} \int_{ B_s^{g_1}(x)} r^{\frac{n-1}{2}-2\alpha} dr dv \\
		& \le & c_{\meas} \int_{ B_s^{g_1}(x)} (r_0+s)^{\frac{n-1}{2}-2\alpha} dr dv = (r_0 +s)^{\frac{n-1}{2}} \mathrm{vol}^{g_1} (B_s^{g_1}(x))\\
		& =  & c_{\meas} (r_0 +s)^{\frac{n-1}{2}}  \pi s^2.
	\end{eqnarray*}
%Same estimate gives 
Similarly we have the lower bound. 
\end{proof}

In general, to get a handle on $\meas (B_s(x))$, we compare it to a ``rectangular" domain which is easier to compute, as follows. 

Note $B_s(x) = F_s (B_1 (x_s))$, where $x_s=F_s^{-1} (x) = (r_0 s^{-1}, v_0 s^{-2\alpha -1})$ for $x = (r_0, v_0) \in Y$. Denote the ``unit rectangle" at $x_s$ by
 $$C_1(x_s) = \{(r,v)\, |\, (r_0 s^{-1} -1)_+ < r < r_0 s^{-1} +1, -1 < v < 1 \}.$$ 
Then 
\begin{equation}
	\frac{\meas (B_s(x))}{\meas (F_s (C_1(x_s)))} = \frac{\meas(B_1 (x_s))}{\meas(C_1(x_s))} 
\end{equation}

Since $F_s (C_1(x_s)) = \{(r,v)\, |\, (r_0  - s)_+ < r < r_0  +s,\ - s^{2\alpha +1} < v < s^{2\alpha +1} \}$, we compute, as $n+1 > 4\alpha$, 
\begin{align*}
	&\meas (F_s (C_1(x_s)))\\
	=\ & c_{\meas} \int_{-s^{1+2\alpha}}^{s^{1+2\alpha}} \int_{(r_0-s)_+}^{r_0+s} r^{\frac{n-1}{2}-2\alpha} dr dv \\
	=\ & c_{\meas} \frac{2}{\frac{n+1}{2}-2\alpha} s^{1+2\alpha} \left((r_0 +s)^{\frac{n+1}{2}-2\alpha} -  (r_0-s)_+^{\frac{n+1}{2}-2\alpha} \right) \\
	=\ & \frac{4c_{\meas}}{n+1-4\alpha} s^{1+2\alpha} r_0^{\frac{n+1}{2}-2\alpha}\left((1 +sr_0^{-1})^{\frac{n+1}{2}-2\alpha} -  (1-sr_0^{-1})_+^{\frac{n+1}{2}-2\alpha} \right) \\
	= \ & 4c_{\meas}  s^{2+2\alpha} r_0^{\frac{n-1}{2}-2\alpha} [1 + O(r_0^{-1}s)] \ \mbox{when} \ \ r_0 > > s.
\end{align*}
 Combining the above with Lemma~\ref{ball-asym}, we have for  $r_0 > > s,$ 
$$	\frac{\meas (F_s (C_1(x_s)))}{\meas (B_s(x))} = \frac{4}{\pi} (sr_0^{-1})^{2\alpha}  [1 + O(r_0^{-1}s)]. $$ 

Since $x_s = (r_0 s^{-1}, v_0 s^{-2\alpha -1})$ and the measure is independent of the $v$-coordinate, $\frac{\meas(C_1(x_s))}{\meas(B_1 (x_s))}$ is a continuous function of $r_0 s^{-1}$ only. Denote $$\tau = sr_0^{-1}, \ \  \ \ 
\frac{\meas(C_1(x_s))}{\meas(B_1 (x_s))} =  f(\tau^{-1}).$$

In summary we have 
\begin{prop}\label{prop:vol}
	Let $C= \frac{4c_{\meas}}{n+1-4\alpha}$. For $x = (r_0, v_0) \in Y$, when $r_0 >0$, \begin{equation}
\frac{1}{	\meas (B_s(x))} = 	C^{-1} s^{-1-2\alpha} r_0^{-\frac{n+1}{2}+2\alpha}G(\tau),  \label{ball-formula}
	\end{equation}
where 
\begin{equation} \label{eq:G}
G(\tau)=\left((1 +\tau)^{\frac{n+1}{2}-2\alpha} -  (1-\tau)_+^{\frac{n+1}{2}-2\alpha} \right)^{-1} f(\tau^{-1}).
\end{equation}
Moreover, $f$ is a continuous function with $f(0) >0$ and   $f(\tau^{-1}) = \frac{4}{\pi} \tau^{2\alpha} [1+ O(\tau)]$ when $\tau \rightarrow 0$. \\
When $r_0 =0$, 	\begin{equation} \meas (B_s(x)) = \frac{C}{f(0)} \,  s^{\frac{n+3}{2}}. \end{equation}
\end{prop}

\subsection{ Heat Trace Integral }
Using the estimates from the above two subsections,  we can compute the integral of the heat kernel on any rectangular coordinate region $\{(r,v) | 0 \le r_1\le r \le r_2 \leq \infty,\  -\infty <  v_1 \le v \le v_2 < \infty\} \subseteq Y$. 

Recall that, for $x=(r, v)$ (we also say that $r=r(x)$), 
\[\Heat (x,x,s^2) = \frac{1}{\meas (B_s(x))}  q(x,x,s) = \frac{1}{\meas (B_s(x))} h(r,s).  \]
Set $\tau = \frac{s}{r}$. 
Using \eqref{ball-formula} and Lemma \ref{lem:h} (also denoting $C'= \frac{n+1-4\alpha}{4}$)
\begin{eqnarray}
\lefteqn{	\int_{v_1}^{v_2} \int_{r_1}^{r_2} \Heat (x,x,s^2) d\meas} \\  & = & C' s^{-1-2\alpha} \int_{v_1}^{v_2} \int_{r_1}^{r_2}  h(r,s) r^{-\frac{n+1}{2}+2\alpha}G(\tfrac{s}{r}) r^{\frac{n-1}{2}-2\alpha} dr dv  \nonumber \\
	& = & C' s^{-1-2\alpha} (v_2-v_1) \int_{r_1}^{r_2}  h(1,\tfrac{s}{r}) r^{-1}G(\tfrac{s}{r})  dr \nonumber \\
	& = &C'  s^{-1-2\alpha} (v_2-v_1) \int_{s/r_2}^{s/r_1}  h(1,\tau) \tau^{-1}G(\tau)  d\tau. \label{eq:Heat Integral}
\end{eqnarray}
Since $h$ is bounded (Lemma \ref{lem:h}) and $n-1 > 4 \alpha$, by \eqref{eq:G}, the integral is convergent at $\tau = \infty$. When $\tau \rightarrow 0$, since $f(\tau^{-1}) \sim \tau^{2\alpha}$, the integrand is $\sim \tau^{2\alpha -2}$. Thus the integral is convergent at $\tau=0$ when $\alpha > \frac 12$; diverges when $\alpha \le \frac 12$. Therefore, by taking various $r_1, r_2$, we obtain
\begin{thm}  \label{int-heat-Y}
	For $\alpha > \frac 12$, $k = 2\alpha + 1$, 
	\begin{eqnarray*}
0 < 	\lim_{s \rightarrow 0}	s^k	\int_{v_1}^{v_2} \int_{0}^{\infty} \Heat (x,x,s^2) d\meas  < \infty, \\
	\lim_{s \rightarrow 0}	s^k	\int_{v_1}^{v_2} \int_{r_1}^{\infty} \Heat (x,x,s^2) d\meas  =  0 \ \ \mbox{for any} \ \ r_1 >0, \\
 \lim_{L \rightarrow \infty}	\lim_{s \rightarrow 0}	s^k	\int_{v_1}^{v_2} \int_{sL}^{\infty} \Heat (x,x,s^2) d\meas =  0. 
	\end{eqnarray*}
 % <r_2 < \infty.
For $\alpha = \frac 12, \ k > 2$, 
\begin{eqnarray*}
		\lim_{s \rightarrow 0}	s^k	\int_{v_1}^{v_2} \int_{0}^{r_2} \Heat (x,x,s^2) d\meas  = 0.% \ \ \mbox{for any}\ \ 0 < r_2 < \infty. 
\end{eqnarray*}
For $0< \alpha < \frac 12, \ k = 2$, 
\begin{eqnarray*}
0 < 	\lim_{s \rightarrow 0}	s^2	\int_{v_1}^{v_2} \int_{0}^{r_2} \Heat (x,x,s^2) d\meas  < \infty.  %\ \mbox{for any}\ \ 0 < r_2 < \infty. 
\end{eqnarray*}
\end{thm}

When $\alpha < \frac 12,\  k =2$, it is in the setting studied in \cite[Theorem 4.3]{AmbrosioHondaTewodrose}, and the  condition \cite[(4.10)]{AmbrosioHondaTewodrose} (namely (\ref{eq:110}))  is satisfied.  

For the critical case $\alpha = \frac 12$, when $r_1=0$, we can show the integral in \eqref{eq:Heat Integral} diverges slowly like $\log$.

\begin{lem}
	\label{int-heat-Y-1/2} For any fixed $r_2 >0$,  let $\tilde{L}(s) = \int_{s/r_2}^{\infty}  h(1,\tau) \tau^{-1}G(\tau)  d\tau$. Then for $\alpha = \frac 12$, $s\in (0, s_0]$ and $s_0$ finite, 
	\[ \tilde{L} (s) =   \frac{1}{(n-1) \pi} (-\log s + \log r_2) +\tilde{C}_0 +O(s)  \]
	for some constant $\tilde{C}_0$ given explicitly in the proof. 	
 In particular 
	 $\tilde{L}$ is slowly varying as $s \rightarrow 0^+$, i.e. for any $a>0$ fixed,  \[
	\lim_{s \rightarrow 0^+} \frac{\tilde{L}(a s)}{\tilde{L}(s)} =1. \]
\end{lem}  
\begin{proof} By Lemma~\ref{lem:h}, $ h(1, \tau)$ is bounded between two positive constant.
% and \[
%	 h(\tfrac{1}{\tau},1) = h(1, \tau). \]
	 Since $(1,0) \in Y$ is a regular point with tangent cone $\mathbb R^2$, in fact a smooth point contained in a smooth neighborhood a definite distance away from the singular set, by the standard small time asymptotic expansion of the heat kernel, for small $\tau$,
	 %by the heat kernel convergence with respect to pointed Gromov-Hausdorff limit (see e.g. \cite[Theorem 3.3]{AmbrosioHondaTewodrose}) we have 
	 \[
	  h(1, \tau) = \frac 14 + O(\tau).
	 \]
	 When $\alpha = \frac 12$, by Proposition~\ref{prop:vol}   $G(\tau)$ is bounded on $(0, \infty)$ and, for small $\tau$,
	 \[   G(\tau) = \frac{4}{\pi} \cdot \frac{1}{n-1} + O(\tau). 
	 \] 
Recall the integral $\int h(1, \tau) \tau^{-1}G(\tau)  d\tau $ is always convergent at infinity. 	Denote $C_0 = \int_{1}^{\infty}  h(1, \tau) \tau^{-1}G(\tau)  d\tau$. Then  
	\begin{eqnarray*} \tilde{L}(s) & = &  \int_{s/r_2}^{1}  h(1, \tau) \tau^{-1}G(\tau)  d\tau +  C_0 \\
		& =  &  \int_{s/r_2}^{1} \tau^{-1} \left( h(1, \tau) G(\tau) - \frac{1}{(n-1) \pi} \right)  d\tau +  \int_{s/r_2}^{1}  \frac{\tau^{-1}}{(n-1) \pi}  d\tau  + C_0 \\
		& = & \frac{-1}{(n-1) \pi} (\log (s/r_2)) +C_0 + \int_{s/r_2}^{1} \tau^{-1} \left( h(1, \tau) G(\tau) - \frac{1}{(n-1) \pi} \right)  d\tau.
	\end{eqnarray*}
The last integral is 
\begin{eqnarray*} \int_{0}^{1} \tau^{-1} \left(h(1, \tau) G(\tau) - \frac{1}{(n-1) \pi} \right)  d\tau
	-\int_0^{s/r_2} \tau^{-1} \left( h(1, \tau) G(\tau) - \frac{1}{(n-1) \pi} \right)  d\tau.
	\end{eqnarray*}	
Since $h(1, \tau) G(\tau) - \frac{1}{(n-1) \pi}=O(\tau)$, the first integral is convergent at $\tau=0$ and the second one is $O(s)$.  
	\end{proof}

As a consequence, for the critical case, 
\begin{thm}  \label{int-heat-Y:crit}
For $\alpha = \frac 12, \ k =2$, and $s\in (0, s_0]$, $s_0$ finite, 
$$
 	s^2	\int_{v_1}^{v_2} \int_{0}^{r_2} \Heat (x,x,s^2) d\meas  =  \frac{v_2-v_1}{4 \pi} (-\log s + \log r_2) +\tilde{C} +O(s),
$$
where $\tilde{C}= \tfrac{(n-1)(v_2-v_1)}{4}\tilde{C}_0$.
And
$$	0 < \lim_{s \rightarrow 0}	s^2	\int_{v_1}^{v_2} \int_{r_1}^{r_2} \Heat (x,x,s^2) d\meas %\le &	\limsup_{s \rightarrow 0}	s^2	\int_0^{v_0} \int_{r_1}^{r_2} \Heat (x,x,s^2) d\meas  
< \infty  \ \ \mbox{for any} \ r_1>0. $$
\end{thm}

%\begin{thm}  \label{int-heat-Y-1/2} For any fixed $r_2 >0$,  let $\tilde{L}(s) = \int_{s/r_2}^{\infty}  h(\tfrac{1}{\tau},1) \tau^{-1}G(\tau)  d\tau$, then for $\alpha = \frac 12$,  $\tilde{L}$ is slowing varing as $s \rightarrow 0^+$, i.e. for any $a>0$ fixed,  \[
%	\lim_{s \rightarrow 0^+} \frac{\tilde{L}(a s)}{\tilde{L}(s)} =1. \]
%And  \begin{equation} C_1^{-1} (-\log s + \log r_2) +C_0 \le \tilde{L}(s) \le C_1 (-\log s + \log r_2) +C_0   \label{L-tilde}\end{equation}	
%for some positive constant $C_1, C_0$. 	
%\end{thm}
%\begin{proof}  Recall that $h$ is uniformly bounded by the Li-Yau estimate.  For $\alpha = \frac 12$, $G$ is also bounded by \eqref{eq:G}. We have
%	 \begin{equation} C_1^{-1} \le h(\tfrac{1}{\tau},1) G(\tau)  \le C_1 \ \ \ \ \forall \tau \in (0, \infty),  \label{C1}
%	\end{equation} with some constant $C_1(n)>0$. 
%Therefore	\[ |\tilde{L}(a s) -\tilde{L}(s)| =  \left|\int_{as/r_2}^{s/r_2} h(\tfrac{1}{\tau},1) \tau^{-1}G(\tau)  d\tau \right| \le C_1 \left|\int_{as/r_2}^{s/r_2} \tau^{-1}  d\tau\right| = 
%C_1|\log a|. 	\]
% By \eqref{1/2=infty} $\tilde{L}(s) \rightarrow \infty$ as $s \rightarrow 0^+$. Thus,
%\[ \lim_{s \rightarrow 0^+} \frac{\tilde{L}(a s) -\tilde{L}(s)}{\tilde{L}(s)} = 0, 
%\]	
%which is the first equation. 
%
%Write 
%\[ \tilde{L}(s) = \int_{s/r_2}^{1}  h(\tfrac{1}{\tau},1) \tau^{-1}G(\tau)  d\tau + \int_{1}^{\infty}  h(\tfrac{1}{\tau},1) \tau^{-1}G(\tau)  d\tau.  
%\]
%Denote $C_0 = \int_{1}^{\infty}  h(\tfrac{1}{\tau},1) \tau^{-1}G(\tau)  d\tau$ and use \eqref{C1} for the first integral gives \eqref{L-tilde}. 
%\end{proof}

\begin{rem}\label{remarkweight} Similarly in above theorems we can show that for any continuous function $\varphi(x)$ with compact support, the same results hold for 
$$
	\lim_{s \rightarrow 0}	s^k	\int\varphi(x) \Heat (x,x,s^2) d\meas. 
$$
For the statements about the nonzero limit, one needs to make sure that the support of $\varphi$ contains the singular set $\mathcal S= \{ r=0\}$.
\end{rem}

\section{Heat Kernel and Weyl's Law of $X$}
\subsection{Heat Kernel and Covering Spaces}\label{subsec:weyl}
Let $\bar{Y} = Y/\mathbb Z$ with the quotient distance $\bar \dist$ and the measure $\bar \meas$, and $\bar H (x,y,t)$ be its heat kernel, where the $\mathbb{Z}$-action is generated by the translation isomorphism $\gamma: Y \rightarrow Y, \ (r,v) \mapsto (r, v+1)$. This is also an $\RCD(0, n+1)$ space (see \cite[Remark 1.8]{PanWei}). In this subsection we study the heat kernel $\bar H$ of $\bar Y$. %It is worth mentioning that it is easy to see from Theorem \ref{dist} (2) and Lemma \ref{measure density} that for both $Y$ and $\bar Y$, the canonical inclusions of $H^{1,2}$ into $L^2$ are not compact operators. 

First we prove the following relation of heat kernels for covering spaces of $\RCD$ spaces which is of independent interest.  %, and $F$ a fundamental domain of $X$ in $\tilde X$ with $(\pi|_F)_{\sharp}\tilde \meas =\meas$, where we denote $\pi$ the projection from $\tilde X$ to $X$.  Then
\begin{prop}\label{prop:heat-cover} Let $(X, \dist, \meas)$ be an $\RCD(K,N)$ space for some $K \in \mathbb{R}$ and some $N \in [1, \infty)$, $\tilde X$ a connected covering space with lifted metric $\tilde \dist$ and measure $\tilde \meas$ and deck transformation $\Gamma$ (then $(\widetilde X, \tilde \dist, \tilde \meas)$ is also an $\RCD(K, N)$ space because of the same reason as in \cite[Lemma 2.8]{MW}). 
Assume that there exists a fundamental domain $D$ such that $\tilde \meas (\partial D)=0$. Then for all $\tilde x,\ \tilde y \in \tilde X$, and $x = \pi(\tilde x), \ y = \pi(\tilde y)$, we have
	\begin{equation} \Heat_X (x,y,t) = \sum_{\gamma \in \Gamma} \Heat_{\widetilde X}(\tilde x, \gamma \tilde y, t). \label{eq:heat-covering} \end{equation}
\end{prop}
We call $D$ as above a metric measure fundamental domain (Definition \ref{mmdomain}). %Moreover although we prove the existence of a fundamental domain in this framework (Proposition \ref{D_domain}), we do not know whether a metric measure fundamental domain exists or not in general. 
See also Remarks \ref{stbr} and \ref{Ye}.

When $(X, \dist, \meas)$ is an unweighted Riemnnian manifold with Ricci curvature bounded from below, Proposition \ref{prop:heat-cover} is proved in \cite[Pages 182-186]{L} (see also \cite[Proposition 2.12]{B}). The same strategy works in the setting of Proposition \ref{prop:heat-cover}, but  
it is more technical. We write the proof in the appendix \S \ref{appp}. From now on we will abuse notation and simply denote $[x]$ by $x$. 

Clearly  our $\bar Y$ has a fundamental domain $D=(0, \infty) \times (0,1)$ with $\meas (\partial D)=0$, thus we can apply Proposition~\ref{prop:heat-cover}  to $\bar Y$. Combining 
%With Proposition~\ref{prop:heat-cover} and 
with Lemma~\ref{r-distance},  we can then quickly get 
\begin{prop}
	 For $x=(r, v)\in \bar Y$, as $s \rightarrow 0$,  \[
	\bar \meas (B^{\bar Y}_s(x)) \bar \Heat (x,x,s^2) = h(r,s) + O(e^{-\frac{C(r_0)}{6s^2} }). \] Moreover
 \begin{equation}  \label{equ:heatYbar}
		\lim_{s \rightarrow 0} s^k \int_0^{1} \int_0^{r_0}\bar \Heat (x,x,s^2) d \bar \meas  =  \lim_{s \rightarrow 0} s^k \int_0^{1} \int_0^{r_0} \Heat (x,x,s^2) d \meas \end{equation}
	for any $0<r_0<\infty$ and all cases of $\alpha, k$ in Theorem~\ref{int-heat-Y}. 
\end{prop} 
\begin{proof}
When $s < < 1,\ \bar \meas (B^{\bar Y}_s(x)) = \meas (B^{Y}_s(x))$, therefore \begin{eqnarray*}
\bar \meas (B^{\bar Y}_s(x)) \bar H(x,x,s^2) & = & \meas (B^{Y}_s(x)) \left(\Heat (x,x,s^2) + \sum_{l \in \mathbb Z, l \not= 0} \Heat (x, \gamma^l x, s^2)\right)  \\
& = & q(x,x,s) + \sum_{l \in \mathbb Z, l \not= 0} q(x, \gamma^l x,s).
\end{eqnarray*}
By the Li-Yau estimate \eqref{eq:gauusian} \[
q(x, \gamma^l x,s) \le C(n) \exp (-\tfrac{\bar \dist^2(x,\  \gamma^l x)}{6s^2}). \]
By Lemma~\ref{r-distance}, $$ \dist(x, \  \gamma^l x) \ge C(r_0) |l|^{\frac{1}{1+2\alpha}}$$ for all $x = (r,v)$ with $0\le r \le r_0$ and all $l \in \mathbb Z$.  Therefore 
\[ \sum_{l \in \mathbb Z, l \not= 0} q(x, \gamma^l x,s) \le C(p, \alpha) e^{-\frac{C(r_0)}{6s^2} }  \ \ \mbox{for} \ \ s \in (0,1) \]
is exponentially decaying as $s \rightarrow 0$. This proves the first equation. 

The second equation follows from the first since 
 $\frac{1}{\bar \meas (B^{\bar Y}_s(x))}$ grows at most polynomially of degree $\tfrac{n+3}{2}$. 
\end{proof}

\subsection{Eigenvalues of $\bar Y$ when $\alpha = \frac 12$}
We use separation of variables (see e.g. \cite[Page 41]{Chavel})  to study the eigenvalues of $\bar Y = [0, \infty) \times \mathbb{S}^1$. First we do the computation with general $\alpha$. Recall $\bar Y$ equipped with the metric $g = dr^2 + r^{-4\alpha} d\theta^2$ (Theorem \ref{dist}), the weighted measure $d\bar \meas = c r^{\frac{n-1}{2} -2\alpha}drd\theta$ (Lemma \ref{measure density}). 
Therefore the weighted Laplacian is $$\frac{\partial^2}{\partial r^2} + \frac{1}{r}\left(\frac{n-1}{2} -2\alpha\right) \frac{\partial}{\partial r} + r^{4\alpha} \frac{\partial^2}{\partial \theta^2}.$$
Write the eigenfunction $u = \varphi (r) \psi (\theta)$, then we have 
\[\psi'' = -k^2 \psi, \ \ \ k \in \mathbb Z\]
and \[
\varphi'' (r) + \left(\tfrac{n-1}{2} -2\alpha\right) \frac 1r \varphi'(r) - k^2  r^{4\alpha} \varphi = -\lambda \varphi.  \]
The change of variable $y(r) = \varphi(r) r^{\frac{n-1}{4} -\alpha}$ gives \[
y''(r) = \left[ \left(\tfrac{n-1}{4} -\alpha\right) \left(\tfrac{n-1}{4} -\alpha -1\right) \frac{1}{r^2} +k^2  r^{4\alpha} -\lambda \right] y. \]
When $n=1$ or $9$ and $\alpha = \frac 12$, this becomes \[ y''(r) = \left[ \tfrac 34  \tfrac{1}{r^2} +k^2  r^{2} -\lambda \right] y.
\]
For $n=1$, it corresponds to the  Grushin half cylinder with unweighted measure
 studied in \cite{Boscain-PS2016}. Surprisingly for $n=9$ it gives the same constant $\tfrac 34$ in above.  For $k \not= 0$,  with the change of variable $z = |k| r^2$, and divide by $4k^2r^2$, we have \footnote{In \cite{Boscain-PS2016} the sentence after equation (3.2) is not exactly correct; we would like to thank Prof. Dario Prandi for communicating this and explaining the correction.} 
\[ \frac{\partial^2y}{\partial z^2} + \frac{z}{2} \frac{\partial y}{\partial z}  = \left(\frac{3}{16z^2} + \frac 14 -\frac{\lambda}{4|k|z} \right) y. 
\]
Letting $\phi (z) = z^{\frac 14} y(z)$ gives 
\[
\phi'' = \left(\frac 14 -\frac{\lambda}{4|k|z} \right) \phi,  \]
which is the Whittaker equation.  Following \cite{Boscain-PS2016}, the eigenvalues are $\lambda_{n, k} = 4 |k| n$ for $n \in \mathbb N$, $k \in \mathbb Z/\{0\}$ ($k=0$ leads to the continuous spectrum).
\begin{rem}\label{bakryemery}
If $(\bar Y, \bar \dist, \bar \meas)$ is an $\RCD(K, N)$ space, then the $N$-Bakry-\`Emery Ricci curvature on the regular set of $\bar Y$ is bounded below by $K$. For our $\bar Y$  with $\alpha = \frac 12$, the $N$-Bakry-\`Emery Ricci curvature is nonnegative if and only if $N \ge 10$. In particular the optimal dimension to have nonnegative Ricci curvature in the RCD sense is indeed $10$, i.e., $n=9$. The same observation is valid for $Y$.

Therefore this allows us to conclude that the optimal dimension to have a lower bound in the RCD sense for our $2$-(Hausdorff) dimensional spaces is also equal to $10$ because of the following two facts;
\begin{itemize}
\item any tangent cone at any point of any $\RCD(K, N)$ space is an $\RCD(0, N)$ space;
\item any tangent cone at any singular point of such spaces is isometric to $Y$.
\end{itemize}

\end{rem}

\subsection{Heat Kernel of Compact Space $X$}  \label{heat-X}
We consider the compact space $X$ obtained by perturbing the warping function of $\bar Y$ to $\tilde h = constant$ for $r\ge 2$, and then double. Precisely let $\widetilde Y$ be the finite cylinder $[0,3] \times \mathbb{S}^1$ with the singular  metric $dr^2 + \tilde h(r)^2 dv^2$, where 
\[ \tilde h(r) = \left\{ \begin{array}{ll} r^{-2\alpha} &  0<r \le 1 \\
		\tfrac 12 & 2 \le r \le 3 \end{array} \right. \] 
	and connect smoothly between $r=1$ and $r=2$ with decreasing and convex function. Correspondingly, the measure $d\tilde \meas= {\tilde h}^{-\frac{n-1}{4\alpha}+1} dr\, dv$. Let $X = \widetilde Y \cup  \widetilde Y$ by gluing at the smooth end $r=3$. Then $\widetilde Y$ and $X$ are $\RCD (K, n+1)$ spaces with $K$ some negative number. 

\begin{rem} 
	The above $X$ can be constructed as the limit of a sequence of closed manifolds $N_i$ with $\mathrm{Ric}\ge K$, where $K<0$ (see \cite[Remark 1.8]{PanWei}); then $\widetilde{Y}\subseteq X$ is the limit of a sequence of open convex subsets in $N_i$. We also point out that nonnegative Ricci curvature cannot hold on $N_i$; otherwise, the universal cover $\widetilde{N_i}$ would split off a line isometrically, resulting in $X$ being isometric to a standard bounded cylinder. 
\end{rem}

%Since $\tilde Y$ is convex in $X$, it is an $\RCD(K, n+1)$ (see \cite[Proposition 7.7]{AMS}).
 Since $\widetilde Y$ is a smooth weighted Riemannian manifold near the end $r=3$, each eigenfunction on $\widetilde Y$ is smooth near the smooth end and it satisfies the Neumann boundary value condition because of the elliptic regularity theorem. 
In particular the heat kernel $\Heat_{\widetilde Y}$ of $\widetilde Y$ is smooth near the smooth end with the same boundary value condition by (\ref{eq:heat exp}).

	By the symmetry of $X$, the heat kernel $\Heat_X$ of $X$ can be identified with the heat kernel of $\widetilde Y$ with the Neumann boundary condition at the smooth end. Now the heat kernel of $\widetilde Y$ can be constructed by gluing the heat kernel of $\bar Y$,  with the Neumann heat kernel of cylinder $\hat Y = [1/2,3] \times \mathbb{S}^1$ as a smooth weighted manifold. The warping function $\hat h$ of $\hat Y$ is defined as  $\hat h = \tilde h$ for $r \in [\tfrac 12, 3]$ and the measure $d\hat \meas= {\hat h}^{-\frac{n-1}{4\alpha}+1} dr\, dv$. 
	
	Following \cite{APS}, let $\rho(a,b) \in C^\infty (\mathbb R)$ be a smooth increasing cut-off function such that $\rho \equiv 0$ when $r\le a$ and $\rho \equiv 1$ when $r\ge b$. 
	
	Define $\phi_1, \phi_2, \psi_1, \psi_2$ as follows. 
	$$\psi_1 = 1-\psi_2, \quad \psi_2 = \rho (1+\tfrac 14, 2- \tfrac 14), \quad \phi_1 =1 - \rho(2,3), \quad \phi_2 = \rho(\tfrac 12, 1).$$ Then 
	$$\psi_1 + \psi_2 \equiv 1,\quad \phi_i \equiv 1 \text{ on supp}\, \psi_i,\quad \dist_{\mathbb R} ({\rm supp}\, \phi_i', {\rm supp}\, \psi_i) \ge \tfrac 14.$$ 
	It follows that
\begin{equation}\label{eq:heat distance}
\tilde \dist ({\rm supp}\, \phi_i', {\rm supp}\, \psi_i) \ge \tfrac 14.
\end{equation}
Here $\tilde \dist$ is the distance of $\tilde Y$.
	
	Let $$\HeatK(x,y,t) = \phi_1(r(x)) \Heat_{\bar Y}(x,y,t) \psi_1 (r(y)) + \phi_2(r(x)) \Heat_{\hat Y}(x,y,t) \psi_2 (r(y)),$$  where $H_{\hat Y}(x,y,t)$ is the heat kernel of the smooth weighted manifold $\hat Y$ with Neumann boundary condition. Then $\HeatK$ is a parametrix of the heat equation on $\tilde Y$. Indeed
\begin{equation}\label{eq:heat parametrix}
(\partial_t -\Delta_x)\HeatK (x, y, t)=-\phi_1'' \Heat_{\bar Y}\psi_1 - \langle \nabla \phi_1, \nabla \Heat_{\bar Y}\rangle \psi_1 -\phi_2'' \Heat_{\widetilde Y}\psi_2 - \langle \nabla \phi_2, \nabla \Heat_{\widetilde Y}\rangle \psi_2,
\end{equation}
where $\Delta_x$ means the Laplacian acting on the $x$-variable. Similarly the gradient $\nabla$ is with respect to the $x$-variable as well. Note that by our choice of the cut-off functions, the right hand side of \eqref{eq:heat parametrix} is away from the singular set in the $x$-variable, and, by the regularity theory, the differentiations can be taken in the usual sense. Let $$\HeatQ (x, y, t)=(\partial_t -\Delta_x)\HeatK (x, y, t).$$ 
 By \eqref{eq:heat parametrix}, $\HeatQ(x, y, t)$ is nonzero only when $r(x) \in {\rm supp}\, \phi_1'$ and $r(y) \in {\rm supp}\,\psi_1$ or when $r(x) \in {\rm supp}\, \phi_2'$ and $r(y) \in {\rm supp}\,\psi_2$. By the Li-Yau estimates \eqref{eq:gauusian}, \eqref{eq:grad gaus} and \eqref{eq:heat distance}, we deduce for $0\leq s <t$ and  $r(z) \in {\rm supp}\, \phi_1'\cup {\rm supp}\, \phi_2' \subset [1/2, 1] \cup [2, 3]$ ,
\begin{eqnarray}\label{estimateq}
| \HeatQ(z, y, t-s)| & \leq  & \frac{C_1}{(t-s)\meas (B_{\sqrt{t-s}}(z))}\exp\left( -\frac{1}{90(t-s)}+C_2(t-s)\right) \nonumber \\
& \leq & \frac{C_3}{(t-s)}\exp\left( -\frac{1}{90(t-s)}+C_2(t-s)\right) \nonumber \\
& \leq & C_4\exp\left( -\frac{1}{180(t-s)}+C_2(t-s)\right) \leq C_5 \exp\left( -\frac{1}{180t}\right)
\end{eqnarray}
for $0\leq s <t\leq 1$.
Duhamel Principle expresses the heat kernel $\Heat_{\tilde Y}(x,y,t)$ in terms of the parametrix $\HeatK (x, y, t)$ (i.e. the approximate solution) and an error term.

\begin{lem}[Duhamel Principle]\label{Duhamelpri}
We have 
\begin{equation}\label{eq:heat duhamel}
\Heat_{\widetilde Y}(x,y,t)=\HeatK(x,y,t) - \int_0^t \int_{\widetilde Y} \Heat_{\tilde Y}(x,z,s) \HeatQ(z, y, t-s) d\tilde \meas(z) ds.
\end{equation} 
\end{lem}
This is well known in the smooth case; see \cite{Cheeger1} for the most general formulation. The proof of the lemma goes by directly checking the heat equation. Namely, it follows from a direct calculation that the right hand side of (\ref{eq:heat duhamel}) 
solves the heat equation. Then the lemma follows by a similar argument as in  Proposition \ref{prop:heat-cover}. See \S \ref{duh} for the detail.

Plugging (\ref{estimateq}) into \eqref{eq:heat duhamel}, we obtain
\begin{equation}\label{eq:heat gluing}
\Heat_{\widetilde Y}(x,y,t)=\HeatK (x,y,t) + O(e^{-C/t})
\end{equation}
for $C=\frac{1}{180}$. From here, we obtain by using Theorems \ref{int-heat-Y} and \ref{int-heat-Y:crit}

\begin{thm}  \label{int-heat-X}
	For $\alpha > \frac 12$, $k = 2\alpha + 1$, the following limit exists and
	\begin{equation}
0 < 	\lim_{s \rightarrow 0}	s^k	\int_X \Heat_X(x,x,s^2) d\meas  = 2 \lim_{s \rightarrow 0}	s^k	\int_0^1 \int_0^3 \Heat_Y(x,x,s^2) d\meas < \infty.  \label{eq:heatX-Y}
	\end{equation}
Here $v$ is from $0$ to $1$, $r$ is from $0$ to $3$, and $3$ can be replaced by any positive number.

For $\alpha = \frac 12, \ k =2$, and $s$ small,
\begin{equation}
	 		s^2	\int_X \Heat_X(x,x,s^2) d\meas = 2	s^2		\int_0^1 \int_0^3 \Heat_Y(x,x,s^2) d\meas = \frac{1}{2\pi}(-\log s) + O(1). \label{heatX-1/2}
\end{equation}
For $\alpha = \frac 12, \ k > 2$, 
\begin{eqnarray*}
		\lim_{s \rightarrow 0}	s^k	\int_X \Heat_X(x,x,s^2) d\meas  = 0.
\end{eqnarray*}
For $0< \alpha < \frac 12, \ k = 2$, 
\begin{eqnarray*}
0 < 	\lim_{s \rightarrow 0}	s^2	\int_X \Heat_X(x,x,s^2) d\meas  < \infty .
\end{eqnarray*}

\end{thm}
	
\begin{proof}
By the symmetry of $X$ and \eqref{eq:heat gluing} we have
\begin{eqnarray}
	s^k	\int_X \Heat_X(x,x,s^2) d\meas & = & 2 s^k	\int_{\tilde Y} \Heat_{\tilde Y}(x,x,s^2) d\meas \nonumber \\
 &=& 2 s^k	\int_{\bar Y} \Heat_{\bar Y}(x,x,s^2)\psi_1(r(x)) d\meas \nonumber \\
 & & + 2 s^k	\int_{\hat Y} \Heat_{\hat Y}(x,x,s^2)\psi_2(r(x)) d\meas  + O(e^{-C'/s^2}) \nonumber \\
 &=& 2 s^k	\int_0^1 \int_0^3 \Heat_Y(x,x,s^2)\psi_1(r(x)) d\meas \label{eq:heatXapp}\\
 & & + 2 s^k	\int_{\hat Y} \Heat_{\hat Y}(x,x,s^2)\psi_2(r(x)) d\meas  + O(e^{-C'/s^2}), \nonumber
	\end{eqnarray}
where we have made use of \eqref{equ:heatYbar} in the last identity.
The first term can be dealt with by Theorems \ref{int-heat-Y} and \ref{int-heat-Y:crit} (more specifically Remark \ref{remarkweight}), while the second term is classical (or we can use \cite{AmbrosioHondaTewodrose}). The desired results follow.
\end{proof}	
		
		\subsection{Weyl's Law of $X$}\label{heat=X}
		
We are now in a position to introduce our main results, singular Weyl's laws on compact Ricci limit, thus $\RCD(K, N)$ spaces $(X, \dist, \meas)$ constructed in the previous subsection with the parameter $\alpha$.
In order to introduce a surprising result (Corollary \ref{cor:weyl sing}), we prepare the following.
\begin{thm}\label{thm:weyl}
Let $\alpha>\frac{1}{2}$ and $k=2\alpha+1$. Then  
\begin{equation}\label{eq:weak limit}
s^k\Heat (x,x,s^2)d\meas(x) \to c\cdot  1_{\mathcal{S}}\, d\mathcal{H}^k,\quad \text{weakly as $s \to 0^+$}
\end{equation}
for some constant $c>0$, where $\mathcal{S}=X \setminus \mathcal{R}_2$ is the singular set. Moreover the support of $1_{\mathcal{S}}\, d\mathcal{H}^k$ coincides with $\mathcal{S}$.
\end{thm}
\begin{proof} By Theorem \ref{int-heat-X}, 
	\[0 < \lim_{s \to 0^+} s^k	\int_X \Heat_X (x,x,s^2) d\meas <\infty.
	\]
Therefore for any sequence $s_i \to 0^+$, there is a subsequence such that the left hand side of \eqref{eq:weak limit}  converges to a Radon measure weakly. Take a sequence $s_i \to 0^+$ and a Radon measure $\nu$ on $X$ such that
	\begin{equation}
		(s_i)^k\Heat (x,x,(s_i)^2)d\meas(x) \to \nu,\quad \text{weakly as $i \to \infty$. }
	\end{equation}
%we will show $\nu_1 = \nu_2$. It is enough to show they have same measures on all boxes of $\tilde Y$ defined below as 
	 %$X$ is the double of $\tilde Y$. 
%Thanks to Theorem \ref{int-heat-Y} and observations in the previous subsection, we know that for any singular point $z \in \mathcal{S}$
%\begin{equation}\label{eq:exists}
%0<\lim_{s \to 0^+} s^k	\int_0^{v_0} \int_{0}^{r_1} \Heat (x,x,s^2) d\meas <\infty,\quad \text{for all $v_0>0$ and $r_1>0$,}
%\end{equation}
%where $z$ can be identified as $0$ in $\bar Y$.

For any $z = (r_0,v_0) \in \tilde Y$, denote $\mathrm{Box}_{(r_1, v_1)}(z) =\{ (r,v)| 0 \le r_0 < r <r_1 \le 3, \ 0 \le v_0 < v < v_1 \le 1\}.$
Then the weak convergence gives
\begin{equation}\label{eq:0001}
\liminf_{i \to \infty}\int_{\mathrm{Box}_{(r_1, v_1)}(z)}(s_i)^k\Heat (x,x,(s_i)^2)d\meas(x)\ge \nu (\mathrm{Box}_{(r_1, v_1)}(z))
\end{equation}
and 
\begin{equation}\label{eq:0002}
\limsup _{i \to \infty} \int_{\overline{\mathrm{Box}_{(r_1,v_1)}(z)}} (s_i)^k\Heat (x,x,(s_i)^2)d\meas(x)\le \nu(\overline{\mathrm{Box}_{(r_1, v_1)}(z)}). 
\end{equation}

On the other hand, by \eqref{eq:heatXapp} and \eqref{eq:Heat Integral}, we know
\begin{equation}\label{eq:0003}
\lim_{\epsilon \to 0^+}\left( \limsup _{s \to 0^+}\int_{\mathrm{Box}_{(r_1, v_1)}(z) \setminus \mathrm{Box}_{(r_1, v_1-\epsilon)}(z) }s^k\Heat (x,x,s^2)d\meas(x)\right)=0.
\end{equation}
Combining (\ref{eq:0001}), (\ref{eq:0002}) with (\ref{eq:0003}) shows
\begin{equation}\label{eq:box conv}
\lim_{i \to \infty}\int_{\mathrm{Box}_{(r_1, v_1)}(z)}(s_i)^k\Heat (x,x,(s_i)^2)d\meas(x)=\nu (\mathrm{Box}_{(r_1, v_1)}(z)),
\end{equation}
and
\begin{equation}
\nu (\overline{\mathrm{Box}_{(r_1, v_1)}(z)})=\nu (\mathrm{Box}_{(r_1,v_1)}(z)).
\end{equation}
%Thus {\color{blue}Theorems \ref{int-heat-Y} and \ref{int-heat-Y:crit} yield}%(\ref{eq:exists}) yields
%\begin{equation}\label{eq:li}
%\nu_1(\mathrm{Box}_{(r_1,v_1)}(z))=\nu_2(\mathrm{Box}_{(r_1,v_1)}(z)),
%\end{equation}
%which implies
%$\nu_1=\nu_2$. Combining this observation with the weak compactness of measures,
%we have 
%\begin{equation}\label{eq:weak limit-nu}
%s^k\Heat(x,x,s^2)d\meas(x) \to \nu \quad \text{weakly as $s \to 0^+$}
%\end{equation}
%for some Borel measure $\nu$ on $X$.

Next let us prove that the support of $\nu$ coincides with $\mathcal{S}$.
Theorem \ref{int-heat-Y} and \eqref{eq:heatXapp} allow us to conclude
that for any compact subset $A$ in $\mathcal{R}$, 
\begin{equation}
\int_As^k\Heat (x,x,s^2)d\meas \to 0,\quad \text{as $s \to 0^+$.}
\end{equation}
This implies that the support of $\nu$ is included in $\mathcal{S}$. Theorem \ref{int-heat-Y} and \eqref{eq:heatXapp} also imply $\nu (\mathrm{Box}_{(r_1,v_1)}(z)) >0$ whenever $z \in \mathcal{S}$. Hence the support of $\nu$ coincides with $\mathcal{S}$.

Moreover  Theorem~\ref{dist} (4) with (\ref{eq:Heat Integral}) and (\ref{eq:box conv}) shows that there exist $C>0$ and $\delta>0$ such that $\nu (B_r(z))=Cr^k$ holds for all $z \in \mathcal{S}$ and $0<r<\delta$. Combining this with the observation above easily implies $\nu=c1_{\mathcal{S}}d\mathcal{H}^k$ for some $c>0$. Finally since we have
\[
c=\mathcal{H}^k(\mathcal{S})^{-1} \cdot \lim_{s \to 0^+} s^k	\int_X \Heat_X (x,x,s^2) d\meas\]
which does not depend on the sequence $s_i \to 0^+$, we conclude because $s_i$ is arbitrary.

\end{proof}

We are now in a position to introduce a main result of this subsection. Let us emphasize that $k$ is not necessarily an integer.
\begin{cor}\label{cor:weyl sing}
Let $\alpha>\frac{1}{2}$ and $k=2\alpha+1$. Then  
\begin{equation}
\lim_{\lambda \to \infty}\frac{N(\lambda)}{\lambda^{k/2}}=\frac{c}{\Gamma (k+1)}\mathcal{H}^k(\mathcal{S}) \in (0, \infty),
\end{equation}
where $c$ is given in (\ref{eq:weak limit}).
\end{cor}
\begin{proof}
This is a direct consequence of Theorem \ref{thm:weyl} and \eqref{taubTer} with $L=1$
%with Tauberian theorem (see {\color{blue}\cite[Theorem 1 in page 443]{Feller} or} \cite[Theorem 5.3]{AmbrosioHondaTewodrose}).
\end{proof}
The following is another surprising result which is an immediate consequence of Theorem \ref{int-heat-X} with the same techniques as in the proof above. It %should be emphasized that this
 in particular shows that for general  compact $\RCD(K, N)$ spaces, the eigenvalue counting function $N(\lambda)$ defined in \eqref{N-lambda} does not satisfy the following asymptotic formula:
\begin{equation*}
N(\lambda) \sim \lambda^{\beta},\quad \text{for some $\beta \ge 0$ as $\lambda \to \infty$.}
\end{equation*}
In connection with (\ref{eq:009a}) below, we recall that any neighborhood of any singular point has infinite $2$-dimensional Hausdorff measure; see Remark \ref{rem:infinite}.
\begin{thm}\label{2weyl}
%Let $X$ be the $\RCD(K, n+1)$ space when 
Let $\alpha=\frac{1}{2}$. Then 
\begin{equation}\label{eq:009a}
\frac{s^2}{-\log s}\Heat (x,x, s^2)d\meas (x) \to c \cdot 1_{\mathcal{S}}\,d\mathcal{H}^2,\quad \text{weakly as $s \to 0^+$}
\end{equation}
for some constant $c>0$.
In particular
	\begin{equation} \label{eq:000001}
	\lim_{\lambda \to \infty}\frac{N(\lambda)}{\lambda \log \lambda} = \frac{c}{\Gamma (3)}\mathcal{H}^2(\mathcal{S})=\frac{1}{4 \pi}. 
	\end{equation}

% we have 
%\begin{equation}
%\limsup_{\lambda \to \infty}\frac{N(\lambda)}{\lambda^{\beta}}=\infty,\quad \text{for any $\beta \le 1$}
%\end{equation}
%and 
%\begin{equation*}
%\lim_{\lambda \to \infty}\frac{N(\lambda)}{\lambda^{\beta}}=0,\quad \text{for any $\beta > 1$}.
%\end{equation*}
\end{thm}
\begin{proof}

Applying a similar argument as in the proof of Theorem \ref{thm:weyl} with Theorem \ref{int-heat-Y:crit}  shows  (\ref{eq:009a}). 
%Thus Karamata's Tauberian theorem with a slowly varying function (see for instance \cite[Theorem 2 in page 445]{Feller}) allows us to prove 
Then the first equality in (\ref{eq:000001}) follows from \eqref{taubTer} with $L = \log$, and 
 the second equality comes from \eqref{heatX-1/2}. 
%again Theorem \ref{int-heat-Y:crit}.

%If 
%\begin{equation*}
%\limsup_{\lambda \to \infty}\frac{N(\lambda)}{\lambda}<\infty,
%\end{equation*}
%then \cite[(5.4)]{AmbrosioHondaTewodrose} shows 
%\begin{equation*}
%\limsup_{s \to 0^+}s^2\int_X\Heat (x,x,s^2)d\meas <\infty
%\end{equation*}
%which contradicts Theorem \ref{int-heat-X}. Thus we have (\ref{eq:000001}). Similarly we have the remaining statement by Theorem \ref{int-heat-X}. 
\end{proof}

\section{Asymptotics of eigenfunctions}
\subsection{Questions and negative answer}
Let us start by recalling a question raised by Ding in \cite[page 511]{Ding02}.
\begin{quest}[Ding]\label{q:ding}
For any integer $n \ge 2$, does there exist a positive small number $\epsilon_n>0$ such that 
\begin{equation}\label{eq:1112}
\int_{M^n \setminus A}f^2d\mathrm{vol}^g>\frac{1}{2}\int_{M^n}f^2d\mathrm{vol}^g
\end{equation}
holds for any $n$-dimensional closed Riemannian manifold $(M^n, g)$ with nonnegative Ricci curvature, any eigenfunction $f$ of $-\Delta^g$ on $M^n$ and any Borel subset $A \subset M^n$ with $\mathrm{vol}^gA <\epsilon_n\mathrm{vol}^gM^n$?
\end{quest}
It is easy to see from \cite[Theorem 1.1]{JZ} that the question has a negative answer even for the $n$-dimensional canonical unit sphere $\mathbb{S}^n$. Actually for any great circle $c$ in $\mathbb{S}^n$, we can find an $L^2$-orthonormal basis $\{f_j\}_j$ consisting of  $j$-th eigenfunctions $f_j$ of $\mathbb{S}^n$ and a subsequence $i(j)$ satisfying that 
\begin{equation}\label{eq:1111}
f_{i(j)}^2d\mathrm{vol}^g \to \nu_c, \quad \text{weakly},
\end{equation}
where $\nu_c$ denotes the canonical probability $1$-dimensional measure whose support coincides with $c$. Then it is easily checked that the sequence $f_{i(j)}$ provides a counterexample to Question \ref{q:ding}. Namely the question above is related to the quantum ergodicity (see also \cite{Z}). 

Since it is also well-known that it is really hard to determine all limit measures appeared as in (\ref{eq:1111}), let us switch the question above to a weaker one; 
\begin{quest}
How many eigenfunctions satisfy (\ref{eq:1112})? 
\end{quest}
This new question should be related to Weyl's law.

The main purpose of this section is to give positive answers to this new question for more general spaces, namely $\RCD(K, N)$ spaces. The obtained results later will be related to the validity of ``regular'' Weyl's law which is completely different from the results in the previous sections. Thus let us start the next subsection by introducing the regular Weyl's law for compact $\RCD(K, N)$ spaces.

\subsection{Regular Weyl's law and its consequence}\label{regw}
%Let us continue to discuss Question \ref{q:modified ding}. Although the question has a negative answer in general as shown in the previous subsection, there is still a possibility to realize this affirmatively after adding a further (geometric) assumption. The main purpose of this subsection is to provide partial positive answers along this direction, but they are stated for more general spaces with Ricci curvature bounded below, so-called $\RCD(K, N)$ spaces.

%Based on these preparations, let us provide a relationship between the validity of ``regular'' Weyl's law and a partial positive answer to Quesiton \ref{q:modified ding} in the RCD setting.
We start this subsection by giving the following elementary lemma.
\begin{lem}\label{lem:haus}
Let $(X, \dist, \meas)$ be an $\RCD(K, N)$ space for some $K \in \mathbb{R}$ and some $N \in [1, \infty)$. Then for any $l \ge 0$ and any Borel subset $A \subset X$, we have
\begin{equation}\label{eq:1114}
\mathcal{H}^l(A)\le C(N, l)\liminf_{s \to 0^+}\int_{B_{s}(A)}\frac{s^l}{\meas (B_s(x))}d\meas.
\end{equation}
\end{lem}
\begin{proof}
This is essentially the same as \cite[Lemma 7.11]{Ding02}.
Let us take a maximal $s$-separated set $\{x_i\}_i$ of $A$. Then
\begin{equation}\label{eq:1113}
\sum_is^l \le C(N)\sum_i\int_{B_{s}(x_i)}\frac{s^l}{\meas (B_{s}(x))}d \meas (x) \le C(N) \int_{B_{s}(A)} \frac{s^l}{\meas (B_{s}(x))}d \meas (x), 
\end{equation}
where in the first inequality in (\ref{eq:1113}), we used
\begin{equation}\label{eq:001}
\meas (B_{s}(x_i)) \ge c(N) \meas(B_{2s}(x_i)) \ge c(N) \meas (B_{s}(x))
\end{equation}
on $B_{s}(x_i)$. Note that the first inequality in (\ref{eq:001}) is a direct consequence of Bishop-Gromov inequality (for small $s>0$) and that the second inequality follows from the inclusion $B_s(x) \subset B_{2s}(x_i)$. 
Then
letting $s \to 0^+$ in (\ref{eq:1113}) implies (\ref{eq:1114}).
\end{proof}
\begin{rem}
The dependence on $l$ in the constant in the lemma above comes from the definition of the $l$-dimensional Hausdorff measure;
\begin{equation}
\mathcal{H}^l(A)=\lim_{\delta \to 0^+}\mathcal{H}^l_{\delta}(A),
\end{equation}
where 
\begin{equation}
\mathcal{H}^l_{\delta}(A)=\inf \left\{ \sum_ic_lr_i^l \Big| A \subset \bigcup_iB_{r_i}(x_i),\,\,r_i<\delta \right\}
\end{equation}
for some constant $c_l>0$. When $l$ is an integer, we always choose $c_l$ as $\omega_l=\mathcal{H}^l(B_1(0_l))$.
\end{rem}

Let us introduce the following result, where $\mathcal{R}_n^*$ denotes the reduced $n$-regular set of $(X, \dist, \meas)$ defined by the set of all $n$-dimensional regular points $x \in X$ with the existence of a finite positive limit:
\begin{equation}
\lim_{r\to 0^+}\frac{\meas(B_r(x))}{r^n} \in (0, \infty).
\end{equation}
Note  that $\meas (X \setminus \mathcal{R}_n^*)=0$ and that $\meas$ and $\mathcal{H}^n$ are absolutely continuous to each other on $\mathcal{R}_n^*$ (see \cite[Theorem 4.1]{AmbrosioHondaTewodrose}). 
\begin{thm}[Regular Weyl's law]  \label{R-weyl}
Let $(X, \dist, \meas)$ be a compact $\RCD(K, N)$ space for some $K \in \mathbb{R}$ and some $N \in [1, \infty)$ and let $n$ be the rectifiable dimension.
Then 
\begin{equation}\label{eq:110}
\lim_{r \to 0^+}\int_X\frac{r^n}{\meas (B_r(x))}d \meas=\int_X\lim_{r\to 0^+}\frac{r^n}{\meas (B_r(x))}d\meas <\infty
\end{equation}
holds if and only if Weyl's law is satisfied in a regular sense, namely
\begin{equation}\label{eq:weyl}
\lim_{\lambda \to \infty}\frac{N(\lambda)}{\lambda^{n/2}}=\frac{\omega_n}{(2\pi)^n}\mathcal{H}^n(\mathcal{R}_n^*)<\infty.
\end{equation}
%where $N(\lambda)=\sharp \{i \in \mathbb{N}| \lambda_i \le \lambda\}$ is the counting function of eigenvalues. 
Moreover if (\ref{eq:110}) is valid, then $\mathcal{H}^n(X)<\infty$, in particular the Hausdorff dimension of $X$ is equal to $n$.
\end{thm}
\begin{proof}
The desired equivalence is already obtained in \cite[Theorem 4.3]{AmbrosioHondaTewodrose}. The remaining one comes from Lemma \ref{lem:haus}
\end{proof}

In the sequel, we always fix an $L^2$-orthonormal basis $\{f_j\}_j$ consisting of $j$-th eigenfunctions $f_j$ of $L^2(X, \meas)$.
The following is a main result of this subsection, which also gives a Weyl type asymptotics.
\begin{thm}\label{prop:ding general}
Let $(X, \dist, \meas)$ be a compact $\RCD(K, N)$ space for some $K \in \mathbb{R}$ and some finite $N \in [1, \infty)$ whose rectifiable dimension is equal to $n$.
If (\ref{eq:110}) is satisfied, 
then for all $0<\delta<\epsilon<1$ and any Borel subset $A \subset X$ with
$\mathcal{H}^n(A\cap \mathcal{R}_n^*)\le \delta \mathcal{H}^n (\mathcal{R}_n^*)$, letting
\begin{equation}
m_{A, \epsilon}(\lambda) =\sharp \left\{j\in \mathbb{N} \,\Big| \,\int_{X \setminus A}f_j^2d\meas \ge (1-\epsilon)\int_Xf_j^2d\meas\,\,\text{and}\,\,\lambda_j \le \lambda \right\}\text{  for any $\lambda>0$},
\end{equation}
we have
\begin{equation}
m_{A, \epsilon}(\lambda)\sim \lambda^{n/2},\quad \text{as $\lambda \to \infty$.}
\end{equation}
In particular 
\begin{equation}
\int_{X \setminus A}f_j^2d\meas \ge (1-\epsilon)\int_Xf_j^2d\meas
\end{equation}
is satisfied for infinitely many eigenfunctions $f_j$ of $-\Delta$ on $X$.
\end{thm}
\begin{proof}
Since $m_{A, \epsilon}(\lambda) \le N(\lambda)$, thanks to the validity of (\ref{eq:weyl}) with \cite[Proposition 5.6]{AmbrosioHondaTewodrose}, it is enough to prove that
\begin{equation}
\liminf_{s\to 0^+}s^n\sum_{i \in \mathcal{M}}e^{-\lambda_is^2}>0,
\end{equation}
where 
\begin{equation}
\mathcal{M}=\left\{ j \in \mathbb{N} \,\Big| \,\int_{X \setminus A}f_j^2d\meas \ge (1-\epsilon)\int_Xf_j^2d\meas \right\}.
\end{equation}
%The proof of this is done by a contradiction. If it is not the case, then there exist a positive number $\epsilon_0>0$ and a sequence of Borel subsets $A_i \subset X$ with $\meas(A_i) \to 0$
%such that %letting 
%\begin{equation}
%m_i(\lambda)=\sharp \left\{ j \in \mathbb{N} \cap [0, \lambda] \,\Big| \,\int_{X \setminus A_i}f_j^2d\meas \ge (1-\epsilon_0)\int_Xf_j^2d\meas \right\},
%\end{equation}
%we have for any $i$
%\begin{equation}\label{eq:1100}
%\lim_{\lambda \to \infty}\frac{m_i(\lambda)}{\lambda^{n/2}}=0.
%\end{equation}
%Note that thanks to \cite{AmbrosioHondaTewodrose}, 
%letting
%\begin{equation}
%\mathcal{M}_i=\left\{ j \in \mathbb{N} \,\Big| \,\int_{X \setminus A_i}f_j^2d\meas \ge (1-\epsilon_0)\int_Xf_j^2d\meas \right\},
%\end{equation}
%we have 
%\begin{equation}\label{eq:1101}
%\liminf_{s\to 0^+}s^n\sum_{j \in \mathcal{M}_i}e^{-\lambda_js^2}= 0.
%\end{equation}
%Find a minimizing sequence $s_i \to 0^+$ of the LHS of (\ref{eq:1101}).
Note
\begin{align}
&\int_Xs^n\Heat (x,x,s^2)d\meas\nonumber \\
&=s^n\sum_{j \in \mathcal{M}}e^{-\lambda_js^2}+s^n\sum_{j \not \in \mathcal{M}}e^{-\lambda_js^2} \nonumber \\
&\ge s^n\sum_{j \in \mathcal{M}}e^{-\lambda_js^2}+\frac{s^n}{1-\epsilon}\sum_{j \not \in \mathcal{M}}e^{-\lambda_js^2}\int_{X\setminus A}f_j^2d\meas \nonumber \\
&\ge s^n\sum_{j \in \mathcal{M}}e^{-\lambda_js^2} +\frac{1}{1-\epsilon}\int_{X \setminus A}s^n\Heat (x,x,s^2)d\meas- \frac{s^n}{1-\epsilon}\sum_{j \in \mathcal{M}}e^{-\lambda_js^2}, 
\end{align}
namely
\begin{equation}
s^n\sum_{j \in \mathcal{M}}e^{-\lambda_js^2} \ge \frac{1}{\epsilon}\int_{X \setminus A}s^n\Heat (x,x,s^2)d\meas-\frac{1-\epsilon}{\epsilon}\int_Xs^n\Heat (x,x,s^2)d\meas.
\end{equation}
Letting $s \to 0^+$ shows
\begin{align}
\liminf_{s\to 0^+}s^n\sum_{j \in \mathcal{M}}e^{-\lambda_js^2} &\ge \frac{1}{(4\pi)^{n/2}\epsilon}\mathcal{H}^n(\mathcal{R}_n^*\setminus A)-\frac{1-\epsilon}{(4\pi)^{n/2}\epsilon}\mathcal{H}^n(\mathcal{R}_n^*) \nonumber \\
&\ge \frac{1-\delta}{(4\pi)^{n/2}\epsilon}\mathcal{H}^n(\mathcal{R}_n^*)-\frac{1-\epsilon}{(4\pi)^{n/2}\epsilon}\mathcal{H}^n(\mathcal{R}_n^*) \nonumber \\
&= \frac{\epsilon-\delta}{(4\pi)^{n/2}\epsilon}\mathcal{H}^n(\mathcal{R}_n^*)>0
\end{align}
%Since
%\begin{align}
%&\int_Xs^n\Heat(x,x,s^2)d\meas \nonumber \\
%&= s^n\sum_{j \in \mathcal{M}_i}e^{-\lambda_js^2}+\int_Xs^n\sum_{j \not \in \mathcal{M}_i}\phi_j^2d\meas \nonumber \\
%&\ge s^n\sum_{j \in \mathcal{M}_i}e^{-\lambda_js^2}+\frac{1}{1-\epsilon_0}\int_{X \setminus A_i}s^n\sum_{j \not \in \mathcal{M}_i}\phi_j^2d\meas \nonumber \\
%&\ge s^n\sum_{j \in \mathcal{M}_i}e^{-\lambda_js^2}+\frac{1}{1-\epsilon_0}\int_{X \setminus A_i}s^n\Heat (x,x,s^2)d\meas-\frac{s^n}{1-\epsilon_0}\sum_{j \in \mathcal{M}_i}e^{-\lambda_js^2},
%\end{align}
%letting $l \to \infty$ after replacing $s$ by $s_l$ yields 
%\begin{equation}\label{eq:1102}
%\mathcal{H}^n(\mathcal{R}_n^*) \ge \frac{1}{1-\epsilon}\mathcal{H}^n(\mathcal{R}_n^* \setminus A_i)
%\end{equation}
because it is known from the proof of \cite[Theorem 4.3]{AmbrosioHondaTewodrose} that (\ref{eq:110}) implies the $L^1$-strong convergence of $s^n\Heat (x,x,s^2)$ to $(4\pi)^{-n/2}1_{\mathcal{R}_n^*}$ as $s \to 0^+$. Thus we conclude.
\end{proof}
Note that the proposition above can be applied to a metric measure space $([0,1], \dist_{\mathbb{R}}, \sin^{N-1}tdt)$, see \cite[Example 4.5]{AmbrosioHondaTewodrose}.

\begin{rem}
In connection with the theorem above, it is worth mentioning that 
under the assumption that $\mathcal{H}^n(\mathcal{R}_n^*)<\infty$ (in particular it is satisfied if (\ref{eq:110}) holds), for any $\epsilon>0$ there exists $\delta>0$ such that if a Borel subset $A\subset X$ satisfies $\meas (A)\le \delta$, then $\mathcal{H}^n(A\cap \mathcal{R}_n^*) \le \epsilon$. The proof of this fact is as follows. 

Denoting
\begin{equation}
d\mathcal{H}^n=\phi d \meas,\quad \text{for some  $\phi: \mathcal{R}_n^* \to [0, \infty)$}
\end{equation}
on $\mathcal{R}_n^*$, since
\begin{equation}
\int_{\mathcal{R}_n^*}\phi d\meas=\int_{\mathcal{R}_n^*}d\mathcal{H}^n=\mathcal{H}^n(\mathcal{R}_n^*)<\infty,
\end{equation}
we have $\phi \in L^1(\mathcal{R}_n^*, \meas)$. In particular for any $\epsilon>0$ there exists $\delta>0$ such that if a Borel subset $B \subset \mathcal{R}_n^*$ satisfies $\meas(B)<\delta$, then
\begin{equation}
\int_B\phi d\meas<\epsilon.
\end{equation}
which proves the desired statement.
\end{rem}

%\begin{prop}
%Let 
%\begin{equation}
%(X_i, \dist_i, \mathcal{H}^N, x_i) \stackrel{\mathrm{mGH}}{\to} (X, \dist, \mathcal{H}^N, x)
%\end{equation}
%be a pointed measured Gromov-Hausdorff convergent sequence of pointed non-collapsed $\RCD(K, N)$ spaces.
%Then for any convergent sequence $s_i \to 0^+$ and  any $p \in [1, \infty)$ we have the $L^p_{\mathrm{loc}}$-strong convergence:
%\begin{equation}
%s_i^Np_i(x,x, s_i^2) \to (4\pi)^{-N/2},
%\end{equation}
%where $p_i$ denotes the heat kernel of $(X_i, \dist_i, \mathcal{H}^N)$. 
%\end{prop}
%\begin{proof}
%{\color{magenta}SH: I will write down the proof later.}
%\end{proof}
\begin{rem}\label{remnoncollapsed}
Introduced in \cite[Definition 1.1]{DePhillippisGigli}, an $\RCD(K, N)$ space $(X, \dist, \meas)$ is said to be non-collapsed if $\meas=\mathcal{H}^N$, which is a synthetic counterpart of volume non-collapsed Ricci limit spaces. It is known from \cite{DePhillippisGigli, KM} that non-collapsed $\RCD(K, N)$ spaces have finer properties than that of general $\RCD(K, N)$ spaces. Note that any noncollapsed $\RCD(K, N)$ space has the rectifiable dimension $N$ and that  (\ref{eq:110}) is satisfied as $n=N$ by Bishop-Gromov inequality if it is compact.

In connection with this, we point out that the following three conditions for a compact $\RCD(K, N)$ space $(X, \dist, \meas)$ are equivalent.
\begin{enumerate}
\item We have 
\[\liminf_{\lambda \to \infty}\frac{N(\lambda)}{\lambda^{N/2}}>0.\]
\item $N$ is an integer with $\meas =c\mathcal{H}^N$ for some $c>0$.
\item $N$ is an integer with
\[\lim_{\lambda \to \infty}\frac{N(\lambda)}{\lambda^{N/2}}=\frac{\omega_N}{(2\pi)^{N/2}}\mathcal{H}^N(X) \in (0, \infty).\]
\end{enumerate}
The nontrivial part is the implication from (1) to (2). If (1) holds, then \cite[(5.5)]{AmbrosioHondaTewodrose} with (\ref{eq:gauusian}) implies 
\[0<\liminf_{s \to 0^+}s^N\int_Xp(x,x,s^2)d\meas \le C\liminf_{r \to 0^+}\int_X\frac{r^N}{\meas (B_r(x))}d\meas =C\int_X\lim_{r\to 0^+}\frac{r^N}{\meas (B_r(x))}d\meas,\]
where the final equality comes from Bishop-Gromov inequality and the dominated convergence theorem. In particular
\[\meas \left( \left\{ x \in X |\lim_{r\to 0^+}\frac{r^N}{\meas (B_r(x))}>0\right\}\right)>0.\]
Thus we can apply \cite[Theorem 1.3]{BGHZ} or \cite[Corollary 1.3]{Hon} to prove that (2) holds.

\end{rem}

Finally let us give an answer to a similar question for $X$ as in Theorem \ref{thm:weyl}, namely in the case of singular Weyl's law.
\begin{prop}\label{weyl eig2}
Let $(X, \dist, \meas)$ be as in Theorem \ref{thm:weyl}.
For all $0<\delta<\epsilon<1$ and any closed subset $A \subset X$ with $\mathcal{H}^k(\partial A \cap \mathcal{S})=0$ and 
$\mathcal{H}^k(A\cap \mathcal{S})\le \delta \mathcal{H}^k (\mathcal{S})$,
letting
\begin{equation}
m_{A, \epsilon}(\lambda) =\sharp \left\{j\in \mathbb{N} \,\Big| \,\int_{X \setminus A}f_j^2d\meas \ge (1-\epsilon)\int_Xf_j^2d\meas\,\,\text{and}\,\,\lambda_j \le \lambda \right\} \text{  for any $\lambda>0$},
\end{equation}
we have
\begin{equation}\label{eq:1105}
m_{A, \epsilon}(\lambda) \sim \lambda^{k/2},\quad \text{as $\lambda \to \infty$.}
\end{equation}
In particular 
\begin{equation}
\int_{X \setminus A}f_j^2d\meas \ge (1-\epsilon)\int_Xf_j^2d\meas
\end{equation}
is satisfied for infinitely many eigenfunctions $f_j$ of $-\Delta$ on $X$.
\end{prop}
\begin{proof}
The proof is similar to that of Theorem \ref{prop:ding general} via
\begin{equation}
\int_{X \setminus A}s^k\Heat(x,x,s^2)d\meas \to c \mathcal{H}^k(\mathcal{S} \setminus A).
\end{equation}
\end{proof}
\begin{rem}
It is a direct consequence of the proposition above  that for any $0<\epsilon<1$ and any compact subset $A \subset \mathcal{R}_2$, we have
\begin{equation}
m_{A, \epsilon}(\lambda) \sim \lambda^{k/2}
\end{equation}
which should be compared with \cite[Theorem 1.1]{JZ}  referred at the beginning of the subsection.
\end{rem}

We are also able to disucuss similar asymptotic behaviors of the eigenfunctions for $X$ as in Theorem \ref{2weyl}.
\begin{prop}
Let $(X, \dist, \meas)$ be as in Theorem \ref{2weyl}.
For all $0<\delta<\epsilon<1$ and any closed subset $A \subset X$ with $\mathcal{H}^2(\partial A \cap \mathcal{S})=0$ and 
$\mathcal{H}^2(A\cap \mathcal{S})\le \delta \mathcal{H}^2 (\mathcal{S})$,
letting
\begin{equation}
m_{A, \epsilon}(\lambda) =\sharp \left\{j\in \mathbb{N} \,\Big| \,\int_{X \setminus A}f_j^2d\meas \ge (1-\epsilon)\int_Xf_j^2d\meas\,\,\text{and}\,\,\lambda_j \le \lambda \right\} \text{  for any $\lambda>0$},
\end{equation}
we have
\begin{equation}\label{eq:11059}
m_{A, \epsilon}(\lambda) \sim \lambda \log \lambda ,\quad \text{as $\lambda \to \infty$.}
\end{equation}
In particular 
\begin{equation}
\int_{X \setminus A}f_j^2d\meas \ge (1-\epsilon)\int_Xf_j^2d\meas
\end{equation}
is satisfied for infinitely many eigenfunctions $f_j$ of $-\Delta$ on $X$.
\end{prop}
\begin{proof}The proof is essentially same to that of Proposition \ref{weyl eig2}. The only difference is to find the ``$t/(-\log t), \lambda \log \lambda$'' version of \cite[Propositions 5.5 and 5.6]{AmbrosioHondaTewodrose} instead of ``$t^{\gamma}, \lambda^{\gamma}$'', respectively. Since the proofs of \cite[Propositions 5.5 and 5.6]{AmbrosioHondaTewodrose} still work in this setting, we omit it.
\end{proof}

\appendix

\section{Proofs of Proposition \ref{prop:heat-cover} and Lemma \ref{Duhamelpri}}\label{ap}
\subsection{Fundamental domain}
let $(X, \dist, \meas)$ be an $\RCD(K, N)$ space for some $K \in \mathbb{R}$ and some $N \in [1, \infty)$, and let $\pi:\widetilde X\to X$ be a connected covering space with the lifted metric $\tilde \dist$ and measure $\tilde \meas$ and the deck transformation $\Gamma$. Then $(\widetilde X, \tilde \dist, \tilde \meas)$ be an $\RCD(K, N)$ space because of the same reason as in \cite[Lemma 2.8]{MW}. Let us begin with recalling the following standard notion;
\begin{defn}[Fundamental domain]\label{def_F_domain}
An open connected subset $D$ of $\widetilde X$ is said to be a \textit{fundamental domain} (of the covering $\pi: \widetilde X \to X$) if\\
(1) $D \cap \gamma D = \emptyset$ holds for any $\gamma \in \Gamma$ with $\gamma \neq 1$, and\\
(2) $\widetilde X = \bigcup_{\gamma \in \Gamma}\gamma \overline{D}$.
\end{defn}

To obtain a fundamental domain, we follow the standard construction of Dirichlet domain. We fix a point $\tilde x\in \widetilde X$. For every $\gamma\not= 1$ in $\Gamma$, we put
$$H_\gamma(\tilde{x})=\{y\in \widetilde{X} | \tilde \dist (y,\tilde{x})<\tilde \dist (y,\gamma x) \}.$$

%$$H(\tilde x<\tilde y):=\{\tilde z \in \tilde X | \tilde \dist (\tilde z, \tilde x)<\tilde \dist (\tilde z, \tilde y)\},\quad H(\tilde x=\tilde y):= \{\tilde z \in \tilde X | \tilde \dist (\tilde z, \tilde x)=\tilde \dist (\tilde z, \tilde y)\},$$
%and $H(\tilde x \le \tilde y):=H(x<y)\cup H(x=y)$.
\begin{prop}[Dirichlet domain]\label{D_domain}
For any $\tilde x \in \tilde X$, we define 
\begin{equation}
D(\tilde x):=\bigcap_{1 \neq \gamma \in \Gamma}H_\gamma(\tilde x),
\end{equation}
and call it the Dirichlet domain centered at $\tilde x$. Then $D(\tilde x)$ is a fundamental domain.
\end{prop}
The proof is essentially the same as the Riemannian case. We include the proof for readers' convenience.

\begin{lem}\label{D_domain_lem}
	Let $\gamma\not= 1$ and let $\sigma: [0, L] \to \widetilde{X}$ be a minimal geodesic from $\tilde{x}$. If $\sigma(t_0)$ satisfies $\tilde \dist (\sigma(t_0),\tilde{x})=\tilde \dist (\sigma(t_0),\gamma \tilde{x})$ for some $t_0 \in [0, L]$, then $\sigma (t) \in H_\gamma(\tilde{x})$ for all $t \in [0, t_0)$ and $\sigma(t) \in \widetilde{X}-H_\gamma(\tilde{x})$ for all $t \in (t_0, L]$.
\end{lem}

\begin{proof}
	Let us prove the first statement by contradiction. If not, then there exists $s \in [0, t_0)$ such that $\tilde \dist(\sigma (s),\tilde{x})=\tilde \dist (\sigma(s),\gamma\tilde{x})$. Then
	\begin{align*}
		\tilde \dist (\gamma\tilde{x}, \sigma(t_0)) =\tilde \dist (\tilde{x}, \sigma(t_0))&=\tilde \dist(\tilde{x}, \sigma (s))+\tilde \dist(\sigma(s), \sigma (t_0)) \\
		&=\tilde \dist (\gamma\tilde{x}, \sigma (s))+\tilde \dist(\sigma (s), \sigma (t_0)),
	\end{align*}
	which in particular proves that a minimal geodesic from $\sigma(t_0)$ to $\gamma \tilde{x}$ can be branched at $\sigma(s)$.
	This contradicts a fact that  any minimal geodesic in $X$ is nonbranching \cite[Theorem 1.3]{Qin20}.
	
	Next, let us prove the second statement by contradiction again. If not, then there exists $s' \in (t_0, L]$ such that $\tilde \dist (\sigma(s'),\tilde{x})\le \tilde \dist (\sigma(s'),\gamma\tilde{x})$.
	If the equality holds, then we have a contradiction by the first statement. Thus we have
	\begin{align*}
		\tilde \dist (\sigma (s'), \gamma\tilde{x}) >\tilde \dist(\sigma (s'), \tilde{x})&=\tilde \dist (\sigma(s'), \sigma(t_0))+\tilde \dist(\sigma(t_0), \tilde{x}) \\
		&=\tilde \dist (\sigma(s'), \sigma(t_0))+\tilde \dist(\sigma(t_0), \gamma\tilde{x}) \ge \tilde \dist(\sigma(s'), \gamma\tilde{x}),
	\end{align*}
	which is also a contradiction.
\end{proof}

\begin{proof}[Proof of Proposition \ref{D_domain}]
	We first prove the claim below.
	
	\textit{Claim 1:} The family of closed sets $\{\widetilde{X}-H_\gamma(\tilde{x})\}_{\gamma\not=1}$ is locally finite. In fact, let $y\in\widetilde{X}$ and let $r=d(y,\tilde{x})$. Note that for all $\gamma\not=1$ and all $z\in (\widetilde{X}-H_\gamma(\tilde{x}))\cap B_1(y)$,
	\begin{align*}
		\tilde \dist (z,\gamma y)\le &\tilde \dist (z,y)+\dist (y,\gamma\tilde{x})+\tilde \dist (\gamma\tilde{x},\gamma y)\\
		< & 1+\tilde \dist(y,\tilde{x})+\tilde \dist (\tilde{x},y)=1+2r.
	\end{align*}
    This proves that
    $$(\widetilde{X}-H_\gamma(\tilde{x}))\cap B_1(y) \subset B_{1+2r}(\gamma y) \cap B_1(y).$$
	Because there are only finitely many $\gamma y$ in $B_{2+2r}(y)$, the right hand of the above inclusion relation is non-empty only for finitely many $\gamma$. Claim 1 follows. 
	
	It follows from Claim 1 that $D(\tilde{x})$ is an open subset of $\widetilde{X}$.
	
	By Lemma \ref{D_domain_lem}, for every $y\in D(\tilde{x})$, there is a minimal geodesic $\sigma$ from $\tilde{x}$ to $y$ such that $\sigma$ is contained in $H_\gamma(\tilde{x})$ for all $\gamma\not=1$, thus in $D(\tilde{x})$. This shows that $D(\tilde{x})$ is connected.
	
	Next, we define a set of representative points, denoted as $F$: for each orbit $\Gamma y$, we choose a point in the orbit such that it is the closest to $\tilde{x}$. 
	
	\textit{Claim 2:} $D(\tilde{x})\subset F \subset \overline{D(\tilde{x})}$. In fact, by the definition of $D(\tilde{x})$, any $y\in D(\tilde{x})$ satisfies
	$$\tilde \dist (y,\tilde{x})<\tilde \dist(y,\gamma\tilde{x})=\tilde \dist (\gamma^{-1}y,\tilde{x})$$
	for all $\gamma\not=1$. Thus $y\in F$ and we see that $D(\tilde{x})\subset F$. We prove the other inclusion. Let $y\in F$ and let $\sigma:[0,L]\to\widetilde{X}$ be a minimal geodesic from $\tilde{x}$ to $y$. Note that
	$$\tilde \dist (y,\tilde{x})\le \tilde \dist(\gamma^{-1}y,\tilde{x})=\tilde \dist(y,\gamma\tilde{x})$$
	for all $\gamma\not=1$. By Lemma \ref{D_domain_lem}, the image of $\sigma|_{[0,L)}$ is contained in $H_\gamma(\tilde{x})$ for all $\gamma\not=1$, thus in $D(\tilde{x})$. Hence $y=\sigma(L)\in \overline{D(\tilde{x})}$. This proves Claim 2.
	
	With Claim 2, now we prove that $D(\tilde{x})$ is a fundamental domain. To prove (1) in Definition \ref{def_F_domain}, suppose that there is $\gamma\not=1$ and $y_1,y_2\in D(\tilde{x})$ such that $\gamma y_1=y_2$. By Claim 2, both $y_1,y_2$ belong to $F$. Since the orbit $\Gamma y_1$ only has one point in $F$, we must have $y_1=y_2$, that is, $\gamma y_1=y_1$. This contradicts the facts that $\Gamma$-action is free and $\gamma\not=1$. To prove (2) in Definition \ref{def_F_domain}, note that
	$$\widetilde{X}=\bigcup_{\gamma\in \Gamma} \gamma F \subset \bigcup_{\gamma\in \Gamma} \gamma \overline{D(\tilde{x})}.$$ 
\end{proof}

\begin{defn}\label{mmdomain}
	A fundamental domain $D$ of the covering $\pi: (\widetilde{X},\tilde{\dist},\tilde{\mathfrak{m}}) \to (X,d,\mathfrak{m})$ is said to be a \textit{metric measure fundamental domain} if $\tilde{\mathfrak{m}}(\partial D)=0$.
\end{defn}
Since $\overline{\gamma D} \cap \overline{D} \subset \partial D$ holds for any $\gamma \neq 1$ and any fundamental domain $D$, we have 
\begin{equation}\label{fubini}
\int_{X}fd \meas = \int_{\widetilde X}\tilde fd\tilde \meas
\end{equation}
for any nonnegative function $\tilde f:\widetilde X\to [0, \infty]$ if $D$ is a metric measure fundamental domain, where $f(x)=\sum_{\gamma \in \Gamma}\tilde f(\gamma \tilde x)$ and $\pi (\tilde x)=x$. The proof of (\ref{fubini}) is as follows. 

Fistly, since $\pi$ is a local isomorphism as metric measure spaces, we have $\meas (X \setminus \pi(D))=0$. Then the monotone convergence theorem allows us to compute
\begin{align}
\int_{X} fd \meas =\int_{\pi(D)}fd\meas =\int_{\pi (D)}\sum_{\gamma \in \Gamma}\tilde f(\gamma \tilde x)d\tilde \meas(\tilde x) 
&=\sum_{\gamma \in \Gamma}\int_D\tilde f(\gamma \tilde x)d\tilde \meas(\tilde x) \nonumber \\
&=\sum_{\gamma \in \Gamma}\int_{\gamma D}\tilde f(\tilde x)d\tilde \meas (\tilde x) =\int_{\widetilde X}\tilde fd\tilde \meas.
\end{align}
This argument will play a role in the next subsection.

\begin{rem}\label{stbr}
%We do not know whether a Dirichlet domain is a metric measure fundamental domain or not. In connection with this, let us provide the following notion;
A minimal geodesic $\sigma:[-\epsilon, \epsilon] \to \widetilde X$ in $\widetilde X$ is said to be \textit{strongly nonbranching} if a minimal geodesic $\eta: [-\delta, \delta] \to X$ for some $\delta<\epsilon$ satisfy $\sigma(0)=\eta(0)$ and $\angle \dot \sigma\dot \eta(0)=0$, then $\sigma=\eta$ on $[-\delta, \delta]$ (see \cite{HM, H} for the definition of angles).
Then if any minimal geodesic in $\widetilde {X}$ is strongly nonbranching, then a Dirichlet domain is a metric measure fundamental domain. The proof of this fact is as follows.

Since the boundary of $D(\tilde x)$ is included in 
\begin{equation}\label{eqboundary}
\bigcup_{1 \neq \gamma \in \Gamma}\left\{z \in \widetilde X \Big| \tilde \dist(z, \tilde x)=\tilde \dist (z, \gamma \tilde x)\right\}, 
\end{equation}
it is enough to prove that the set (\ref{eqboundary}) is $\tilde \meas$-negligible. If not, then there exists $\gamma \neq 1$ such that 
\begin{equation}\label{cutloc}
\tilde \meas \left( \left\{z \in \widetilde X \Big| \tilde \dist(z, \tilde x)=\tilde \dist (z, \gamma \tilde x)\right\} \setminus \left( \mathrm{Cut}(\tilde x)\cup \mathrm{Cut}(\gamma \tilde x)\right) \right)>0
\end{equation}
because the proof of \cite[Lemma 3.1]{GRS} allows us to conclude that the cut locus $\mathrm{Cut}(\tilde z)$ of any point $\tilde z$ in $\tilde X$ is $\tilde \meas$-negligible. 
On the other hand the locality of the minimal relaxed slope yields
\begin{equation}\label{angle}
|\nabla (\tilde \dist_{\tilde x}-\tilde \dist_{\gamma \tilde x})|^2=0,\quad \text{for $\tilde \meas$-a.e. $\tilde z \in \left\{z \in \widetilde X \Big| \tilde \dist(z, \tilde x)=\tilde \dist (z, \gamma \tilde x)\right\}$,}
\end{equation}
where $\tilde \dist_{\tilde z}$ denotes the distance function from $\tilde z$. In particular (\ref{cutloc}) and (\ref{angle}) imply that there exists $\widetilde z \in X \setminus  (\mathrm{Cut}(\tilde x) \cup \mathrm{Cut}(\gamma \tilde x))$ such that all two minimal geodesics $\sigma, \eta$ from $\tilde z$ to $\tilde x$, $\gamma \tilde x$ respectively satisfy $\angle \dot \sigma \dot \eta (0)=0$ which contradicts the strongly nonbranching property. Thus $\tilde \meas (\partial D(\tilde x))=0$.

Note that this strongly nonbranching property does not follow the nonbranching property established in \cite[Theorem 1.3]{Qin20}.
\end{rem}
\begin{rem}
We can easily see that Proposition \ref{D_domain} can be generalized to geodesic metric spaces with nonbranching property. It is worth mentioning that any Alexandrov space satisfies the strongly nonbranching property by definition. Thus in particular all of the above discussion can be applied to Alexandrov spaces with any Borel measure.
\end{rem}

\begin{rem}\label{Ye}
A result in a very recent preprint \cite{Ye23} proves that the right hand side of (\ref{eqboundary}) is actually $\meas$-negligible without assuming strongly nonbranching property, thus the Dirichlet domain centered at any point is actually a metric measure fundamental domain. See \cite[Theorem 4]{Ye23}.
\end{rem}

%\begin{rem}
%We do not know whether $\tilde \meas (H(\tilde x=\gamma \tilde x))=0$ holds for any $\gamma \neq 1$. However if this is true, then $\tilde \meas (\partial D(\tilde x))=0$ because of $\partial D(\tilde x)\subseteq \bigcup_{1\neq \gamma \in \Gamma}H(\tilde x = \gamma \tilde x)$.
%\end{rem}
\subsection{Proof of Proposition \ref{prop:heat-cover}; the heat kernel on a covering space}\label{appp}
We continue our discussion with the same setting as in the previouse subsection.
%We start by recalling our setting;
%let $(X, \dist, \meas)$ be an $\RCD(K, N)$ space for some $K \in \mathbb{R}$ and some $N \in [1, \infty)$, and let $\pi:\tilde X\to X$ be a connected covering space with the lifted metric $\tilde \dist$ and measure $\tilde \meas$ and the deck transformation $\Gamma$. Then $(\tilde X, \tilde \dist, \tilde \meas)$ be an $\RCD(K, N)$ space because of the same reason as in \cite[Lemma 2.8]{MW}. 

For the simplicity of notation we give a proof of  Proposition \ref{prop:heat-cover} only in the case when $K=0$, however it is easy to generalize it to the general case when $K<0$ via the Li-Yau inequalities  (\ref{eq:gauusian}), (\ref{eq:grad gaus}) and (\ref{eq:lap gaussian}).
The proof is divided into several lemmas as follows.

For any $y \in X$, take (a sufficiently small) $r(y)>0$ satisfying that the preimage $\pi^{-1}(B_{s}(y))$ consists of the pairwise disjoint union of $\{B_{s}(\gamma \tilde y)\}_{\gamma \in \Gamma}$ for any $s \in (0, r(y)]$, where $\pi(\tilde y)=y$. Without loss of generality we can assume 
\begin{equation}\label{unifrem}
\inf_{y \in A}r(y)>0,\quad \text{for any bounded subset $A \subset X$.}
\end{equation}
Let us recall our main target;
for any $t>0$, define a symmetric function on $X \times X$;
\begin{equation}\label{heatcovering}
\Heat_1(x, y, t):=\sum_{\gamma \in \Gamma}\Heat_{\widetilde X}(\widetilde x, \gamma \tilde y, t) \in (0, \infty],
\end{equation} 
where $x, y \in X$ and $\tilde x, \tilde y \in \widetilde X$ with $\pi(\tilde x)=x$ and $\pi(\tilde y)=y$.
Note that the right hand side of (\ref{heatcovering}) does not depend on the choices of $\tilde x, \tilde y$, thus it is well-defined. 

\begin{lem}\label{finitenesslem}
$\Heat_1$ is finite. More precisely for any $s  \in (0,  r(y)]$, we have
\begin{align}\label{heatestimate}
\meas (B_{s}(y))\Heat_1(x, y, t) &\le C(N) \exp\left(\frac{s^2}{4t}\right)\int_{\widetilde X}\Heat_{\widetilde X} (\tilde x, \tilde z, 2t)d \tilde \meas( \tilde z) \nonumber \\
&\le C(N) \exp\left(\frac{s^2}{4t}\right)\tilde \meas (B_{\sqrt{t}}(\tilde x)).
\end{align}
%In particular if $\sqrt{t} \le r(y)$, then
%\begin{align}\label{heatestimate2}
%\meas (B_{\sqrt{t}}(y)) \sum_{\gamma \in \Gamma}\Heat_{\tilde X}(\tilde x, \gamma \tilde y, t) &\le C(K, N) \int_{\tilde X}\Heat_{\tilde X} (\tilde x, \tilde z, 2t)d \tilde \meas(\tilde z) \le C(K, N).
%\end{align}
\end{lem}
\begin{proof}
Applying the parabolic Harnack inequality proved in \cite[Theorem 1.4]{GM} and in \cite[Theorem 1.3]{Jiang2015} implies
\begin{equation*}
\tilde \meas (B_{s}(\tilde y)) \Heat_{\widetilde X} (\tilde x, \tilde y, t) \le C(N) \exp\left(\frac{s^2}{4t}\right) \int_{B_{s}(\tilde y)}\Heat_{\widetilde X} (\tilde x, \tilde z, 2t)d \tilde \meas( \tilde z). 
\end{equation*}
Thus
\begin{align*}
\meas (B_s(y))\Heat_1(x,y,t)
&= \sum_{\gamma \in \Gamma} \tilde \meas (B_s(\gamma \tilde y))\Heat_{\widetilde X}(\tilde x, \gamma \tilde y, t)\\
&\le C(N) \exp\left(\frac{s^2}{4t}\right) \int_{\widetilde X}\Heat_{\widetilde X} (\tilde x, \tilde z, 2t)d \tilde \meas( \tilde z). 
\end{align*}
The final inequality in (\ref{heatestimate}) comes from the Li-Yau inequality (\ref{eq:gauusian}), Bishop-Gromov inequality and \cite[Lemma 2.7]{BGHZ}.
%\begin{equation}\label{cava}
%\int_{\tilde X}\tilde \meas (B_{\sqrt{t}}(y))^l\exp \left( -\frac{\tilde \dist(\tilde x, \tilde y)^2}{t}\right)d\tilde \meas (\tilde y) \le C(K, N, T, L)\tilde \meas (B_{\sqrt{t}}(\tilde x))^{l+1}
%\end{equation}
%for all $\tilde x \in \tilde X, l \in [-L, L]$ and $t \in (0, T)$. Applying (\ref{cava}) as $l=0$, we conclude.
\end{proof}
%\begin{cor}
%For all $y \in X$ and $t>0$, $\Heat_1(x, y, t)$ viewed as a function of $x$ on $X$, is in $L^2(X, \meas)$ if 
%\begin{equation}\label{eq:l2bounded}
%\min \left\{\int_X\frac{1}{\meas (B_1(x))^2}d\meas (x), \meas(X)\right\}<\infty.
%\end{equation}
%\end{cor}
%\begin{proof}
%This is a direct consequence of (\ref{heatestimate}) and a fact that $\Heat_1$ is symmetric.
%\end{proof}
\begin{cor}\label{corsum}
If $s \in (0, r(y)]$, then  
\begin{equation}\label{infinitesum}
\sum_{\gamma \in \Gamma}\exp \left(-\frac{\tilde \dist(\tilde x, \gamma \tilde y)^2}{3t}\right)\le C(N) \exp\left(\frac{s^2}{4t}\right)\frac{\tilde \meas (B_{\sqrt{t}}(\tilde x))^2}{\meas (B_s(y))}.
\end{equation}
\end{cor}
\begin{proof}
Since the Li-Yau inequality (\ref{eq:gauusian}) implies
\begin{equation*}
\Heat_1(x, y, t) =\sum_{\gamma \in \Gamma}\Heat_{\widetilde X}(\tilde x, \gamma \tilde y, t) \ge \frac{C(N)}{\tilde \meas (B_{\sqrt{t}}(\tilde x))} \sum_{\gamma \in \Gamma}\exp \left( -\frac{\tilde \dist (\tilde x, \gamma \tilde y)^2}{3t}\right),
\end{equation*}
we have (\ref{infinitesum}) because of (\ref{heatestimate}).
\end{proof}
In order to continue the proof of Proposition \ref{prop:heat-cover}, let us recall local notions of Sobolev functions and Laplacian.
\begin{defn}[Local Sobolev and Laplacian]
Let us denote by $H^{1,2}_{\mathrm{loc}}(X, \dist, \meas)$ the set of all $f \in L^2_{\mathrm{loc}}(X, \meas )$ satisfying that $\phi f \in H^{1,2}(X, \dist, \meas)$ for any $\phi \in \mathrm{Lip}_c(X, \dist)$ and that $|\nabla f| \in L^2_{\mathrm{loc}}(X, \meas)$, where $|\nabla f|$ makes sense on $X$ because of the locality of the minimal relazed slope, where $\mathrm{Lip}_c(X, \dist)$ denotes the set of all Lipschitz functions on $X$ with compact supports. Moreover we denote by $D_{\mathrm{loc}}(\Delta)$ the set of all $f \in H^{1,2}_{\mathrm{loc}}(X, \dist, \meas)$ satisfying that there exists a (unique) $\psi \in L^2_{\mathrm{loc}}(X, \meas)$, denoted by $\psi=\Delta f$ if no confusion, such that 
\begin{equation}
\int_X\langle \nabla f, \nabla \phi\rangle d\meas=-\int_X\psi \phi d\meas,\quad \text{for any $\phi \in \mathrm{Lip}_c(X, \dist)$.}
\end{equation}
\end{defn}
The stability properties of local Sobolev functions and of local Laplacian can be found in \cite{AmbrosioHonda}.
\begin{lem}\label{lapstab}
$\Heat_1$ is locally Lipschitz continuous. Moreover for all $y \in X$ and $t>0$, $\Heat_1(x, y, t)$ viewed as a function of $x$ on $X$, is in $D_{\mathrm{loc}}(\Delta)$ with
\begin{equation}
\Delta_x\Heat_1 (x,y,t)=\sum_{\gamma \in \Gamma}\Delta_x \Heat_{\widetilde X} (\tilde x,\gamma y,t)=\sum_{\gamma \in \Gamma}\frac{d}{dt}\Heat_{\widetilde X} (x,\gamma y,t), \quad \text{in $L^2(X, \meas)$.}
\end{equation}
\end{lem}
\begin{proof}
Let $A$ be a finite subset of $\Gamma$.
Then for any $s \in (0, r(y)]$ by (\ref{eq:grad gaus}) and Corollary \ref{corsum} we have
\begin{align}\label{eq:01}
\left|\nabla_{\tilde x} \sum_{\gamma \in A}\Heat_{\widetilde X} (\tilde x, \gamma \tilde y, t)\right|
&\le \sum_{\gamma \in \Gamma}\left| \nabla_{\widetilde x}\Heat_{\tilde X} (\tilde x,\gamma \tilde y,t)\right| \nonumber \\
&\le \frac{C(N)}{\sqrt{t}\tilde \meas (B_{\sqrt{t}}(\tilde x))}\sum_{\gamma \in \Gamma} \exp \left( -\frac{\tilde \dist( \tilde x, \gamma \tilde y)^2}{5t}\right)\nonumber \\
&\le \frac{C(N)}{\sqrt{t}\tilde \meas (B_{\sqrt{t}}(\tilde x))} \cdot \exp\left(\frac{s^2}{t}\right)\cdot \frac{\tilde \meas (B_{\sqrt{5t/3}}(\tilde x))^2}{\meas(B_{s}(y))}\nonumber \\
& \le C(N)\exp\left(\frac{s^2}{t}\right)\frac{\tilde \meas (B_{\sqrt{t}}(\tilde x))}{\sqrt{t}\meas (B_s(y))}
\end{align}
which gives an equi-local Lipschitz continuity of $\{\sum_{\gamma \in A}\Heat_{\widetilde X} (\tilde x, \gamma \tilde y, t)\}_{A \subseteq \Gamma}$ on $\tilde X\times \tilde X \times (0, \infty)$ because of (\ref{unifrem}). Thus letting $A \uparrow \Gamma$ completes the proof of the desired local Lipschitz continuity.

The remaining statements come from the dominated convergence theorem, the stability of the Laplacian and an estimate;
\begin{equation*}\label{eq:02}
\left|\frac{d}{dt} \sum_{\gamma \in A}\Heat_{\widetilde X}(\tilde x, \gamma \tilde y, t)\right| \le C(N)\exp\left(\frac{s^2}{t}\right)\cdot\frac{\tilde \meas (B_{\sqrt{t}}(\tilde x))}{t\meas (B_s(y))}
\end{equation*}
which is proved by a similar way as in (\ref{eq:01}) via (\ref{eq:lap gaussian}).
\end{proof}
%From now on we need the following notion, a fundamental domain, whose existence is assumed in  Proposition \ref{prop:heat-cover}.
%\begin{defn}[Fundamental domain]\label{fundamental}
%We say that a Borel subset $X_1\subset \tilde X$ is a \textit{fundamental domain of $X$ in $\tilde X$} if the following four conditions are satisfied;
%\begin{enumerate}
%\item if $\meas (\gamma_1X_1 \cap \gamma_2 X_1)>0$, then $\gamma_1=\gamma_2$;
%\item $\tilde \meas (\tilde X \setminus \bigcup_{\gamma \in \Gamma}\gamma X_1)=0$;
%\item $\pi|_{X_1}$ preserves the distance locally and the measure;
%\item $\meas (X \setminus \pi (X_1))=0$.
%\end{enumerate} 
%\end{defn}
%It is remarked that in general we do not know whether there exists a fundamental domain of $X$ in $\tilde X$ or not.
\begin{lem}\label{l2lem}
For all $x \in X$ and $t>0$, $\Heat_1(x, y, t)$ viewed as a function of $y$ on $X$, is in $D(\Delta)$ if some fundamental domain $D$ of $X$ has the $\meas$-negligible topological boundary.
\end{lem}
\begin{proof}
Since
\begin{align}\label{l2estimateheat}
\int_X\Heat_1(x, y, t)^2d\meas (y)
&=\sum_{\gamma_1, \gamma_2 \in \Gamma}\int_{D}\Heat_{\widetilde X}(\tilde x, \gamma_1 \tilde y, t)\Heat_{\widetilde X}(\tilde x, \gamma_2\tilde y, t)d\tilde \meas (\tilde y)\nonumber \\
&=\sum_{\gamma_1, \gamma_2 \in \Gamma}\int_{D}\Heat_{\widetilde X}(\gamma_1^{-1}\tilde x, \tilde y, t)\Heat_{\widetilde X}(\tilde x, \gamma_2\tilde y, t)d\tilde \meas (\tilde y)\nonumber \\
&=\sum_{\gamma_1, \gamma_2 \in \Gamma}\int_{D}\Heat_{\widetilde X}(\gamma_2\gamma_1^{-1}\tilde x,  \gamma_2\tilde y, t)\Heat_{\widetilde X}(\tilde x, \gamma_2\tilde y, t)d\tilde \meas (\tilde y)\nonumber \\
&=\sum_{\gamma_2 \in \Gamma}\left( \int_{D}\left(\sum_{\gamma_1 \in \Gamma}\Heat_{\tilde X}(\gamma_2\gamma_1^{-1}\tilde x,  \gamma_2\tilde y, t)\right)\Heat_{\widetilde X}(\tilde x, \gamma_2\tilde y, t)d\tilde \meas (\tilde y)\right)\nonumber \\
&=\sum_{\gamma_2 \in \Gamma}\left( \int_{D}\left(\sum_{\gamma \in \Gamma}\Heat_{\widetilde X}(\gamma \tilde x,  \gamma_2\tilde y, t)\right)\Heat_{\widetilde X}(\tilde x, \gamma_2\tilde y, t)d\tilde \meas (\tilde y)\right)\nonumber \\
&=\sum_{\gamma \in \Gamma}\left(\sum_{\gamma_2 \in \Gamma}\int_{X_1}\Heat_{\widetilde X} (\gamma \tilde x, \gamma_2\tilde y, t)\Heat_{\widetilde X} (\tilde x, \gamma_2\tilde y, t)d\tilde \meas (\tilde y)\right) \nonumber \\
&=\sum_{\gamma \in \Gamma}\int_{\widetilde X}\Heat_{\widetilde X} (\gamma \tilde x, \tilde y, t)\Heat_{\widetilde X}(\tilde x, \tilde y, t)d\tilde \meas (\tilde y) \nonumber \\
&=\Heat_1(\tilde x, \tilde x, 2t)<\infty,
\end{align}
we have $\Heat_1 \in L^2(X, \meas)$.
On the other hand, 
\begin{align}\label{h12est}
\int_X|\nabla_y\Heat_1|^2(x,y,t)d\meas (y)
&=\int_{D}\left\langle \nabla_{\tilde y}\sum_{\gamma \in \Gamma}\Heat_{\widetilde X}(\tilde x, \gamma \tilde y, t), \nabla_{\tilde y}\sum_{\gamma \in \Gamma}\Heat_{\widetilde X}(\tilde x, \gamma \tilde y, t)\right\rangle d\tilde\meas (\tilde y) \nonumber \\
&\le \sum_{\gamma_1, \gamma_2 \in \Gamma}\int_{D}\left|\nabla_{\tilde y}\Heat_{\widetilde X}(\tilde x, \gamma_1 \tilde y, t)\right| \cdot \left|\nabla_{\tilde y}\Heat_{\widetilde X}(\tilde x, \gamma_2 \tilde y, t)\right| d\tilde \meas (\tilde y). %\nonumber \\
%&\le \int_{X_1}\left| \sum_{\gamma \in \Gamma} \exp \left(-\frac{\tilde \dist (\tilde x, \gamma \tilde y)^2}{5t} +C(K, N)t \right)\right|^2d\tilde \meas (\tilde x) \nonumber \\
%&\le e^{C(K, N)t}\int_{X_1}\left| \sum_{\gamma \in \Gamma}\exp \left(-\frac{\tilde \dist (\tilde x, \gamma \tilde y)^2}{5t}\right)\right|^2d\tilde \meas(\tilde x) \nonumber \\
%&\le  e^{C(K, N)t}\int_{X_1}\left| \sum_{\gamma \in \Gamma}\Heat_{\tilde X}(\tilde x, \gamma \tilde y, 2t)\right|^2d\tilde \meas(\tilde x) \nonumber \\
\end{align}
Since the Li-Yau inequalities (\ref{eq:gauusian}) and (\ref{eq:grad gaus}) show for $\tilde x , \tilde y \in \tilde X$
\begin{equation*}
|\nabla_{\tilde y}\Heat_{\widetilde X}(\tilde x, \tilde y, t)| \le \frac{C(N)}{\sqrt{t}\meas (B_{\sqrt{t}}(\tilde y))}\exp \left( -\frac{\tilde \dist (\tilde x, \tilde y)^2}{5t}\right) \le \frac{C(N)}{\sqrt{t}}\Heat_{\widetilde X}(\tilde x, \tilde y, 2t), 
\end{equation*}
the right hand side of (\ref{h12est}) is bounded above by
\begin{equation*}
\frac{C(N)}{t}\sum_{\gamma_1, \gamma_2 \in \Gamma}\int_{D}\Heat_{\widetilde X}(\tilde x, \gamma_1\tilde y, 2t) \Heat_{\widetilde X}(\tilde x, \gamma_2\tilde y, 2t)d\tilde \meas (\tilde y)
\end{equation*}
which is also bounded above by $C(N)t^{-1}\Heat_1(x,x,4t)$ because of (\ref{l2estimateheat}). Thus $\Heat_1 \in H^{1,2}(X, \dist, \meas)$. Similarly, using Lemma \ref{lapstab} with
\begin{equation*}
|\Delta_{\tilde y}\Heat_{\widetilde X}(\tilde x, \tilde y, t)| \le C(N)t^{-1}\Heat_{\widetilde X}(\tilde x, \tilde y, 2t)
\end{equation*}
which is justified by the Li-Yau inequalities (\ref{eq:gauusian}) and (\ref{eq:lap gaussian}) (or \cite[Theorem 3]{Davies}), we have
\begin{equation*}
\int_X\left(\Delta_y\Heat_1(x,y,t) \right)^2d\meas (y)\le \frac{C(N)}{t^{2}}\Heat_{\widetilde X}(\tilde x, \tilde x, 4t)<\infty.
\end{equation*}
Thus $\Heat_1 \in D(\Delta)$.
\end{proof}
From now on we assume that some fundamental domain $D$ has the $\meas$-negligible topological boundary (see also Remark \ref{Ye}).
Note that 
\begin{equation*}
\int_X\Heat_1(x,y,t)d\meas(\tilde y)=\sum_{\gamma \in \Gamma}\int_{D}\Heat_{\widetilde X}(\tilde x, \gamma \tilde y, t)d\tilde \meas(\tilde y)=\int_{\widetilde X}\Heat_{\widetilde X}(\tilde x, \tilde y, t)d\tilde \meas (\tilde y)=1.
\end{equation*}

Then thanks to Lemma \ref{l2lem}, for all $f \in L^2(X, \meas), x \in X$ and $t>0$,
\begin{equation*}
\heat_t^1f(x):=\int_X\Heat_1(x,y,t)f(y)d\meas (y)
\end{equation*}
is well-defined.
\begin{lem}\label{lem:2}
Let $f \in C_c(X)$. Then we have the following;
\begin{enumerate}
\item $\heat^1_tf(x)=\heat_t (f\circ \pi)(x)$ with $\heat^1_tf(x) \to f(x)$ as $t \to 0^+$ for  any $x \in X$. In particular  $\heat^1_tf$ $L^2$-locally strongly converge to $f$ as $t \to 0^+$;
\item for any $t>0$, $\heat^1_tf \in D(\Delta)$ with $\frac{d}{dt}\heat^1_tf=\Delta \heat^1_tf$ in $L^2(X, \meas)$.
\end{enumerate}
\end{lem}
\begin{proof}
Since (2) is a direct consequence of Lemma \ref{lapstab}, let us check (1).
By definition of $\tilde X$, we have
\begin{align}
\heat_t(f \circ \pi)(\tilde x)&=\int_{\widetilde X}\Heat_{\widetilde X} (\tilde x, \tilde y, t) f\circ \pi (\tilde y)d \tilde \meas (\tilde y) \nonumber \\
&=\sum_{\gamma \in \Gamma}\int_{D}\Heat_{\widetilde X} (\tilde x,\gamma \tilde y,t)f \circ \pi (\tilde y)d \tilde \meas (\tilde y) \nonumber \\
&=\int_{X}\Heat_1 (x, y, t)f(y)d \meas(y) =\heat^1_tf(x).
\end{align}
On the other hand since $f \circ \pi \in C_b(Y)$, 
thanks to \cite[Lemma 2.54]{MS}, we know $\heat_t(f\circ \pi)(x)$ converges to $f \circ \pi(x)$, namely  
$
\heat^1_tf(x) \to f(x).
$
The final statement of (1) comes from this with the dominated convergence theorem and a fact $|\heat^1_tf(x)| \le \sup |f|$.
\end{proof}
As an immediate consequence of Lemma \ref{lem:2} (1), we have $\heat^1_{t_i}f_i(x) \to f(x)$ for all $x \in X, t_i \to 0^+$ and uniform convergent sequence $f_i \to f$ in $C_c(X)$.
We are now in a position to prove Proposition \ref{prop:heat-cover}.
\begin{lem}
We have $\Heat_1=\Heat_X$. Namely Proposition \ref{prop:heat-cover} holds.
\end{lem}
\begin{proof}
Fix $w \in X$.
For any $R \ge 1$, take a Lipschitz cut-off $\phi_R:X \to [0, 1]$ with $\phi_R\equiv 1$ on $B_R(w)$, $\mathrm{supp}\,\phi_R \subset B_{2R}(w)$ and $|\nabla \phi_R|\le 2R^{-1}$. Then
Lemma \ref{lem:2} (1) with the continuity of $\Heat_1$ allows us to show
\begin{equation}\label{heat difference}
\Heat(x, y, t)-\Heat_1(x, y, t)=\int_0^t\frac{d}{ds}\int_X\phi_R(z)^2\Heat_1(x,z,t-s)\Heat_X(z,y,s)d\meas (z)ds. 
\end{equation}
Moreover, Lemma \ref{lem:2} (2) also allows us to prove that the right hand side of (\ref{heat difference}) is equal to 
\begin{align}
&-\int_0^t\int_X\phi_R(z)^2\Delta_z \Heat_1(x, z, t-s)\cdot \Heat_X (z, y, s)d\meas(z) ds\nonumber \\
&+\int_0^t\int_X\phi_R(z)^2H_1(x, z, t-s)\Delta_z\Heat_X (z,y,s)d\meas (z) ds,
\end{align}
namely, $\Heat-\Heat_1$ is equal to
\begin{align}\label{keys}
&2\int_0^t\int_X\phi_R(z)\Heat_X (z, y, s)\left\langle \nabla_z \Heat_1(x, z, t-s), \nabla \phi_R(z)\right\rangle d\meas(z) ds\nonumber \\
&-2\int_0^t\int_X\phi_R(z)H_1(x, z, t-s)\left\langle \nabla_z \Heat_X (z, y, s), \nabla \phi_R(z)\right\rangle d\meas (z) ds.
\end{align}
On the other hand the first term of (\ref{keys}) can be estimated as;
\begin{align}
&2\left|\int_0^t\int_X\phi_R(z)\Heat_X (z, y, s)\left\langle \nabla_z \Heat_1(x, z, t-s), \nabla \phi_R(z)\right\rangle d\meas(z) ds\right|\nonumber \\
&\le \frac{4}{R}\int_0^t\left(\int_{B_{2R}(w)\setminus B_R(w)}\Heat_X^2(z,y,s)d\meas(z) \right)^{\frac{1}{2}} \cdot \left( \int_{B_{2R}(w)\setminus B_R(w)}|\nabla_z\Heat_1(x,z,t-s)|^2d\meas (z)\right)^{\frac{1}{2}}ds \nonumber \\
&\le \frac{4}{R}\int_0^t\left(\int_{X\setminus B_R(w)}\Heat_X^2(z,y,s)d\meas(z) \right)^{\frac{1}{2}} \cdot \left( \int_{X \setminus B_R(w)}|\nabla_z\Heat_1(x,z,t-s)|^2d\meas (z)\right)^{\frac{1}{2}}ds. 
\end{align}
Thus letting $R \to \infty$ with the monotone convergence theorem and Lemma \ref{l2lem} shows that the first term of (\ref{keys}) converges to $0$. Similarly we see that under the same limit of $R\to \infty$, the second term of (\ref{keys}) also converges to $0$. Thus letting $R\to \infty$ in (\ref{heat difference}), we have $\Heat=\Heat_1$.
\end{proof}
\subsection{Proof of Lemma \ref{Duhamelpri}; Duhamel principle}\label{duh}
Since the proof is essentially similar to the discussions in the previous subsection, let us provide only an outline of the proof because this situation is simpler than the previous one since the space $\widetilde Y$ is compact.
Firstly let us denote by $\Heat_2$ the right hand side of (\ref{eq:heat duhamel}). Then it is trivial that $\Heat_2$ is locally Lipschitz on $\widetilde Y \times \widetilde Y \times (0, \infty)$ because $\Heat_{\bar Y}$ and $\Heat_{\widetilde Y}$ are locally Lipschitz.
Thus in particular for all $f \in L^2(\widetilde Y, \tilde \meas), x \in X$ and $t>0$, 
\begin{equation}
\heat^2_tf(x):=\int_{\widetilde Y}\Heat_2 (x, y, t)f(y)d\tilde \meas(y)
\end{equation}
is well-defined. Recall (\ref{estimateq});
\begin{equation}
|\HeatQ(z, y, t)| \le \exp \left(-\frac{C}{t}\right).
\end{equation}
In particular  for any $f \in C(\tilde Y)$
we have 
\begin{equation}
\lim_{s \to 0^+}\int_{\widetilde Y}f(z)\HeatQ(z, y, s)d\tilde \meas(z)=0
\end{equation}
and
\begin{equation}\label{pointwiselimit}
\lim_{t \to 0^+}\heat^2_tf(x)=\int_{\widetilde Y}\HeatK(x,y,t)f(y)d\tilde \meas(y)=f(x),
\end{equation}
where the final equality in (\ref{pointwiselimit}) comes from the definition of $\HeatK$.

Then recalling a formula;
\begin{equation}
\frac{d}{dt}\int_0^t\phi(t,s)ds=\int_0^t\frac{d}{dt}\phi(t,s)ds+\phi(t,s),
\end{equation}
we have
\begin{align}\label{calc}
\frac{d}{dt}\Heat_2(x,y,t)
&=\frac{d}{dt}\HeatK(x,y,t)-\int_0^t\int_{\widetilde Y}\Heat_{\widetilde Y} (x, z, s)\frac{d}{dt}\HeatQ(z, y, t-s)d\tilde \meas(z)ds \nonumber \\
&\qquad -\lim_{u \to 0^+}\int_{\widetilde Y}\Heat_{\widetilde Y}(x, z, t-u)\HeatQ(z, y, u)d\tilde \meas(z) \nonumber \\
&=\frac{d}{dt}\HeatK(x,y,t)+\int_0^t\int_{\widetilde Y}\Heat_{\widetilde Y} (x, z, s)\frac{d}{ds}\HeatQ(z, y, t-s)d\tilde \meas(z)ds \nonumber \\
&=\frac{d}{dt}\HeatK(x,y,t)+\lim_{u \to 0}\int_{\widetilde Y}\Heat_{\widetilde Y}(x, z, t-u) \HeatQ(z,y,u)d\tilde \meas(z)-\HeatQ(x,y, t) \nonumber \\
&\qquad-\int_{\tilde Y}\int_0^t\frac{d}{ds}\Heat_{\widetilde Y}(x,z,s) \cdot \HeatQ(z, y, t-s)dsd\tilde \meas(z) \nonumber \\
&=\frac{d}{dt}\HeatK(x,y,t)-\HeatQ(x,y,t)-\int_{\widetilde Y}\int_0^t\frac{d}{ds}\Heat_{\widetilde Y}(x,z,s) \cdot \HeatQ(z, y, t-s)dsd\tilde \meas(z) \nonumber \\
&=\Delta_x\HeatK(x, y,t)-\int_{\widetilde Y}\int_0^t\Delta_x\Heat_{\widetilde Y}(x,z,s) \cdot \HeatQ(z, y, t-s)dsd\tilde \meas(z) \nonumber \\
&=\Delta_x\Heat_2(x, y, t).
\end{align}
In particular combining (\ref{calc}) with (\ref{pointwiselimit}) shows that for any $f \in C(\widetilde Y)$, we have $\frac{d}{dt}\heat^2_tf=\Delta \heat^2_tf$ and $\heat^2_tf \to f$ as $t \to 0^+$ in $L^2(\widetilde Y, \tilde \meas)$ because of the dominated convergence theorem with $|\heat^2_tf| \le \sup |f|$ by definition.
In particular 
$\heat^2_tf$  coincides with the heat flow $\tilde \heat_tf$ on $\widetilde Y$.

In order to conclude that $\Heat_2 =\Heat_{\tilde Y}$, it is enough to check $\heat^2_tf=\tilde \heat_tf$ for any $f \in L^2(\widetilde Y, \tilde \meas)$ because of the continuity of $\Heat_2$. 
Fix $f \in L^2(\widetilde Y, \tilde \meas)$ and take a sequence $f_i \in C(\widetilde Y)$ with $f_i \to f$ in $L^2(\widetilde Y, \tilde \meas)$.
It is known that $\tilde \heat_tf_i\to \tilde \heat_tf$ in $L^2(\widetilde Y, \tilde \meas)$ in general. In particular after passing to a subsequence,  $\tilde \heat_tf_i(x)\to \tilde \heat_tf(x)$ as $i \to \infty$ for $\tilde \meas$-a.e. $x \in \widetilde Y$. Finally since 
\begin{align}
\left|\heat^2_tf_i(x)-\heat^2_tf(x )\right|&\le \int_{\widetilde{Y}}\Heat_2(x,y,t)\left| f_i(y)-f(y)\right|d\tilde \meas(y) \nonumber \\
&\le \left( \int_{\bar{Y}}\Heat_2(x,y,t)^2d\tilde \meas(y)\right)^{1/2}\cdot \|f_i-f\|_{L^2} \to 0,
\end{align}
we have $\heat^2_tf(x)=\tilde{\heat}_tf(x)$ for $\tilde \meas$-a.e. $x \in \widetilde Y$ because of $\heat^2_tf_i(x)=\tilde \heat_tf_i(x)$. Thus we have Lemma \ref{Duhamelpri}.


\begin{thebibliography}{10}

      \bibitem{Ambrosio}
	Luigi Ambrosio.
	\newblock Calculus, heat flow and curvature-dimension bounds in metric measure spaces.
	\newblock {\em Proceedings of the ICM 2018, Vol. 1, World Scientific, Singapore}, 301--340, 2019.	

\bibitem{AmbrosioGigliSavare13}
	Luigi Ambrosio, Nicola Gigli and Giuseppe Savar\'e.
	\newblock Calculus and heat flow in metric measure spaces and applications to spaces with Ricci bounds from below.
	\newblock {\em Invent. Math.}, 195:289--391, 2014.

\bibitem{AmbrosioGigliSavare14}
	Luigi Ambrosio, Nicola Gigli and Giuseppe Savar\'e.
	\newblock Metric measure spaces with Riemannian Ricci curvature bounded from below.
	\newblock {\em Duke Math. J.}, 163:14-5--1490, 2014.

%\bibitem{AmbrosioGigliMondinoRajala}
%       Luigi Ambrosio, Nicola Gigli, Andrea Mondino and Tapio Rajala.
%        \newblock Riemannian Ricci curvature lower bounds in metric measure spaces with $\sigma$-finite measure. 
%\newblock {\em Trans. of the AMS.}, 367:4661--4701, 2015.

 \bibitem{AmbrosioHonda}
	Luigi Ambrosio and Shouhei Honda.
	\newblock Local spectral convergence in $\RCD^*(K, N)$ spaces.
	\newblock {\em Nonlinear Anal.}, 177(Part A):1--23, 2018.	

      \bibitem{AmbrosioHondaTewodrose}
	Luigi Ambrosio, Shouhei Honda and David Tewodrose.
	\newblock Short-time behavior of the heat kernel and Weyl's law on $\RCD^*(K, N)$-spaces.
	\newblock {\em Ann. Glob. Anal. Geom.}, 53(1):97--119, 2018.	

%\bibitem{AmbrosioHondaTewodrosePortegies}
	%Luigi Ambrosio, Shouhei Honda, David Tewodrose and Jacobs. W. Portegies.
	%\newblock Embedding of $\RCD^*(K, N)$-spaces in $L^2$ via eigenfunctions.
	%\newblock {\em J. Func. Anal.}, 280(10): 108968, 2021.	

%\bibitem{AMS}
%	Luigi Ambrosio, Andrea Mondino and Giuseppe Savar\'e.
%	\newblock On the Bakry$–$\'Emery Condition, the Gradient Estimates and the Local-to-Global Property of $RCD^*(K,N)$ Metric Measure Spaces
%	\newblock {\em J. Geom. Anal.}, 26:24--56,  2016.	

\bibitem{AmbrosioMondinoSavare}
        Luigi Ambrosio, Andrea Mondino and Giuseppe Savar\'e.
\newblock Nonlinear diffusion equations and curvature conditions in metric measure spaces.
 \newblock {\em Mem. Amer. Math. Soc.}, 262, 2019, no. 1270, v+121 pp.

\bibitem{APS}
Michael F. Atiyah, Vijay K. Patodi and Isadore M. Singer.
\newblock Spectral asymmetry and Riemannian geometry. I.
\newblock {\em Math. Proc. Cambridge Philos. Soc.} 77 (1975), 43–69.

\bibitem{BH}	
Martin Thomas Barlow and Ben M. Hambly.
\newblock Transition density estimates for Brownian motion on scale irregular Sierpinski gaskets.
\newblock {\em Ann. Inst. Henri Poincare Probab. Stat}, 33(5):531-557, 1997.

\bibitem{Bel}	
Andre Bella\"{\i}che.
\newblock The tangent space in sub-{R}iemannian geometry, In {\em Sub-Riemannian geometry}.
\newblock {\em Progr. Math.}, Vol. 144, Birkhäuser, Basel, 1-78, 1996.

\bibitem{bb}
Anders Bj\"orn and Jana Bj\"orn.
\newblock Nonlinear Potential theory on metric measure spaces.
\newblock {\em European Mathematical Society}, 2011.

\bibitem{B} Manlio Bordoni
\newblock Comparing heat operators through local isometries of fibrations.
\newblock {\em Bull. Soc. math. France} 128:151-178, (2000), 32-51. 

\bibitem{Boscain-PS2016} Boscain, U. and Prandi, D. and Seri, M.
\newblock Spectral analysis and the {A}haronov-{B}ohm effect on certain
almost-{R}iemannian manifolds.
\newblock {\em Comm. Partial Differential Equations} 41 (2016), 32-51. 

\bibitem{BGHZ}
          Camillo Brena, Nicola Gigli, Shouhei Honda and Xingyu Zhu.
          \newblock  Weakly non-collapsed $\RCD$ spaces are strongly non-collapsed.
          \newblock  {\em J. Reine Angew. Math.} 794 (2023), 215–252.

%{\color{blue}
%\bibitem{BGT}
          %Nicholas Hugh Bingham, Charles M. Goldie and  Jozef L. Teugels.
          %\newblock  Regular variation.
          %\newblock  {\em Encyclopedia of Mathematics and its Applications 27,} Cambridge University Press. 1987.
%}


\bibitem{BrueSemola}
          Elia Bru\`e and Daniele Semola.
          \newblock  Constancy of dimension for $\RCD^*(K,N)$ spaces via regularity of Lagrangian flows.
          \newblock  {\em Comm. Pure and Appl. Math.}, 73:1141--1204, 2019.

\bibitem{CavallettiMilman}
         Fabio Cavalletti and Emanuel Milman. 
\newblock The Globalization Theorem for the Curvature Dimension Condition.
         \newblock {\em Invent. Math.}, 226: 1--137, 2021.

	
%	\bibitem{Ber86}
%	Lionel B\'{e}rard-Bergery.
%	\newblock Quelques exemples de vari\'{e}t\'{e}s riemanniennes compl\`etes non
%	compactes \`a courbure de {R}icci positive.
%	\newblock {\em C. R. Acad. Sci. Paris S\'{e}r. I Math.}, 302(4):159--161, 1986.

 \bibitem{Chavel} Isaac Chavel, 
\newblock{Eigenvalues in Riemannian geometry.} 
 Pure and Applied Mathematics, 115. Academic Press, Inc., Orlando, FL, 1984. xiv+362 pp.

\bibitem{Cheeger1}
Jeff Cheeger.
	\newblock Analytic torsion and the heat equation.
\newblock {\em Ann. of Math.} (2) 109 (1979), no. 2, 259–322.

\bibitem{Cheeger2}
Jeff Cheeger.
	\newblock Differentiability of Lipschitz functions on metric measure spaces.
\newblock {\em Geom. Funct. Anal.}  9 (1999), no. 3, 428–517.
	
	\bibitem{ChCo96}
	Jeff Cheeger and Tobias~H. Colding.
	\newblock Lower bounds on {R}icci curvature and the almost rigidity of warped
	products.
	\newblock {\em Ann. of Math. (2)}, 144(1):189--237, 1996.
	
	\bibitem{ChCo97}
	Jeff Cheeger and Tobias~H. Colding.
	\newblock On the structure of spaces with {R}icci curvature bounded below. i.
	\newblock {\em J. Differential Geom.}, 46(3):406--480, 1997.
	
	\bibitem{ChCo00a}
	Jeff Cheeger and Tobias~H. Colding.
	\newblock On the structure of spaces with {R}icci curvature bounded below. ii.
	\newblock {\em J. Differential Geom.}, 54(1):13--35, 2000.
	
	\bibitem{ChCo00b}
	Jeff Cheeger and Tobias~H. Colding.
	\newblock On the structure of spaces with {R}icci curvature bounded below. iii.
	\newblock {\em J. Differential Geom.}, 54(1):37--74, 2000.

\bibitem{CPR}
	Yacine Chitour, Dario Prandi and Luca Rizzi.
	\newblock Weyl's law for singular Riemannian manifolds.
	\newblock {\em arXiv.2208.13962}.
	
	\bibitem{CoNa12}
	Tobias~H. Colding and Aaron Naber.
	\newblock {Sharp Hölder continuity of tangent cones for spaces with a lower
		Ricci curvature bound and applications}.
	\newblock {\em Annals of Mathematics}, 176(2):1173--1229, 2012.

	\bibitem{DaiYan20}
	Xianzhe Dai and Junrong Yan.
	\newblock {Witten Deformation on Non-compact Manifolds--Heat Kernel Expansion and Local Index Theorem.}
	\newblock {\em arXiv:2011.05468}, 2020. To appear in {\em Math. Z.}

\bibitem{DaiYan22}
	Xianzhe Dai and Junrong Yan.
	\newblock {Non-semiclassical Weyl Law for Schr\"odinger Operators 
	on Non-compact Manifolds.}
	\newblock {\em preprint}, 2022.

\bibitem{Davies}
	Edward Brian Davies.
         \newblock Non-Gaussian aspects of heat kernel behaviour.  
         \newblock {\em J. London Math. Soc.},  55(2):105--125, 1997.

\bibitem{Qin20}
Qin Deng. 
\newblock H\"older continuity of tangent cones in $\RCD(K,N)$ spaces and applications to non-branching.
\newblock {\em arXiv:2009.07956}, 2020
	
%	\bibitem{Deng20}
%	Qin {Deng}.
%	\newblock {H\"older continuity of tangent cones in $\RCD(K,N)$ spaces and
%		applications to non-branching}.
%	\newblock {\em arXiv:2009.07956}, 2020.

      \bibitem{DePhillippisGigli}
         Guido De Philippis and Nicola Gigli.
         \newblock Non-collapsed spaces with Ricci curvature bounded below.
         \newblock {\em Journal de l'\'Ecole polytechnique}, 5, 613--650, 2018.

\bibitem{dVHT}
         Yves Colin de Verdi\`ere, Luc Hillairet and Emmanuel Tr\'elat.
         \newblock Spectral asymptotics for sub-Riemannian Laplacians.
         \newblock {\em arXiv:2212.02920},  2022.

      \bibitem{Ding02}
	Yu {Ding}.
	\newblock {Heat kernels and Green's functions on limit spaces}.
	\newblock {\em Comm. Anal. Geom.}, 10(3):475--514, 2002.


\bibitem{Feller}
          William Feller.
          \newblock  An introduction to probability theory and its applications. Vol II, Second edition.
          \newblock  {\em  Wiley, New Work}, 1971.

	
%	\bibitem{FY92}
%	Kenji Fukaya and Takao Yamaguchi.
%	\newblock The fundamental groups of almost nonnegatively curved manifolds.
%	\newblock {\em Annals of Mathematics}, 136(2):253--333, 1992.

\bibitem{ErbarKuwadaSturm}
         Matthias Erbar, Kazumasa Kuwada and Karl-Theodor Sturm.
        \newblock On the equivalence of the entropic curvature-dimension condition and Bochner's inequality on metric measure spaces.
        \newblock {Invent. Math.}, 201:993--1071, 2015. 

%\bibitem{GKMS}
%	Fernando Galaz-Garc\'ia, Martin Kell, Andrea Mondino and Gerardo Sosa.
%	\newblock On quotients of spaces with Ricci curvature
%bounded below.
%	\newblock {\em J. Funct. Anal.}, 275:1368--1446, 2018.

%\bibitem{G}
%	Nicola Gigli.
%	\newblock The splitting theorem in non-smooth context.
%	\newblock {\em 	arXiv:1302.5555.}


\bibitem{GM}
        Nicola Garofalo, Andrea Mondino.
\newblock Li-Yau and Harnack type inequalities in $\RCD^*(K, N)$ metric measure spaces.
 \newblock {\em Nonlinear Anal.}, 96:721--734 , 2014


\bibitem{G2}
        Nicola Gigli.
\newblock On the differential structure of metric measure spaces and applications.
 \newblock {\em Mem. Amer. Math. Soc.}, 236, 2015, no. 1113.

\bibitem{GP}
	Nicola Gigli and Enrico Pasqualetto.
	\newblock Lectures on Nonsmooth Differential Geometry.
	\newblock {\em SISSA Springer Series, Vol. 2, Springer International Publishing},  2020, pp. xi+204

\bibitem{GRS}
	Nicola Gigli, Tapio Rajala and Karl-Theodor Sturm.
	\newblock Optimal maps and exponentiation on finite dimensional spaces with Ricci curvature bounded from below.
	\newblock {\em J. Geom. Anal.}, 26:2914–2929, 2016

\bibitem{Gr}
        Alexander Grigor\'yan.
\newblock Heat kernel and analysis on manifolds.
 \newblock {\em AMS/IP Studies in Advanced Mathematics, 47. American Mathematical Society, Providence, RI, International
Press, Boston, MA.}, 2009.

\bibitem{HM}
	Bang-Xian Han and Andrea Mondino.
	\newblock Angles between curves in metric measure spaces.
	\newblock {\em Anal. Geom. Metr. Spaces}, 5:47--68, (2017).


\bibitem{H}
	Shouhei Honda.
	\newblock A weakly second-order differential structure on rectifiable metric measure spaces.
	\newblock {\em Geom. Topol}. 18(2) 633-668, (2014).


 \bibitem{Hon}
      Shouhei Honda.
       \newblock New differential operators and non-collapsed $\RCD$ spaces.
        \newblock {\em Geom. Topol.}, 24:2127--2148, 2020. 

 \bibitem{IV}
      Victor Ivrii.
       \newblock 100 years of Weyl's law.
        \newblock {\em Bull. Math. Sci.}, 6:379--452, 2016. 


	\bibitem{JZ}
	Dmitry Jakobson and Steve Zelditch. 
	\newblock Classical Limits of Eigenfunctions for Some Completely Integrable Systems.
	\newblock {\em Emerging applications of number theory (Minneapolis, MN, 1996)}, 329-354, IMA Vol. Math. Appl., 109, Springer, New York, 1999.

\bibitem{Jiang2015}
	Renjin Jiang,
	\newblock The Li-Yau inequality and heat kernels on metric measure spaces.
\newblock{\em 	J. Math. Pures Appl.} (9) 104 no. 1, 29–57, (2015). 

	\bibitem{JiangLiZhang}
	Renjin Jiang,  Huaiqian Li and Huichun Zhang.
	\newblock Heat Kernel Bounds on Metric Measure Spaces and Some Applications.
	\newblock {\em Potent. Anal.}, 44:601--627, 2016.	
	
\bibitem{Juillet2021}
	Nicolas Juillet.
	\newblock Sub-Riemannian structures do not satisfy Riemannian
		Brunn-Minkowski inequalities,
	\newblock {\em Rev. Mat. Iberoam.} 37 no. 1, 177-188,  (2021). 


%      \bibitem{HK}
%     Piotr Haj\l asz and Pekka Koskela.
%      \newblock Sobolev met Poincar\'e. 
%     \newblock Mem. Amer. Math. Soc., no. 688, 145,  2000.

\bibitem{KKK}
      Vitali Kapovitch, Martin Kell and Christian Ketterer.
       \newblock On the structure of  $\RCD$ spaces with upper curvature bounds.
        \newblock {\em Math. Z.}, 301:3469--3502, 2022. 

\bibitem{KM}
      Vitali Kapovitch and Andrea Mondino.
       \newblock On the topology and the boundary of $N$-dimensional $\RCD(K,N)$ spaces.
        \newblock {\em Geom. Topol.}, 25(1):445--495, 2021. 

\bibitem{KL}
      Yu Kitabeppu and Sajjad Lakzian.
       \newblock Characterization of low dimensional $\RCD^*(K,N)$ spaces.
        \newblock {\em Analysis and Geometry in Metric Spaces}, 4:187--215, 2016. 

      

\bibitem{L}
      Peter Li.
      \newblock Geometric Analysis. 
      \newblock {\em Cambridge University Press.}, 134, 2012.

\bibitem{Lian}
      Maria Lianantonakis.
      \newblock On the eigenvalue counting function for weighted Laplace-Beltrami operators. 
      \newblock {\em J. Geom. Anal.}, 10(2):299-322, 2000.

\bibitem{John}
	John Lott.
	\newblock Some geometric properties of the Bakry-\'Emery-Ricci tensor.
	\newblock {\em Comm. Math. Helv.} 78:865--883, 2003.	


      \bibitem{LottVillani}
      John Lott and Cedric Villani.
      \newblock Ricci curvature for metric-measure spaces via optimal transport. 
      \newblock {\em Ann. of Math.} 169:903--991, 2009.




\bibitem{LS}
      Alexander Lytchak and Stephan Stadler.
      \newblock Ricci curvature in dimension $2$. 
      \newblock {\em J. Eur. Math. Soc.}, DOI:  10.4171/JEMS/1196, 2022.

\bibitem{Milnor} John Milnor.
\newblock A note on curvature and fundamental group.
\newblock {\em J. Differential Geometry}, 2:1-7, 1968. 

 \bibitem{Menguy}
	Xavier Menguy.
	\newblock Examples of nonpolar limit spaces.
	\newblock {\em Amer. J. Math.}, 122(5):927--937, 2000.

\bibitem{MN}
      Andrea Mondino and Aaron Naber.
       \newblock Structure Theory of Metric-Measure Spaces with Lower Ricci Curvature Bounds.
        \newblock {\em J. Eur. Math Soc.}, 21:1809--1854,  2019.

\bibitem{MS}
	Andrea Mondino and Daniele Semola.
	\newblock Weak Laplacian bounds and minimal boundaries in non-smooth spaces with Ricci curvature lower bounds.
	\newblock {\em arXiv:2107.12344v3}, to appear in Mem. Amer. Math. Soc.

\bibitem{MW}
	Andrea Mondino and Guofang Wei.
	\newblock On the universal cover and the fundamental group of an $\RCD^*(K, N)$-space,.
	\newblock {\em Journal f\"ur die reine und angewandte Mathematik}, 753:211--237, (2019).
%	\bibitem{Na20}
%	Aaron Naber.
%	\newblock Conjectures and open questions on the structure and regularity of
%	spaces with lower {R}icci curvature bounds.
%	\newblock {\em SIGMA Symmetry Integrability Geom. Methods Appl.}, 16:Paper No.
%	104, 8, 2020.
	
%	\bibitem{Nab80}
%	Philippe Nabonnand.
%	\newblock Sur les vari\'{e}t\'{e}s riemanniennes compl\`etes \`a courbure de
%	{R}icci positive.
%	\newblock {\em C. R. Acad. Sci. Paris S\'{e}r. A-B}, 291(10):A591--A593, 1980.
%	
%	\bibitem{Pan20}
%	Jiayin {Pan}.
%	\newblock {Nonnegative {R}icci curvature and escape rate gap}.
%	\newblock {\em arXiv:2009.00226}, 2020.
%	
%	\bibitem{Pan21}
%	Jiayin Pan.
%	\newblock On the escape rate of geodesic loops in an open manifold with
%	nonnegative {R}icci curvature.
%	\newblock {\em Geom. \& Topol.}, 25(2):1059--1085, 2021.
\bibitem{Pan22}
	Jiayin Pan.
	\newblock 
	The Grushin hemisphere as a Ricci limit space with curvature $\ge 1$.
	\newblock {\em 	arXiv:2211.02747}, 2022, to appear in Proc. Amer. Math. Soc.
	
       \bibitem{PanWei}
	Jiayin Pan and Guofang Wei.
      \newblock Examples of Ricci limit spaces with non-integer Hausdorff dimension.
      \newblock {\em Geom. Funct. Anal.}, 32, 676--685, 2022.

     \bibitem{PanWei22}
     Jiayin Pan and Guofang Wei.
     \newblock Examples of open manifolds with positive Ricci curvature and non-proper Busemann functions.
     \newblock {\em arXiv: 2203.15211}, 2022



       \bibitem{Rajala}
        Tapio Rajala.
       \newblock Local Poincar\'e inequalities from stable curvature conditions on metric spaces.
       \newblock  {\em Calc. Var. Partial Differential Equations}, 44(3), 477--494, 2012.

%\bibitem{Savare}
       %Giuseppe Savar\'e.
       %\newblock Self-improvement of the Bakry-\'Emery condition and
  %Wasserstein contraction of the heat flow in $\RCD(K,\infty)$ metric
  %measure spaces, 
       %\newblock {\em Discrete Contin. Dyn. Syst.}, 34:1641--1661, 2014.

%\bibitem{Simon}
%	Barry Simon.
%	\newblock Functional Integration and Quantum Physics.
%	 \newblock {\em Academic Press}, 1979.

\bibitem{Sturm95}
	Karl-Theodor Sturm.
	\newblock Analysis on local Dirichlet spaces. II. Upper Gaussian estimates for the fundamental
solutions of parabolic equations.
	 \newblock {\em Osaka J. Math.}, 32(2):275--312, 1995.

\bibitem{Sturm96}
	Karl-Theodor Sturm.
	\newblock Analysis on local Dirichlet spaces. III. The parabolic Harnack inequality.
	\newblock {\em J. Math. Pures Appl.}, 75(9):273--297, 1996.

%      \bibitem{Sturm06a}
%	Karl-Theodor Sturm.
%	\newblock On the geometry of metric measure spaces, I.
%	\newblock {\em Acta Math.}, 196:65--131, 2006.

\bibitem{Sturm06b}
	Karl-Theodor Sturm.
	\newblock On the geometry of metric measure spaces, II.
	\newblock {\em Acta Math.}, 196:133--177,  2006.
		
	
%	\bibitem{Wei88}
%	Guofang Wei.
%	\newblock Examples of complete manifolds of positive {R}icci curvature with
%	nilpotent isometry groups.
%	\newblock {\em Bull. Amer. Math. Soc. (N.S.)}, 19(1):311--313, 1988.

\bibitem{Ye23}
	Zhu Ye.
	\newblock Maximal first Betti number rigidity of noncompact $\RCD(0,N)$ spaces
	\newblock {\em arXiv:2301.04600}, 2023.


\bibitem{Z}
	Steve Zelditch. 
	\newblock Eigenfunctions of the Laplacian on a Riemannian manifold.
	\newblock {\em volume 125 of CBMS
Regional Conference Series in Mathematics}, Published for the Conference Board of the Mathematical
Sciences, Washington, DC; by the American Mathematical Society, Providence, RI, 2017

\bibitem{ZZ}
	Huichun Zhang and Xiping Zhu.
	\newblock Weyl's law on $\RCD^*(K,N)$ metric measure spaces.
	\newblock {\em Comm. Anal. Geom.}, 27(8):1869--1914, 2019.	
	
\end{thebibliography}
\end{document}